\documentclass[11pt,a4paper]{article}

\usepackage{float}
\usepackage{amssymb}
\usepackage{amsmath}
\usepackage{bbm}
\usepackage{amsthm}
\usepackage{color}
\usepackage{mathrsfs}
\usepackage{eucal}
\usepackage{tikz}
\usepackage{amsfonts}
\usepackage[T1]{fontenc}
\usepackage{inputenc}
\usepackage[english]{babel}
\usepackage{lmodern}
\usepackage{hyperref}
\usepackage{mathtools}
\usepackage{geometry}
\usepackage{changepage}
\changepage{0pt}{}{}{}{}{0pt}{}{0pt}{6pt}
\usepackage[numbers]{natbib}
\usepackage{hyperref}
\hypersetup{
pdfpagemode=none,
pdftoolbar=true,        
pdfmenubar=true,        
pdffitwindow=false,     
pdfstartview={Fit},    
pdftitle={Trapped modes in electromagnetic waveguides},    
pdfauthor={A.-S. Bonnet-Ben Dhia, L. Chesnel, S. Fliss},     
pdfsubject={},  
pdfcreator={A.-S. Bonnet-Ben Dhia, L. Chesnel, S. Fliss},   
pdfproducer={}, 
pdfkeywords={}, 
pdfnewwindow=true,      
colorlinks=true,       
linkcolor=magenta,          
citecolor=red,        
filecolor=cyan,      
urlcolor=blue           
}

\newcommand{\dsp}{\displaystyle}
\newcommand{\eps}{\varepsilon}
\newcommand{\om}{\omega}
\newcommand{\Om}{\Omega}
\newcommand{\mrm}[1]{\mathrm{#1}}
\newcommand{\bfx}{\boldsymbol{x}}
\newcommand{\bfnu}{\boldsymbol{\nu}}
\newcommand{\Cplx}{\mathbb{C}}
\newcommand{\N}{\mathbb{N}}
\newcommand{\R}{\mathbb{R}}
\newcommand{\mL}{\mrm{L}}
\newcommand{\mH}{\mrm{H}}
\newcommand{\mX}{\mrm{X}}

\newcommand{\curlvec}{\boldsymbol{\mrm{curl}}\,}
\renewcommand{\div}{\mrm{div}}
\newcommand{\GX}{X}
\newcommand{\GL}{L}
\def\E{\boldsymbol{E}}

\def\H{\boldsymbol{H}}

\newtheorem{theorem}{Theorem}[section]
\newtheorem{lemma}[theorem]{Lemma}
\newtheorem{remark}[theorem]{Remark}

\newtheorem{proposition}[theorem]{Proposition}

\definecolor{MonVert}{RGB}{0,243,126}

\begin{document}

~\vspace{0.0cm}
\begin{center}
{\sc \bf\huge\selectfont Trapped modes in electromagnetic waveguides}\\[16pt]
\end{center}
\begin{center}
\textsc{Anne-Sophie Bonnet-Ben Dhia}$^1$, \textsc{Lucas Chesnel}$^2$, \textsc{Sonia Fliss}$^1$\\[16pt]
\begin{minipage}{0.97\textwidth}
		{\small
$^1$ POEMS, CNRS, Inria, ENSTA, Institut Polytechnique de Paris, Palaiseau, France;\\
$^2$ Inria, UMA, ENSTA, Institut Polytechnique de Paris, Palaiseau, France.\\[10pt]		
			E-mails: \texttt{anne-sophie.bonnet-bendhia@ensta.fr}, \texttt{lucas.chesnel@inria.fr},  \texttt{sonia.fliss@ensta.fr}.\\[-14pt]
			\begin{center}
				(\today)
			\end{center}
		}
	\end{minipage}
\end{center}
\textbf{Abstract.} We consider Maxwell's equations with perfect electric conductor boundary conditions in three-dimensional unbounded domains which are the union of a bounded resonator and one or several semi-infinite waveguides. We are interested in the existence of electromagnetic trapped modes, \textit{i.e.} $\mL^2$ solutions of the problem without source term. These trapped modes are associated to eigenvalues of Maxwell's operator, that can be either below the essential spectrum or embedded in it. First for homogeneous waveguides, we present different families of geometries for which we can prove the existence of eigenvalues. Then we exhibit certain non-homogeneous waveguides with local perturbations of the dielectric constants that support trapped modes. Let us mention that some of the mechanisms we propose are very specific to Maxwell's equations and have no equivalent for the scalar Dirichlet or Neumann Laplacians.\\[5pt] 

\section{Introduction}

Trapping waves in a bounded region of an unbounded domain is in general a difficult task because waves can escape by radiating to infinity. For instance, as a consequence of Rellich's theorem, see \textit{e.g.} \cite[Thm. 2.14]{CoKr13}, there is no non-trivial $\mL^2$ solution of the Helmholtz equation in an exterior domain in $\R^d$, $d=2,3$. This means that a bounded obstacle embedded in a homogeneous material filling $\R^d$, whether it is penetrable or not, cannot trap waves. This is true in acoustics but in electromagnetism and elasticity as well. However, it has been known for a long time that they can exist in unbounded waveguides (see the historical articles \cite{Urse51,Jone53} concerning water-waves problems). The reason is that, roughly speaking, it is less easy for waves to escape in these geometries. Indeed, at a given frequency, only a finite number (possibly none) of modes can propagate  while the others are evanescent.\\
\newline 
The study of trapped waves in waveguides can be reformulated as  a problem of spectral theory for self-adjoint operators. The latter have a so-called essential spectrum corresponding to the range of frequencies for which propagating modes exist. Trapped waves, more commonly referred to as trapped modes, are  eigenfunctions of these operators associated with eigenvalues which can be in the discrete spectrum or embedded in the essential spectrum. The cases of the Laplacian operators in two and three dimensions with either Dirichlet (quantum waveguides) or Neumann (acoustic waveguides) boundary conditions have been considered first. In quantum waveguides, the lower bound of the essential spectrum, is positive. Therefore it is rather simple to trap modes at low frequencies and existence of discrete eigenvalues below the essential spectrum has been established in a large variety of situations (waveguides with a so-called resonator large enough \cite{na457}, bent waveguides \cite{DuEx95}, L-shaped or X-shaped waveguides \cite{ExSt89}...). Most of the proofs rely on the min-max principle. They consist in finding test functions for which the Rayleigh quotient is less than the lower bound of the essential spectrum. By contrast, in acoustic waveguides, the essential spectrum fills the half-line $[0;+\infty)$. As a consequence, the study of eigenvalues is more subtle than in quantum waveguides because they are necessarily embedded eigenvalues. Their existence has been proved first in geometries with symmetries, so that embedded eigenvalues correspond to discrete spectrum for a restriction of the operator to some subspace \cite{Wits90,Evan92,EvLV94,DaPa98,LiMc07}. Trapped modes associated with embedded eigenvalues are also often called bound states in the continuum (BSCs or BICs) in quantum mechanics or in optics (see for example \cite{SaBR06,Mois09,GPRO10,zhen2014topological,gomis2017anisotropy} as well as the review \cite{HZSJS16}). Let us emphasize that contrary to eigenvalues below the essential spectrum in quantum waveguides, they are unstable objects with respect to perturbations of the geometry. Of course embedded eigenvalues in symmetric domains remain embedded for symmetric perturbations. However, as shown in \cite{AsPV00}, in general embedded eigenvalues become complex resonances under small changes of the geometry. Let us mention though that by tuning carefully the shape of the perturbation, one can force the eigenvalues to stay embedded in the essential spectrum (see \cite{Naza13}), and for example, one can construct non-symmetric waveguides supporting embedded eigenvalues. For another technique of proof of existence of embedded eigenvalues, based on the use of the augmented scattering matrix introduced in \cite{KaNa02}, we refer the reader to \cite{ChPa18}.\\ 
\newline
Thus the literature on trapped modes for scalar problems is relatively rich. Extending the results to the vector problems that arise in elasticity and electromagnetism is not easy. There have been several works concerning elastic waveguides, see \textit{e.g.} \cite{Naza08Elast,Pagn13,Naza20Elast}. On the other hand, there seem to be very few studies for Maxwell's problem except in geometries with complete separation of variables where the analysis boils down to that of scalar operators (see \cite{ExSe90,GoJa92,CLMM93,AYCM06} and \S\ref{ParaSca} below). In more general configurations, the recent article \cite{BCOZ25} where the authors investigate the influence of bending and twisting on the appearance of eigenvalues can be considered as a pioneer in the mathematical community.\\ 
\newline	
The fact that dealing with vector problems is more difficult technically  does not mean that it is more difficult to trap waves for these models. On the contrary, it has been noticed for elastic waveguides that the existence of waves with different polarizations allows for trapping in simpler geometries than for the classical scalar problems. For instance, it is proved in \cite{Pagn13} that a simple semi-infinite elastic cylinder with traction-free boundaries supports trapped waves localized near its edge. Similarly, below we will exploit properties specific to Maxwell's equations, like the fact that Transverse Electro-Magnetic (TEM) modes propagate at all frequencies when the cross-section is not simply connected (see \S\ref{paraTEM}), to exhibit simple electromagnetic waveguides supporting trapped modes.\\
\newline
In addition to their own mathematical interest, there are several motivations to prove existence of trapped modes. First, they are important in the mathematical analysis of scattering problems. At a frequency for which trapped modes exist, the scattering problem is ill-posed. Indeed, trapped modes are outgoing solutions of the problem without incident field. Besides, trapped modes play a crucial role in the study of the transient regime. Indeed, in a geometry with trapped modes, one does not have the classical local energy decay property. Trapped modes can also be useful to exhibit exotic phenomena for the scattering coefficients. More precisely, as said above, if trapped modes are associated to an eigenvalue embedded in the essential spectrum, a slight perturbation of the geometry or material coefficients induces in general a displacement of the eigenvalue from the real axis to the complex plane: trapped modes become so-called quasi-normal modes associated with complex resonances. Then one observes rapid variations of the scattering coefficients for (real) frequencies varying in a neighborhood of such complex resonances. This is the Fano resonance phenomenon which can be exploited to exhibit situations of zero reflection or complete reflection (see \cite{ShTu12,ChNa18,ChNaFano}).\\
\newline 
The goal of this work is to present various three-dimensional  electromagnetic waveguides supporting trapped modes. Concerning the study of the essential spectrum, we rely on the articles \cite{Filo19,FeMa24}.  For a homogeneous waveguide with a simply connected cross-section, the lower bound of the essential spectrum is positive and corresponds to the cut-off frequency of the fundamental Transverse Electric (TE) mode.  If on the contrary the cross-section of the waveguide is not simply connected, then the essential spectrum fills the half-line $[0;+\infty)$ due to  TEM modes. To prove existence of trapped modes, we shall use two different approaches. 
\begin{enumerate}
	\item In geometries allowing for complete separation of variables, we construct exact trapped modes from eigenfunctions of the 2D Dirichlet/Neumann Laplacians. Depending on the case, the eigenvalues we obtain for Maxwell's problem are either below or embedded in the essential spectrum. 
	\item When complete separation of variables cannot be used and when the essential spectrum has a positive lower bound, we exploit the min-max principle to establish existence of discrete spectrum. The design of appropriate test fields is delicate, due in particular to the divergence free condition. 
To find relevant test fields, we follow different strategies.
\begin{itemize}
\item If the perturbed part of the infinite waveguide is itself a portion of a waveguide where separation of variables can be used, we can choose as 
test fields particular eigenmodes of this bounded domain, extended by zero outside. Such test fields are compactly supported.
\item We prove other existence results involving non-compactly supported test fields. The latter are obtained by working with trapped modes of simpler problems, either for the scalar Dirichlet  Laplacian or for Maxwell's operator in subdomains where separation of variables can be used. In both cases, a gradient of a function defined on the entire domain must be added so that test fields be divergence free.
\item Finally, we work with test fields with non-compact support which are obtained by multiplying the TE mode associated to the lower bound of the essential spectrum by a slowly decaying
function. We use this technique to address the case of local variations of the material coefficients in straight waveguides. 
\end{itemize}
\end{enumerate}

\noindent The outline of the article is as follows. The setting and notation are introduced in Section \ref{SectionSetting}. 
In Section \ref{SectionFullySeparable}, we work in geometries where one has complete separation of variables to give first simple examples of waveguides supporting trapped modes. In Section \ref{Section1}, we exploit separation of variables only in a bounded part of the domain (the  resonator) to build \textit{ad hoc} test fields for the min-max principle.  
In Section \ref{Section6Legs}, we show the existence of trapped modes in situations where separation of variables cannot be used. We start by proving a result of comparison of the first eigenvalue of Maxwell's operator with the first eigenvalue of the Dirichlet Laplacian which allows us to guarantee that for domains with large enough resonators, for sure there are eigenvalues below the essential spectrum. Then we study two canonical geometries, respectively with three and six infinite branches, for which from eigenfunctions of the 2D Dirichlet Laplacians in L-shaped and X-shaped domains, we build well-suited test fields for the min-max principle. This part requires very sharp estimates to prove that the Rayleigh quotient is less than the lower bound of the essential spectrum. Section \ref{SectionVariableCoef} is dedicated to the study of locally heterogeneous waveguides. In Section \ref{SectionSym}, by playing with symmetries and adapting a classical trick used for scalar operators, we explain how to create geometries supporting embedded eigenvalues. We give concluding remarks in Section \ref{SectionConclusion}. Finally, in the Appendix, we discuss the case of non-simply connected waveguides with non-connected boundaries and give the proof of one technical result needed in Section \ref{Section6Legs}.

\section{The electric operator}\label{SectionSetting}

\begin{figure}[!ht]
\centering
\includegraphics[width=8cm,trim={4.6cm 0cm 4.4cm 0cm},clip]{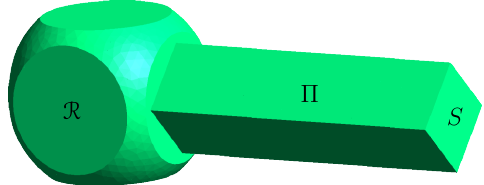}
\caption{Example of waveguide $\Om$ (the guide $\Pi$ extends to infinity).\label{Marteau0}}
\end{figure}

Denote by $\bfx=(x,y,z)$ the coordinates in $\R^3$. Let $\Om$ be a simply connected open subset of $\R^3$ whose boundary $\partial\Om$ is connected\footnote{See a brief discussion in the Appendix, \S\ref{ParaNonTrivTop}, when these assumptions are not met.} and Lipschitz continuous. To set ideas, we start by describing a simple geometry. Assume here that for $z>0$, $\Om$ coincides with the cylinder 
\begin{equation}\label{DefOm1}
\Pi\coloneqq S\times(0;+\infty),
\end{equation}
where the section $S$ is a bounded domain of $\R^2$ with Lipschitz boundary $\partial S$. Moreover, we assume that the remaining part of the guide
\begin{equation}\label{DefOm2}
\mathcal{R}\coloneqq\Om\setminus\overline{\Pi}=\{\bfx\in\Om\,|\,z<0\},
\end{equation}
that we call the resonator, is bounded (see Figure \ref{Marteau0}\footnote{This 3D representation as well as the similar ones below have been obtained with the mesh generator \href{https://gmsh.info/}{Gmsh} \cite{GeRe09}.}). We look at the spectrum of the electric problem which writes, when the dielectric permittivity $\eps$ and the magnetic permeability $\mu$ are such that $\eps\equiv1$, $\mu\equiv1$ in $\Om$, 
\begin{equation}\label{DefPb0}
\begin{array}{|rcll}
\curlvec\curlvec\E&=&\lambda\E&\mbox{ in }\Om\\[2pt]
\div\,\E&=&0&\mbox{ in }\Om\\[2pt]
\E\times\bfnu&=&0&\mbox{ on }\partial\Om.
\end{array}
\end{equation}
Here $\bfnu$ stands for the unit outward normal vector to $\partial\Om$. Let us construct an unbounded operator associated with (\ref{DefPb0}). To proceed, introduce the spaces 
\begin{equation}\label{DefSobolev}
\begin{array}{rcl}
\boldsymbol{\mL}^2(\Om)&\coloneqq&(\mL^2(\Om))^3\\[3pt]
\boldsymbol{\mH}(\div;0)&\coloneqq&\{\E\in\boldsymbol{\mL}^2(\Om)\,|\,\div\,\E=0\mbox{ in }\Om\}\\[3pt]
\boldsymbol{\mH}_N(\curlvec)&\coloneqq&\{\E\in\boldsymbol{\mL}^2(\Om)\,|\,\curlvec\E\in\boldsymbol{\mL}^2(\Om)\mbox{ and }\E\times\bfnu=0\mbox{ on }\partial\Om\}\\[3pt]
\boldsymbol{\mX}_N(\Om)&\coloneqq&\boldsymbol{\mH}_N(\curlvec)\cap\boldsymbol{\mH}(\div;0).
\end{array}
\end{equation}
Next, we define the unbounded operator $A$ in the Hilbert space $\boldsymbol{\mH}(\div;0)$ such that 
\begin{equation}\label{DefOpA}
\begin{array}{|rcl}
D(A)&=&\{\E\in\boldsymbol{\mX}_N(\Om)\,|\,\curlvec\curlvec\E\in\boldsymbol{\mL}^2(\Om)\}\\[4pt]
A\E&=&\curlvec\curlvec\E,
\end{array}
\end{equation}
where $D(A)$ is the domain of $A$. Note that indeed $A\E$ belongs to $\boldsymbol{\mH}(\div;0)$ for $\E\in D(A)$. One shows that $A$ is self-adjoint (see \textit{e.g.} \cite{BiSo87b}). Since it is also non-negative, its spectrum, that we denote by $\sigma(A)$, satisfies $\sigma(A)\subset[0;+\infty)$. Let $\sigma_{\mrm{ess}}(A)$ stand for its essential spectrum. Classically, essential spectrum for $A$ is due to the existence of propagating modes, \textit{i.e.} solutions of the form $\E(\bfx)=\mathscr{E}(x,y)e^{i\beta z}$, with $\beta\in\R$, solving (\ref{DefPb0}) in the straight part $\Pi$ of the guide. When computing these solutions with separate variables, one obtains different families of fields, the so-called Transverse Electric (TE), Transverse Magnetic (TM) and Transverse ElectroMagnetic (TEM) modes, the TEM modes existing if and only if $S$ is not simply connected (see an example of such geometry in Figure \ref{EmbeddedGeom} left).\\ 
\newline
More precisely, for the first TE mode, one finds
\begin{equation}\label{DefTEmodes}
\E^{\mrm{TE}}_\pm(\bfx)=\left(\begin{array}{c}\curlvec_{\mrm{2D}} \varphi_N(x,y)\\0\end{array}\right)e^{\pm i\beta_{N} z},\qquad\mbox{ with }\beta_{N}\coloneqq\sqrt{\lambda-\lambda_N}
\end{equation}
and
\[
\curlvec_{\mrm{2D}}\,\varphi_N=(\partial_y\varphi_N,-\partial_x\varphi_N)^{\top}.
\]
Here $\lambda_N$ denotes the first positive eigenvalue of the Neumann Laplacian in $S$ and $\varphi_N$ is a corresponding eigenfunction. This TE mode is propagating for $\lambda>\lambda_N$. Note that the zero eigenvalue of the Neumann Laplacian in $S$ plays no role because the associated eigenfunction is constant and therefore in this case the field (\ref{DefTEmodes}) is null.\\
\newline
For the first TM mode, one gets
\begin{equation}\label{DefTMmodes}
\E^{\mrm{TM}}_\pm(\bfx)=\left(\begin{array}{c}\nabla \varphi_D(x,y)\\[2pt]\mp i\beta_D^{-1}\lambda_D\varphi_D(x,y)\end{array}\right)e^{\pm i\beta_D z},\qquad\mbox{ with }\beta_D\coloneqq\sqrt{\lambda-\lambda_D}.
\end{equation}
Here $\lambda_D$ stands for the first eigenvalue of the Dirichlet Laplacian in $S$ and $\varphi_D$ is a corresponding eigenfunction. This TM mode is propagating for $\lambda>\lambda_D$.\\
\newline
We will give the definition of TEM modes later in (\ref{DefSectionsNotSC}) because it requires to introduce additional notation (for example, they depend on the number of holes in $S$). However what is important for the features of $\sigma_{\mrm{ess}}(A)$ is that they are propagating for all $\lambda>0$.\\  
\newline
N. Filonov establishes in \cite{Filo05} that one has always $\lambda_N<\lambda_D$ (see also the proof of L. Friedlander in \cite{Frie91} for $\mathscr{C}^1$ domains). In other words, the first positive eigenvalue of the Neumann Laplacian is less than the first eigenvalue of the Dirichlet Laplacian, actually in any arbitrary domain of finite measure in $\R^d$, $d > 1$. From this, we infer the following statement whose proof can be obtained from the results of \cite{Filo19,FeMa24}:
\begin{proposition}\label{DefSigmaEss}~\\[2pt]
If $S$ is simply connected, then $\sigma_{\mrm{ess}}(A)=[\lambda_N;+\infty)$.\\[2pt]
If $S$ is not simply connected, then $\sigma_{\mrm{ess}}(A)=[0;+\infty)$.
\end{proposition}
\noindent Let us give a concrete example for a geometry that we will encounter later on. If $S$ is the rectangle $(0;a)\times(0;b)$ with $a\ge b>0$, the first eigenvalues of the Neumann and Dirichlet Laplacians are respectively 
\[
0,\pi^2/a^2,\min(\pi^2/b^2,4\pi^2/a^2), \dots\qquad\mbox{ and }\qquad \pi^2/a^2+\pi^2/b^2,\dots.
\] 
Indeed one has $\pi^2/a^2<\pi^2/a^2+\pi^2/b^2$ and there holds $\sigma_{\mrm{ess}}(A)=[\pi^2/a^2;+\infty)$. Note that in the recent work \cite{Rohl25}, J. Rohleder shows that actually the first two positive eigenvalues of the Neumann Laplacian are less than the first eigenvalue of the Dirichlet Laplacian in bounded simply connected Lipschitz domains of $\R^2$ (observe that for the above rectangle, one has indeed $\min(\pi^2/b^2,4\pi^2/a^2)<\pi^2/a^2+\pi^2/b^2$).\\
\newline
This ends the description of $\sigma_{\mrm{ess}}(A)$. We denote by $\sigma_{\mrm{p}}(A)$ the point spectrum of $A$, \textit{i.e.} the set of eigenvalues of $A$. Moreover, we decompose this point spectrum 
into two parts, namely
\[
\sigma_{\mrm{p}}(A)=\sigma_{\mrm{d}}(A)\cup\sigma_{\mrm{emb}}(A),
\]
where the discrete spectrum and the set of embedded eigenvalues are respectively defined by
\[
\begin{array}{|rclcl}
\sigma_{\mrm{emb}}(A) & \coloneqq &  \sigma_{\mrm{p}}(A)\cap \sigma_{\mrm{ess}}(A) & =& \{\lambda\in\sigma_{\mrm{p}}(A)\,|\,\lambda\ge\lambda_N\} \\[3pt]
\sigma_{\mrm{d}}(A) & \coloneqq & \sigma_{\mrm{p}}(A)\setminus \sigma_{\mrm{emb}}(A)& = & \{\lambda\in\sigma_{\mrm{p}}(A)\,|\,\lambda<\lambda_N\} .
\end{array}
\] 
With the results concerning the modal decompositions of the solutions to Maxwell's equations in homogeneous waveguides (see \textit{e.g.} \cite{Cess96,PlPo2014,Kim17,PlPo18,PlPoSa18}), one proves that trapped modes, \textit{i.e.} the eigenfunctions in $\boldsymbol{\mL}^2(\Om)$ associated with the eigenvalues of $\sigma_{\mrm{p}}(A)$, decay exponentially as $z\to+\infty$.\\ 
\newline
Above, to simplify the presentation, we considered a domain $\Om$ which is the union of a resonator and one unbounded semi-infinite waveguide. The operator $A$ can be defined similarly when $\Om$ is the union of a resonator and several unbounded semi-infinite waveguides $\Pi_1,\dots,\Pi_M$ as in Figure \ref{GeomBug} left. If one of the $\Pi_m$, $m=1,\dots,M$, has a transverse section $S_m\subset\R^2$ which is not simply connected, then $\sigma_{\mrm{ess}}(A)=[0;+\infty)$. Otherwise $\sigma_{\mrm{ess}}(A)=[\min_{m=1,\dots,M}\lambda_N(S_m);+\infty)$ where $\lambda_N(S_m)>0$ is the smallest positive eigenvalue of the Neumann Laplacian in $S_m$.

\section{Waveguides with complete separation of variables}\label{SectionFullySeparable}
In this section, we consider geometries where we can construct explicitly eigenpairs of the operator $A$ from eigenpairs of the Dirichlet and Neumann scalar Laplacians in 2D domains. Global separation of variables is crucial in the approach.

\subsection{Construction from eigenvalues of the 2D Dirichlet Laplacian}\label{ParaSca}

\begin{figure}[!ht]
\centering
\hspace{-1.4cm}\begin{tikzpicture}[scale=1.2]
\draw[draw=black,fill=gray!30] (0,0) circle (0.9);
\draw[fill=gray!30,draw=none](-0.4,-0.4) rectangle (0.4,2);
\draw (-0.4,-0.4)--(-0.4,2);
\draw (0.4,-0.4)--(0.4,2);
\draw[dashed] (-0.4,2.3)--(-0.4,2);
\draw[dashed] (0.4,2.3)--(0.4,2);
\draw[fill=gray!30,draw=none,rotate=20](-2,-0.3) rectangle (0,0.3);
\draw[fill=gray!30,draw=none,rotate=-20](0,-0.5) rectangle (2,0.5);
\draw[dashed,rotate=-20] (2.3,0.5)--(2,0.5); 
\draw[dashed,rotate=-20] (2.3,-0.5)--(2,-0.5); 
\draw[dashed,rotate=20] (-2.3,0.3)--(-2,0.3); 
\draw[dashed,rotate=20] (-2.3,-0.3)--(-2,-0.3); 
\draw[rotate=20] (-2,0.3)--(-0.24,0.3);
\draw[rotate=-20] (2,0.5)--(0.24,0.5);
\draw[rotate=20] (-2,-0.3)--(-0.17,-0.3);
\draw[rotate=-20] (2,-0.5)--(0.17,-0.5);
\begin{scope}[rotate=-20]
\draw[<->] (1.1,-0.45)--(1.1,0.45);
\end{scope}
\draw[draw=none,fill=gray!30] (0,0) circle (0.89);
\node at (0,0) {$\Om_{\mrm{2D}}$};
\node at (1.5,-0.5) {$H_{\mrm{max}}$};
\end{tikzpicture}\qquad
\includegraphics[width=4.6cm]{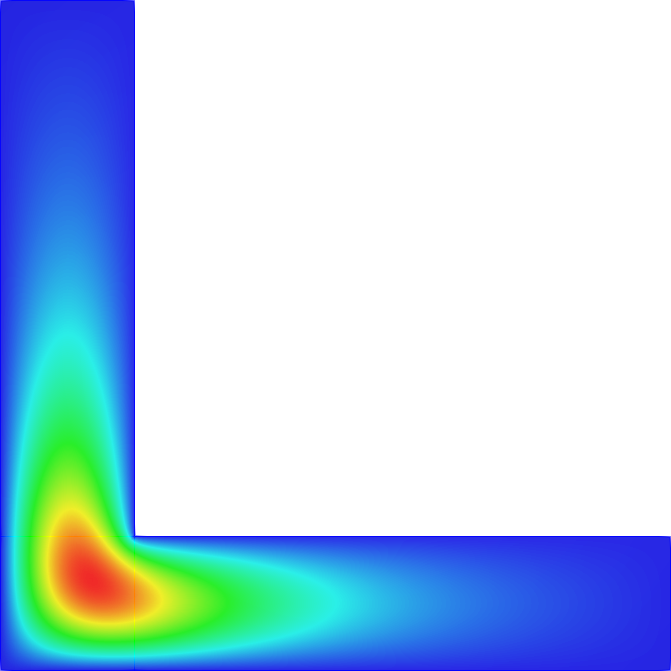}
\begin{tikzpicture}
\draw[<->] (0,0)--(0,0.9);
\node at (0.2,0.45) {$1$};
\end{tikzpicture}
\hspace{-3.6cm}\raisebox{1.6cm}{\begin{tikzpicture}
\draw[->] (0,0)--(0,1);
\draw[->] (0,0)--(1,0);
\node at (1.2,0) {$z$};
\node at (0,1.2) {$y$};
\end{tikzpicture}}
\qquad\qquad
\includegraphics[width=5cm]{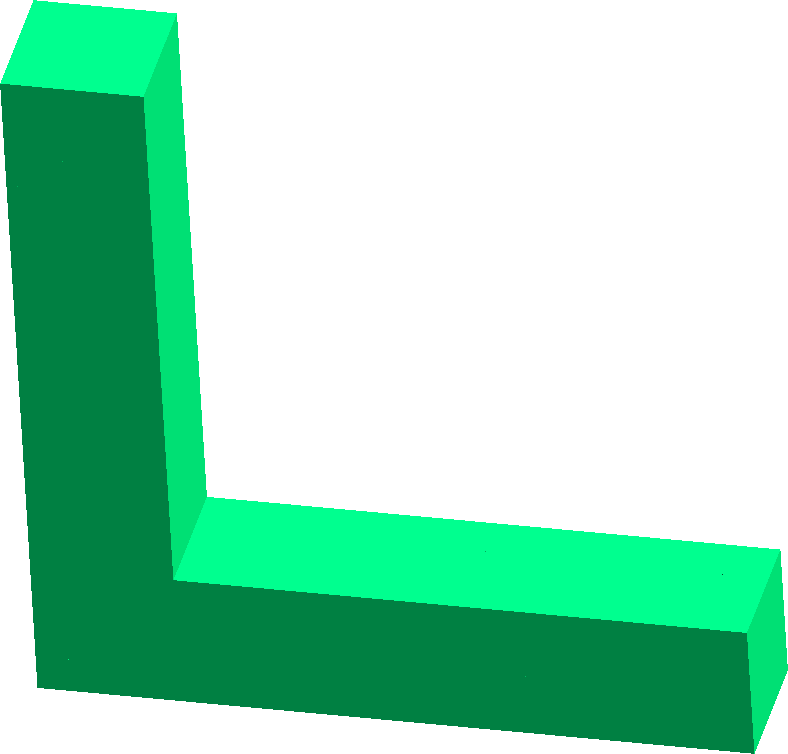}
\hspace{-3cm}\raisebox{1.8cm}{\begin{tikzpicture}
\draw[->] (0,0)--(-0.05,0.9);
\draw[->] (0,0)--(0.8,-0.08);
\draw[->] (0,0)--(0.17,0.4);
\node at (0.25,0.6) {$x$};
\node at (-0.05,1.1) {$y$};
\node at (1,-0.1) {$z$};
\end{tikzpicture}}
\caption{Example of $\Om_{\mrm{2D}}$ (left), trapped mode for the 2D L-shaped domain (center), 3D L-shaped domain (right).\label{LGeom}}
\end{figure}
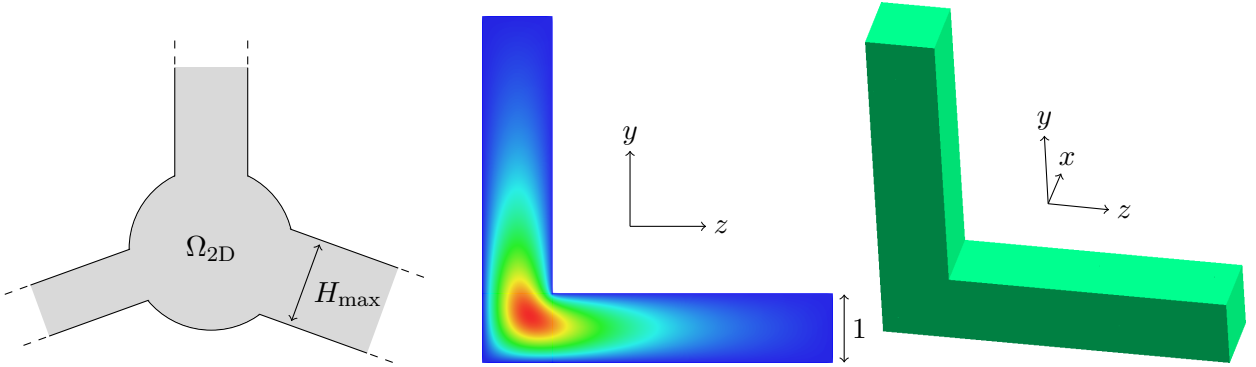

\noindent Let $\Om_{\mrm{2D}}$ be a connected domain of $\R^2$ with Lipschitz boundary which coincides with a union of semi-infinite strips outside of a bounded region (see Figure \ref{LGeom} left). Assume that the geometry is such that the Dirichlet Laplacian in $\Om_{\mrm{2D}}$ admits an eigenvalue $\lambda_\bullet>0$ (in the discrete spectrum or  embedded in the essential spectrum, see \cite{Wits90}). Now for $a>0$, define the 3D waveguide
\begin{equation}\label{DefOmSepVar}
\Om\coloneqq\{(x,y,z)\,|\,x\in(0;a)\mbox{ and }(y,z)\in\Om_{\mrm{2D}}\}. 
\end{equation}
\begin{proposition}\label{PropoSeparationVaria}
Assume $\Om$ as in (\ref{DefOmSepVar}). Then $\lambda_\bullet$ is an eigenvalue of $A$, \textit{i.e.} belongs to $\sigma_{\mrm{p}}(A)$.
\end{proposition}
\begin{remark}
This result already appears in \cite{ExSe90} and \cite[\S IV.A]{GoJa92}.
\end{remark}
\begin{remark}
According to Proposition \ref{DefSigmaEss}, the essential spectrum of the operator $A$ in the above $\Om$ is equal to $[\min(\pi^2/H^2_\mrm{max},\pi^2/a^2);+\infty)$ where $H_\mrm{max}$ stands for the maximum of the heights of the semi-infinite strips of $\Om_{\mrm{2D}}$ (see Figure \ref{LGeom} left). Therefore if $\lambda_\bullet$ is in the discrete spectrum of the Dirichlet Laplacian in $\Om_{\mrm{2D}}$, depending on the value of $a$, $\lambda_\bullet$ can belong to $\sigma_{\mrm{d}}(A)$ or to $\sigma_{\mrm{emb}}(A)$. Otherwise $\lambda_\bullet$ belongs to $\sigma_{\mrm{emb}}(A)$ for all $a>0$.
\end{remark}
\begin{proof}
Let $\varphi\in\mH^1_0(\Om_{\mrm{2D}})$ be an eigenfunction of the Dirichlet Laplacian associated with $\lambda_\bullet$, \textit{i.e.} a non-zero function such that
\begin{equation}\label{PbForPhi}
\Delta\varphi+\lambda_\bullet\varphi=0\mbox{ in }\Om_{\mrm{2D}}.
\end{equation}
Define the field $\E_{\mrm{tr}}$ (``$\mrm{tr}$'' like trapped) such that
\begin{equation}\label{DefEtr1}
\E_{\mrm{tr}}(\bfx)=\left(\begin{array}{c}
\varphi(y,z)\\
0\\
0
\end{array}\right).
\end{equation}
It belongs to $\boldsymbol{\mH}^1(\Om)\coloneqq(\mH^1(\Om))^3$ and we have $\div\,\E_{\mrm{tr}}=0$ as well as $\boldsymbol{\Delta} \E_{\mrm{tr}}+\lambda_\bullet\E_{\mrm{tr}}=0$ in $\Om$ ($\boldsymbol{\Delta}\cdot$ is the vector Laplacian). Since $\curlvec\curlvec\cdot=-\boldsymbol{\Delta}\cdot+\nabla(\div\cdot)$, this ensures that $\E_{\mrm{tr}}$ satisfies the first two lines of (\ref{DefPb0}) with $\lambda=\lambda_{\bullet}$. It remains to check that $\E_{\mrm{tr}}$ is normal to the boundary on $\partial\Om$. This is true on $\{\bfx\in\partial\Om\,|\,x=0\mbox{ or }x=a\}$ because there, one has
\[
\bfnu=\left(\begin{array}{c}
\pm1\\
0\\
0
\end{array}\right)\qquad\mbox{ and }\qquad
\E_{\mrm{tr}}(\bfx)=\left(\begin{array}{c}
\varphi(y,z)\\
0\\
0
\end{array}\right).
\]
On the other hand, there holds $
\E_{\mrm{tr}}=0\mbox{ on }(0;a)\times\partial\Om_{\mrm{2D}}$ because $\varphi=0$ on $\partial\Om_{\mrm{2D}}$. This shows that $\E_{\mrm{tr}}\times\bfnu=0$ on $\partial\Om$. Finally, we deduce that $\lambda_\bullet$ is an eigenvalue of $A$. 
\end{proof}
\noindent This result can be generalized to prove that in these geometries with complete separation of variables, $A$ has actually an unbounded sequence of embedded eigenvalues.

\begin{proposition}\label{PropSeparationVaria}
Assume $\Om$ as in (\ref{DefOmSepVar}). For all $m\in\N\coloneqq\{0,1,\dots\}$, $\lambda_{\bullet}+m^2\pi^2/a^2$ is an eigenvalue of $A$, \textit{i.e.} belongs to $\sigma_{\mrm{p}}(A)$.
\end{proposition}
\begin{proof}
Let $\varphi$ be as in (\ref{PbForPhi}). Define, for $m\in\N$, the field $\E_{\mrm{tr}}\in\boldsymbol{\mL}^2(\Om)$ such that
\[
\E_{\mrm{tr}}(\bfx)=\left(\begin{array}{c}
\cos(m\pi x/a)\varphi(y,z)\\
-\cfrac{m\pi}{a\lambda_\bullet}\,\sin(m\pi x/a)\nabla_{\mrm{2D}}\varphi(y,z)
\end{array}\right)=\left(\begin{array}{c}
\cos(m\pi x/a)\varphi(y,z)\\
-(m\pi/(a\lambda_\bullet))\sin(m\pi x/a)\partial_y\varphi(y,z)\\
-(m\pi/(a\lambda_\bullet))\sin(m\pi x/a)\partial_z\varphi(y,z)
\end{array}\right).
\]
Observe that for $m=0$, this $\E_{\mrm{tr}}$ coincides with the one introduced in (\ref{DefEtr1}). One can check easily that $\div\,\E_{\mrm{tr}}=0$ in $\Om$ and $\curlvec\E_{\mrm{tr}}\in\boldsymbol{\mL}^2(\Om)$. Moreover, there holds 
\[
\boldsymbol{\Delta} \E_{\mrm{tr}}+\bigg(\lambda_\bullet+\cfrac{m^2\pi^2}{a^2}\bigg)\E_{\mrm{tr}}=0\mbox{  in }\Om. 
\] 
Using again that $\curlvec\curlvec\cdot=-\boldsymbol{\Delta}\cdot+\nabla(\div\cdot)$, we deduce that $\E_{\mrm{tr}}$ satisfies the first two lines of (\ref{DefPb0}) with $\lambda=\lambda_{\bullet}+m^2\pi^2/a^2$. On $\{\bfx\in\partial\Om\,|\,x=0\mbox{ or }x=a\}$, $\E_{\mrm{tr}}$ is normal to $\partial\Om$ because there, one has
\[
\bfnu=\left(\begin{array}{c}
\pm1\\
0\\
0
\end{array}\right)\qquad\mbox{ and }\qquad
\E_{\mrm{tr}}(\bfx)=\left(\begin{array}{c}
\pm\varphi(y,z)\\
0\\
0
\end{array}\right).
\]
This is also true on $(0;a)\times\partial\Om_{\mrm{2D}}$, because there, one finds 
\[
\E_{\mrm{tr}}(\bfx)=\left(\begin{array}{c}
0\\
-\cfrac{m\pi}{a\lambda_\bullet}\,\sin(m\pi x/a)\nabla_{\mrm{2D}}\varphi(y,z)
\end{array}\right)
\]
and $\varphi$ vanishes on $\partial\Om_{\mrm{2D}}$. Thus we have $\E_{\mrm{tr}}\times\bfnu=0$ on $\partial\Om$. This proves that $\lambda_{\bullet}+m^2\pi^2/a^2$ is an eigenvalue of $A$.
\end{proof}

\noindent This can be used to show that $A$ has an unbounded sequence of eigenvalues in the L-shaped domain represented in Figure \ref{LGeom} right. More precisely, set 
\begin{equation}\label{Lshape2D}
\Om_{\mrm{2D}}\coloneqq\{(y,z)\in(0;+\infty)^2\,|\,y<1\mbox{ or }z<1\}
\end{equation}
and define $\Om\coloneqq(0;1)\times\Om_{\mrm{2D}}$. In this geometry, Proposition \ref{DefSigmaEss} ensures that $\sigma_{\mrm{ess}}(A)=[\pi^2;+\infty)$. On the other hand, it is known from \cite{ExSt89}, \cite[Prop.\,4.1]{DaLR12} that the Dirichlet Laplacian in $\Om_{\mrm{2D}}$ has an eigenvalue $\lambda_\bullet<\pi^2$ in its discrete spectrum (see a corresponding eigenfunction in Figure \ref{LGeom} center). Numerically, for example with \texttt{Freefem++}\footnote{\url{https://freefem.org/} \cite{Hech12}.}, one obtains $\lambda_\bullet\approx 9.1722 \approx 0.9293\pi^2$. Proposition \ref{PropSeparationVaria} guarantees that for all $m\in\N$, $\lambda_{\bullet}+m^2\pi^2/a^2$ is an eigenvalue of $A$. More precisely, $\lambda_{\bullet}$ belongs to $\sigma_{\mrm{d}}(A)$ whereas $\sigma_{\mrm{emb}}(A)$ contains an unbounded sequence.

\subsection{Construction from eigenvalues of the 2D Neumann Laplacian}

\begin{figure}[!ht]
\centering
\includegraphics[width=7cm]{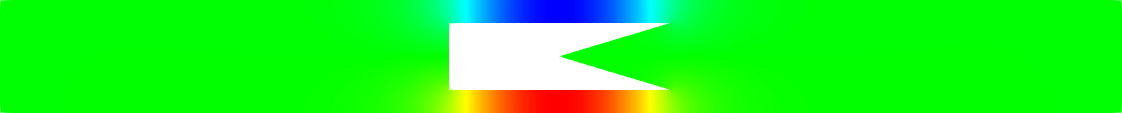}\\[12pt]
\includegraphics[trim={0cm 0cm 4cm 0cm},clip,width=7cm]{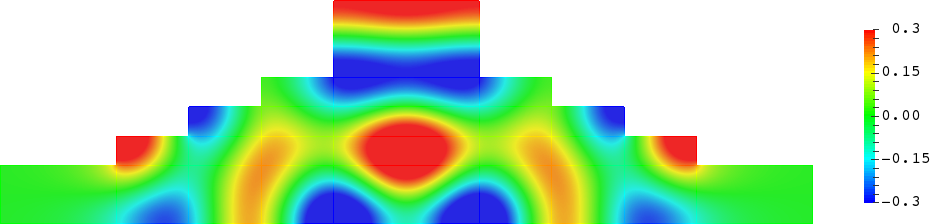}\quad\includegraphics[trim={0cm 0cm 4cm 0cm},clip,width=7cm]{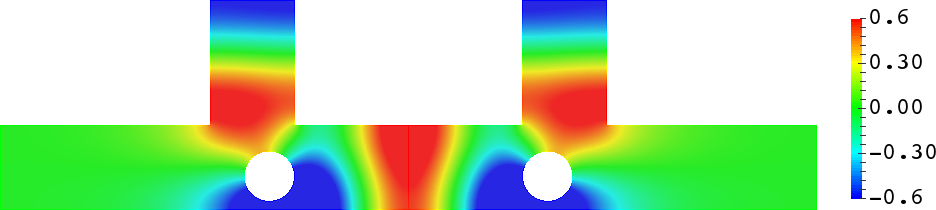}
\caption{Examples of trapped modes for the 2D Neumann Laplacian. The last two come respectively from \cite{ChPa18}, \cite{ChPa19}. \label{NeumannTrappedM}}
\end{figure}

\noindent Consider again $\Om_{\mrm{2D}}$ a domain with Lipschitz boundary which coincides with a union of semi-infinite strips outside of a bounded region. Assume this time that the geometry is such that the Neumann Laplacian in $\Om_{\mrm{2D}}$ admits an (embedded) eigenvalue $\lambda_\bullet>0$. Existence of $\lambda_\bullet$ can be shown for example in waveguides which have symmetries as in Figure \ref{NeumannTrappedM} (see \cite{EvLV94} for the classical technique as well as \cite{ChPa18,ChPa19} for other strategies). For $a>0$, define the 3D domain
\begin{equation}\label{DefOmSepVarBis}
\Om\coloneqq(0;a)\times\Om_{\mrm{2D}}.
\end{equation}
\begin{proposition}
Assume $\Om$ as in (\ref{DefOmSepVarBis}). For all $m\in\N^\ast\coloneqq\{1,2,\dots\}$, $\lambda_{\bullet}+m^2\pi^2/a^2$ is an eigenvalue of $A$ and more precisely, belongs to $\sigma_{\mrm{emb}}(A)$.
\end{proposition}
\begin{proof}
Let $\varphi\in\mH^1(\Om_{\mrm{2D}})$ be an eigenfunction 
of the Neumann Laplacian associated with $\lambda_\bullet$, \textit{i.e.} a function such that
\[
\begin{array}{|rcll}
\Delta\varphi+\lambda_\bullet\varphi&=&0&\mbox{ in }\Om_{\mrm{2D}}\\[2pt]
\partial_{\boldsymbol{n}}\varphi&=&0& \mbox{ on }\partial\Om_{\mrm{2D}},
\end{array}
\]
where $\boldsymbol{n}$ denotes the outward normal unit vector to $\partial\Om_{\mrm{2D}}$. Define, for $m\in\N^\ast$,
\[
\E_{\mrm{tr}}(\bfx)=\left(\begin{array}{c}
0\\
\sin(m\pi x/a)\curlvec_{\mrm{2D}}\,\varphi(y,z)
\end{array}\right)=\left(\begin{array}{c}
0\\
\phantom{-}\sin(m\pi x/a)\partial_z\varphi(y,z)\\
-\sin(m\pi x/a)\partial_y\varphi(y,z)
\end{array}\right).
\]
One can check as in the proof of Proposition \ref{PropSeparationVaria} that $\E_{\mrm{tr}}$ belongs to $\boldsymbol{\mX}_N(\Om)$ and satisfies (\ref{DefPb0}) for $\lambda=\lambda_{\bullet}+m^2\pi^2/a^2$. Let us simply detail the boundary conditions. On $\partial\Om$, one finds
\[
\E_{\mrm{tr}}\times\bfnu=\left(\begin{array}{c}
\sin(m\pi x/a)(\partial_y\varphi\,\nu_y+\partial_z\varphi\,\nu_z)\\
-\sin(m\pi x/a)\partial_y\varphi\,\nu_x\\
-\sin(m\pi x/a)\partial_z\varphi\,\nu_x
\end{array}\right).
\]
Thus, we have $\E_{\mrm{tr}}\times\bfnu=0$ on $(0;a)\times\partial\Om_{\mrm{2D}}$ because $\partial_{\boldsymbol{n}}\varphi=\nu_x=0$ there. And there holds $\E_{\mrm{tr}}=0$, and so $\E_{\mrm{tr}}\times\bfnu=0$, at $x=0$ and at $x=a$ because $\sin(m\pi x/a)=0$ there. Finally, this shows that for all $m\in\N^\ast$, $\lambda_{\bullet}+m^2\pi^2/a^2$ is an eigenvalue of $A$. Since $\sigma_{\mrm{ess}}(A)=[\min(\pi^2/H^2_\mrm{max},\pi^2/a^2);+\infty)$, these eigenvalues are all embedded in $\sigma_{\mrm{ess}}(A)$.
\end{proof}

\begin{remark}\label{RmkFourierSeries}
Denote generically here by $\lambda_D(0;a)/\lambda_N(0;a)$ (resp. $\lambda_D(\Om_{\mrm{2D}})/\lambda_N(\Om_{\mrm{2D}})$) the eigenvalues of the Dirichlet/Neumann Laplacian in $(0;a)$ (resp. $\Om_{\mrm{2D}}$). Above, we proved that all the quantities 
\begin{equation}\label{FormEigProduct}
\lambda_D(\Om_{\mrm{2D}})+\lambda_N(0;a),\qquad\qquad\lambda_N(\Om_{\mrm{2D}})+\lambda_D(0;a),
\end{equation}
are eigenvalues of $A$. Using decomposition in Fourier series in the $x$ direction, one can show that actually all the eigenvalues of $A$ write as (\ref{FormEigProduct}). Thus we retrieve a result established in bounded domains in the work \cite{zbMATH07224760} concerning Maxwell's eigenvalues in product domains.
\end{remark}

\section{Waveguides with separation of variables in the resonator}\label{Section1}

In the previous section it was important to have separation of variables in the whole waveguide. Here we give examples of geometries where we have separation of variables only in the resonator and where $\sigma_{\mrm{d}}(A)\ne\emptyset$. For these waveguides, we cannot construct explicitly trapped modes for $A$. To prove the existence of eigenvalues, we will work with the min-max principle.

\subsection{Strategy}\label{paraStrategy}

Assume that
\begin{equation}\label{HypoOm}
\mbox{$\Om$ is as in (\ref{DefOm1})--(\ref{DefOm2}) and $S$, the cross-section of $\Pi$, is simply connected.} 
\end{equation}
Since $A$ is self-adjoint, according to the min-max principle (cf. \cite[Thm.\,10.2.2]{BiSo87}), we know that there holds
\begin{equation}\label{MaxMinFormula}
\inf\,\sigma(A)=\underset{\E\in\boldsymbol{\mX}_N(\Om)\setminus\{0\}}{\inf}\cfrac{\dsp\int_{\Om}|\curlvec\E|^2\,d\bfx}{\dsp\int_{\Om}|\E|^2\,d\bfx}\,.
\end{equation}
Combining Proposition \ref{DefSigmaEss} and (\ref{MaxMinFormula}), we infer that if we are able to exhibit some $\E\in\boldsymbol{\mX}_N(\Om)\setminus\{0\}$ such that
\begin{equation}\label{RelationsMinMax}
\cfrac{\dsp\int_{\Om}|\curlvec\E|^2\,d\bfx}{\dsp\int_{\Om}|\E|^2\,d\bfx}<\lambda_N\qquad\Leftrightarrow\qquad \dsp\int_{\Om}|\curlvec\E|^2\,d\bfx-\lambda_N\int_{\Om}|\E|^2\,d\bfx<0,
\end{equation}
then necessarily $\sigma_{\mrm{d}}(A)\ne\emptyset$. This is a traditional strategy to prove existence of discrete spectrum for self-adjoint operators. The question now is: how to build \textit{ad hoc} test fields? A natural idea is to work with 
\begin{equation}\label{DefEpGe}
\E_p=\begin{array}{|cl}
\E_\mathcal{R} & \mbox{ in }\mathcal{R}\\
0 & \mbox{ in }\Pi
\end{array}
\end{equation}
(``p'' like particular) where $\E_\mathcal{R}$ is an eigenfunction of the resonator problem
\begin{equation}\label{ResPb}
\begin{array}{|rcll}
\curlvec\curlvec\E_\mathcal{R} &= & \lambda_\mathcal{R}\E_\mathcal{R} & \mbox{ in }\mathcal{R}\\[3pt]
\E_\mathcal{R}\times\bfnu & =&0 &  \mbox{ on }\partial\mathcal{R}.
\end{array}
\end{equation}
Then we would obtain
\[
\cfrac{\dsp\int_{\Om}|\curlvec\E_p|^2\,d\bfx}{\dsp\int_{\Om}|\E_p|^2\,d\bfx}=\cfrac{\dsp\int_{\mathcal{R}}|\curlvec\E_\mathcal{R}|^2\,d\bfx}{\dsp\int_{\Om_\mathcal{R}}|\E_\mathcal{R}|^2\,d\bfx}=\lambda_\mathcal{R}
\]
and if $\lambda_\mathcal{R}<\lambda_N$, one could conclude to the existence of discrete spectrum for $A$. The problem here is that though there holds $\E_p\in\boldsymbol{\mH}_N(\curlvec)$, in general $\div\,\E_p\ne0$ in $\Om$ except if 
\begin{equation}\label{ResBC}
\E_\mathcal{R}\cdot\bfnu=0\quad\mbox{ on }\partial\mathcal{R}\cap\partial\Pi.
\end{equation}
Therefore in general, some field $\E_p$ defined as in (\ref{DefEpGe}) does not belong to $\boldsymbol{\mX}_N(\Om)$. However in this section we will see that if $\mathcal{R}$ is a product domain (see examples in Figure \ref{Marteau}), then condition (\ref{ResBC}) can be satisfied.

\begin{figure}[!ht]
\centering
\includegraphics[width=7cm,trim={6cm 0cm 3cm 0cm},clip]{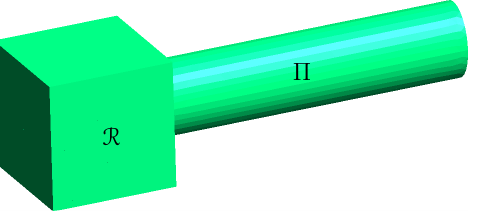}
\includegraphics[width=6.8cm,trim={5.7cm 0cm 3.4cm 0cm},clip]{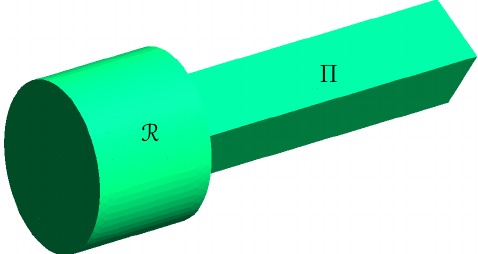}
\caption{Waveguides where the resonator is a product domain.\label{Marteau}}
\end{figure}

\subsection{A first simple case}

Assume first that the resonator $\mathcal{R}$ coincides with the cuboid such that for $a,b,L>0$, 
\begin{equation}\label{DefRCuboid}
\mathcal{R}=(-a/2;a/2)\times(-b/2;b/2)\times(-L;0)
\end{equation}
(see Figure \ref{Marteau} left). Define the field $\E_{p}$ as in (\ref{DefEpGe}) with 
\begin{equation}\label{DefEpPart}
\E_{\mathcal{R}}(\bfx)=\left(\begin{array}{c}
0\\
\cos(\pi x/a)\sin(\pi z/L)\\
0
\end{array}\right)\mbox{ in }\mathcal{R}.
\end{equation}
The crucial point here is that one has $\E_{\mathcal{R}}=0$ at $z=0$, which ensures that condition (\ref{ResBC}) is satisfied because $\partial\mathcal{R}\cap\partial\Pi\subset \{\bfx\in\R^3\,|\,z=0\}$. Besides, one has $\E_{\mathcal{R}}=0$, and so $\E_{\mathcal{R}}\times\bfnu=0$, at $x=\pm a/2$ and at $z=-L$. At $y=\pm b/2$, there holds $\bfnu=(0,\pm1,0)^{\top}$ so that $\E_{\mathcal{R}}$ is also normal to the boundary there. This is enough to conclude that $\E_{\mathcal{R}}\times\bfnu=0$ on $\partial\mathcal{R}$. All in all, this shows that $\E_{p}\in\boldsymbol{\mX}_N(\Om)$. A direct calculation gives
\[
\cfrac{\dsp\int_{\Om}|\curlvec\E_{p}|^2\,d\bfx}{\dsp\int_{\Om}|\E_{p}|^2\,d\bfx}=\cfrac{\pi^2}{a^2}+\cfrac{\pi^2}{L^2}\,.
\]
Therefore we have the following result:
\begin{proposition}\label{PropoCasPart}
Assume $\Om$ as in (\ref{HypoOm}) with $\mathcal{R}$ as in (\ref{DefRCuboid}). If $\pi^2/a^{2}+\pi^2/L^{2}<\lambda_N$, then there holds $\sigma_{\mrm{d}}(A)\ne\emptyset$. 
\end{proposition}
\noindent Thus for any cylinder $\Pi$ with simply connected cross section $S$, $A$ has a non-empty discrete spectrum for $a$, $L$ large enough (independently of $b$).

\subsection{Generalization: working with TE modes in the resonator}\label{paraGene}

\noindent This construction can be generalized to resonators of the form 
\begin{equation}\label{DefSections}
\mathcal{R}=S_{\mathcal{R}}\times(-L;0),
\end{equation}
where $S_{\mathcal{R}}$ is a bounded domain of $\R^2$ with Lipschitz boundary (see Figure \ref{Marteau} right for a case where $S_{\mathcal{R}}$ is a disk). Denote by $\lambda_N(S_{\mathcal{R}})$ the first positive eigenvalue of the 2D Neumann Laplacian in $S_{\mathcal{R}}$. 
\begin{proposition}\label{PropoTMmodes}
Assume $\Om$ as in (\ref{HypoOm}) with $\mathcal{R}$ as in (\ref{DefSections}). If $\lambda_N(S_{\mathcal{R}})+\pi^2/L^2<\lambda_N$, then there holds $\sigma_{\mrm{d}}(A)\ne\emptyset$. 
\end{proposition}
\begin{proof}
The idea is to work with the TE modes in the resonator (see (\ref{DefTEmodes}) for their definition in $\Pi$). For $\lambda\in(\lambda_N(S_{\mathcal{R}});\lambda_N)$, there are propagating TE modes in the resonator but not in the cylinder $\Pi$. By combining two propagating TE modes in $\mathcal{R}$, we can create a field that vanishes at $z=0$ and that we can extend by zero. More precisely, define $\E_{p}$ as in (\ref{DefEpGe}) with
\begin{equation}\label{DefEModeTEbis}
\E_{\mathcal{R}}(\bfx)= \left(\begin{array}{c}
\curlvec_{\mrm{2D}}\,\varphi_N(x,y)\\
0
\end{array}\right)\sin(\pi z/L)
\mbox{ in }\mathcal{R}.
\end{equation}
Here $\varphi_N$ is an eigenfunction of the Neumann Laplacian in $S_{\mathcal{R}}$ associated with $\lambda_N(S_{\mathcal{R}})$. Let us emphasize that (\ref{DefEpPart}) is a particular case of such construction with $\varphi_N(x,y)=\sin(\pi x/a)$ (up to a multiplicative constant). Since $\E_{\mathcal{R}}=0$ at $z=0$, condition (\ref{ResBC}) is satisfied. Moreover, one can check that $\E_{\mathcal{R}}\times\bfnu=0$ on $\partial\mathcal{R}$. Thus $\E_{p}$ belongs to $\boldsymbol{\mX}_N(\Om)$. By a direct calculation, we obtain
\[
\cfrac{\dsp\int_{\Om}|\curlvec\E_{p}|^2\,d\bfx}{\dsp\int_{\Om}|\E_{p}|^2\,d\bfx}=\lambda_N(S_{\mathcal{R}})+\cfrac{\pi^2}{L^2}\,.
\]
With the min-max principle (\ref{MaxMinFormula}), this gives the conclusion of the proposition.
\end{proof}
\noindent Let us mention that with this approach we can find geometries such that the total multiplicity of $\sigma_{\mrm{d}}(A)$ is as large as desired. To proceed it suffices to work with $S_{\mathcal{R}}$ large enough so that the Neumann Laplacian in $S_{\mathcal{R}}$ has several eigenvalues below $\lambda_N$. Then with the corresponding eigenfunctions, as in (\ref{DefEModeTEbis}), we can create families of linearly independent fields of $\boldsymbol{\mX}_N(\Om)$ such that (\ref{RelationsMinMax}) is satisfied for $L$ large enough, which is enough to conclude with the min-max principle.

\subsection{Absence of monotonicity of the spectrum with respect to the geometry}\label{AbsenceOfMono}

It is known that there is no monotonicity result for the spectrum of Maxwell's operator with respect to the geometric inclusion in bounded domains. More precisely, we can have $\om_1\subset\om_2$ with $\inf\,\sigma(A_1)<\inf\,\sigma(A_2)$, where $\sigma(A_i)$ denotes the spectrum of the operator (\ref{DefOpA}) in the geometry $\om_i$, $i=1,2$. One can check that this occurs for example for the two domains illustrated in Figure \ref{NoMinMaxBounded}. For related questions, we also refer the reader to \cite{KrLZ25} where the authors show that the problems of minimization and maximization of the first eigenvalue of the vector Laplacian under perimeter or volume constraints are degenerate. 

\begin{figure}[!ht]
\centering
\begin{tikzpicture}[scale=1]
\draw[fill=blue!20,draw=blue](0,0) rectangle (6,2);
\draw[fill=blue!20,draw=blue] (6,0)--(7,1)--(7,3)--(6,2)--cycle;
\draw[fill=blue!20,draw=blue] (7,3)--(6,2)--(0,2)--(1,3)--cycle;
\draw[fill=none,draw=blue,dashed] (0,0)--(1,1)--(7,1);
\draw[fill=none,draw=blue,dashed] (1,3)--(1,1);

\draw[fill=red!40,opacity=0.4,draw=red] (0.08,0)--(6,1.8)--(5.92,2)--(0,0.2)--cycle;
\draw[fill=red!40,opacity=0.4,draw=red] (6,1.8)--(5.92,2)--(6.92,3)--(7,2.8)--cycle;
\draw[fill=red!40,opacity=0.4,draw=red] (0,0.2)--(5.92,2)--(6.92,3)--(1,1.2)--cycle;
\draw[fill=none,draw=red,dashed] (0.08,0)--(1.08,1)--(7,2.8);
\draw[fill=none,draw=red,dashed] (1.08,1)--(1,1.2);
\node at (-0.4,1) {$\textcolor{blue}{\om_2}$};
\node at (-0.4,0.2) {$\textcolor{red}{\om_1}$};
\end{tikzpicture}
\caption{Classical examples of bounded domains $\om_1\subset\om_2$ such that $\inf\,\sigma(A_1)<\inf\,\sigma(A_2)$.\label{NoMinMaxBounded}}
\end{figure}
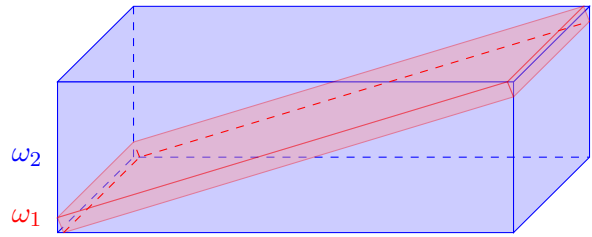

\noindent In the present paragraph, we show a similar property for the unbounded waveguides we consider in this work. To proceed, we construct two domains $\Om_1$, $\Om_2$ with $\Om_1\subset\Om_2$ such that  $\inf \sigma(A_1) <\inf \sigma(A_2)$.\\

\begin{figure}[!ht]
\centering
\includegraphics[width=7.5cm]{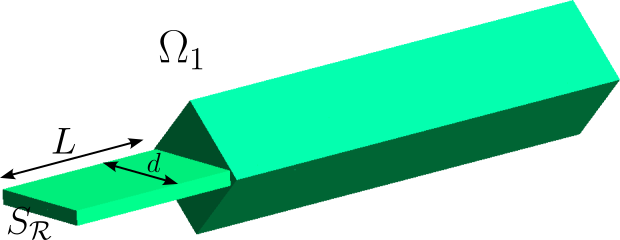}\qquad \includegraphics[width=7.5cm]{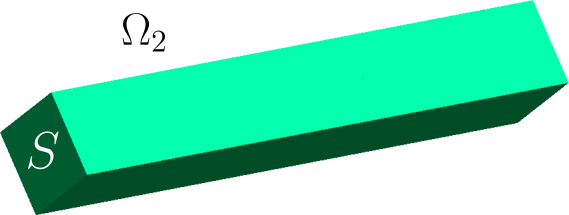}
\caption{Domains $\Om_1$ (left) and $\Om_2$ (right). Though $\Om_1\subset\Om_2$, we have $\inf \sigma(A_1) <\inf \sigma(A_2)$.\label{DefDomains}}
\end{figure}

\noindent Consider $d\in(1;\sqrt{2})$ and introduce $\eps>0$ small enough such that $\overline{S_\mathcal{R}}\subset S$, where
\[
S_\mathcal{R}\coloneqq(-d/2;d/2)\times(-\eps/2;\eps/2),\qquad 
S\coloneqq\{(x,y)\in\R^2\,|\,|x|+|y|<\sqrt{2}/2\}.
\]
For $L>0$, set
\[
\Om_1\coloneqq S_\mathcal{R}\times(-L;0]\cup S\times(0;+\infty),\qquad \Om_2\coloneqq S\times(-L;+\infty)
\]
(see Figure \ref{DefDomains}). Working with a symmetry argument as in Section \ref{SectionSym}, one can show that the operator associated with (\ref{DefPb0}) in $\Om_2$ has only essential spectrum coinciding with $[\pi^2;+\infty)$. Thus, there holds $\sigma(A_2)=[\pi^2;+\infty)$. Now Proposition \ref{PropoCasPart} ensures that if $\pi^2/d^{2}+\pi^2/L^{2}<\pi^2$, which is true for $L$ large enough, then there holds $\sigma_{\mrm{d}}(A_1)\ne\emptyset$. In that case, we have $\inf \sigma(A_1) <\pi^2=\inf \sigma(A_2)$.

\subsection{Working with TEM modes in the resonator}\label{paraTEM}

\begin{figure}[!ht]
\centering
\includegraphics[width=6cm]{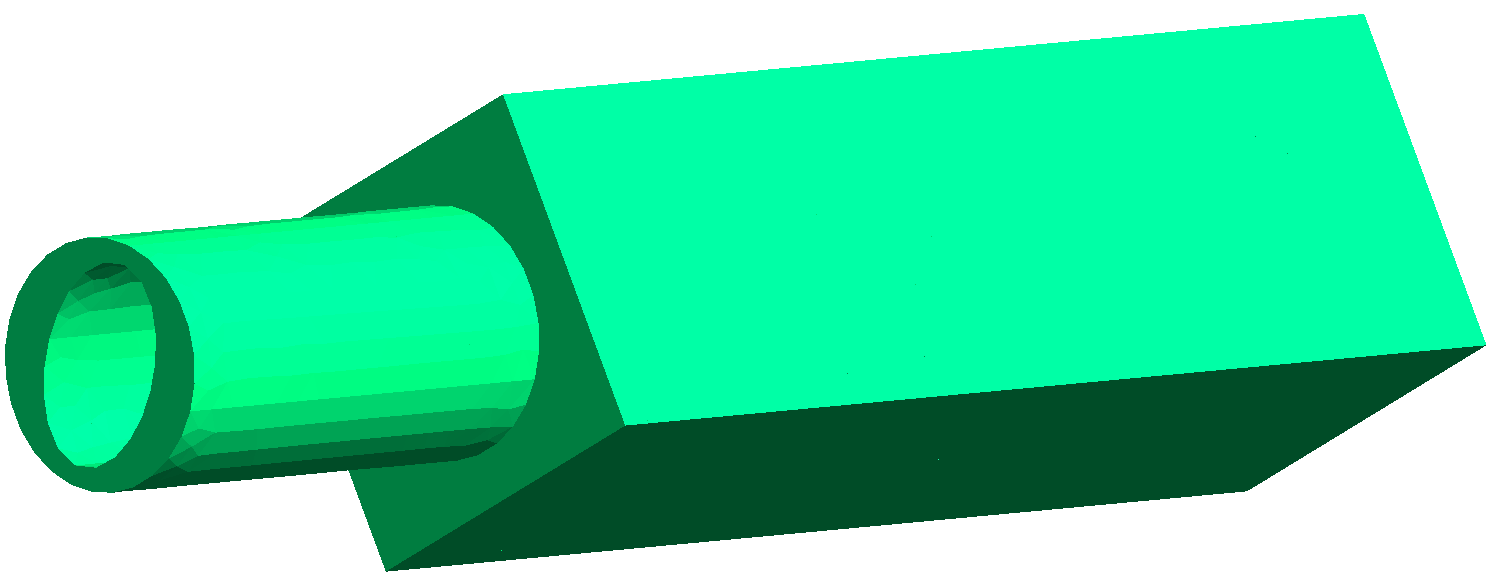}\qquad
\begin{tikzpicture}[scale=1.3]
\draw[draw=black,fill=gray!30] (0,0) circle (0.9);
\draw[draw=none,fill=gray!30] (0,0) circle (0.89);
\draw [draw=black,fill=white](0,0) ellipse (0.56  and 0.84);
\node at (-0.72,0) {$S_\circ$};
\end{tikzpicture}
\caption{Domain $\Om$ (left) with a resonator having a non-simply connected section $S_\circ$ (right). \label{CoaxialTEM}}
\end{figure}
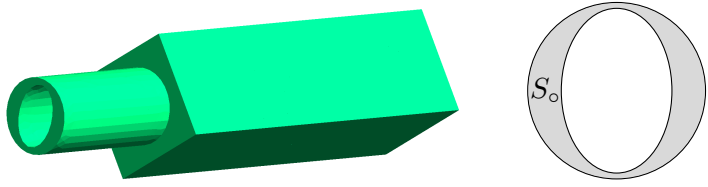

\noindent With the approach presented in \S\ref{paraGene}, we have to compare the first positive eigenvalues of the Neumann Laplacians in $S_\mathcal{R}$ and in $S$ to claim that $\sigma_{\mrm{d}}(A)\ne\emptyset$. For generic sections $S_\mathcal{R}$ and $S$, this requires some numerical computations. Here we show that by working with some $S_\mathcal{R}$ which are not simply connected, we can ensure that $\sigma_{\mrm{d}}(A)\ne\emptyset$ for $L$ large enough, independently of the geometries of the sections. This is directly related to the well-known fact that if the section of $\mathcal{R}$ is not simply connected, propagating TEM modes always exist in $\mathcal{R}$. \\
\newline 
Let $\tilde{S}\subset \R^2$ be a bounded simply connected domains with Lipschitz boundary. For $\mathcal{O}$ a simply connected domain with Lipschitz boundary such that $\overline{\mathcal{O}}\subset \tilde{S}$, define 
\[
S_\circ\coloneqq\tilde{S}\setminus\overline{\mathcal{O}}
\]
(Figure \ref{CoaxialTEM} right). Note that $S_\circ$ is not simply connected. Let us work with a resonator of the form 
\begin{equation}\label{DefSectionsNotSC}
\mathcal{R}=S_\circ\times(-L;0)
\end{equation}
(Figure \ref{CoaxialTEM} left). For $\lambda>0$, introduce the TEM mode
\[
\E^{\mrm{TEM}}_\pm(\bfx)=\left( \begin{array}{c}
\nabla \varphi(x,y) \\[2pt]
0
\end{array}
\right)e^{\pm i\sqrt{\lambda}z},
\]
where $\varphi$ is the function of $\mH^1(S_\circ)$ such that 
\begin{equation}\label{DefPhi}
\begin{array}{|rccl}
\Delta\varphi&=&0 & \mbox{ in }S_\circ\\[2pt]
\varphi&=&1 & \mbox{ on }\partial\mathcal{O}\\[2pt]
\varphi&=&0 & \mbox{ on }\partial \tilde{S}.
\end{array}
\end{equation}
One can check that $\E^{\mrm{TEM}}_\pm$ solves (\ref{DefPb0}) in $\mathcal{R}$.

\begin{proposition}
Assume $\Om$ as in (\ref{HypoOm}) with $\mathcal{R}$ as in (\ref{DefSectionsNotSC}). If $\pi^2/L^2<\lambda_N$, then there holds $\sigma_{\mrm{d}}(A)\ne\emptyset$. 
\end{proposition}
\begin{proof}
Define the field $\E_{p}$ as in (\ref{DefEpGe}) with 
\[
\E_\mathcal{R}(\bfx)= \left( \begin{array}{c}
\nabla \varphi(x,y) \\[2pt]
0
\end{array}
\right)\sin(\pi z/L),
\]
$\varphi$ being the function defined in (\ref{DefPhi}). One can check as above that $\E_p\in\boldsymbol{\mX}_N(\Om)$. Moreover, we have 
\[
 \int_{\Om} |\curlvec\E_p|^2\,d\bfx = \cfrac{\pi^2}{L^2}\int_{\Om} |\E_p|^2\,d\bfx .
\]
Again with the min-max principle (\ref{MaxMinFormula}), we obtain the conclusion of the proposition.
\end{proof}
\noindent Let us formulate two comments concerning this result. First, note that the domain $\Om$ represented in Figure \ref{CoaxialTEM} offers another example showing that monotonicity of the spectrum with respect to the geometric inclusion does not hold for Maxwell's operator. Second, observe that only the length of the resonator matters in the upper bound for the eigenvalue in the discrete spectrum of $A$, and not the area of $S_\circ$ (in particular the hole in $S_\circ$ can be as small as desired). This example where the topology plays a key role is very specific to electromagnetism and has no equivalent for Dirichlet/Neumann Laplacians.

\subsection{Working with TM modes in the resonator}\label{ParaTMmodes}

In \S\ref{paraGene} (resp. \S\ref{paraTEM}), we exploited TE (resp. TEM) modes of the resonator to derive a criterion ensuring the existence of discrete spectrum for $A$. One may wonder how to proceed with TM modes in $\mathcal{R}$. We will see that the approach must be adapted. Consider a resonator as in (\ref{DefSections}), \textit{i.e.} of the form 
\[
\mathcal{R}=S_{\mathcal{R}}\times(-L;0).
\]
Denote by $\lambda_D(S_{\mathcal{R}})$ the first eigenvalue of the Dirichlet Laplacian in $S_{\mathcal{R}}$. By combining two propagating TM modes in $\mathcal{R}$ (see (\ref{DefTMmodes}) for their definition in $\Pi$), it would be natural to try to work with the field $\hat{\E}_p$ such that 
\[ 
\hat{\E}_p(\bfx)=\left(\begin{array}{c}\nabla \varphi_D(x,y)\sin(\pi z/L)\\[2pt]-\alpha \varphi_D(x,y)\cos(\pi z/L)\end{array}\right)\mbox{ in }\mathcal{R}.
\]
Here $\alpha\coloneqq L\lambda_D(S_{\mathcal{R}})/\pi$ and $\varphi_D$ is an eigenfunction of the Dirichlet Laplacian in $S_{\mathcal{R}}$ associated with $\lambda_D(S_{\mathcal{R}})$. However, since the third component of $\hat{\E}_p$ does not vanish at $z=0$, extension by zero is not allowed in $\boldsymbol{\mX}_N(\Om)$. To circumvent this difficulty, let us remove to $\hat{\E}_p$ a divergence free (in $\mathcal{R}$) field which has zero tangential trace on $\partial\mathcal{R}\cap\partial\Om$ and which has the same trace as $\hat{\E}_p$ at $z=0$. This leads us to work with $\E_{p}$ such that
\begin{equation}\label{DefEpTM}
\E_{p}(\bfx)= \left(\begin{array}{c}
\nabla \varphi_D(x,y)\sin(\pi z/L)\\[2pt]
\alpha\varphi_D(x,y)(1-\cos(\pi z/L))
\end{array}\right)
\mbox{ in }\mathcal{R},\qquad \E_{p}(\bfx)=\left(\begin{array}{c}
0\\
0\\
0
\end{array}\right)\mbox{ in }\Pi.
\end{equation}
Observe that $\E_{p}=0$ at $z=0^\pm$. Moreover, we have $\div\,\E_{p}=0$ in $\Om$ and $\E_{p}\times\bfnu=0$ on $\partial\Om$, which shows that $\E_{p}\in\boldsymbol{\mX}_N(\Om)$. We find
\[
\curlvec\E_{p}(\bfx)= \left(\begin{array}{c}
\phantom{-}\partial_y\varphi_D(\alpha-\cos(\pi z/L)(\alpha+\pi/L))\\[2pt]
-\partial_x\varphi_D(\alpha-\cos(\pi z/L)(\alpha+\pi/L))\\[2pt]
0
\end{array}\right)\mbox{ in }\mathcal{R}.
\]
Assuming that $\int_{S_{\mathcal{R}}}\varphi_D^2\,dxdy=1$, we obtain
\[
\begin{array}{rcl}
\dsp\int_{\Om} |\curlvec\E_{p}|^2\,d\bfx&=&\dsp\lambda_D(S_{\mathcal{R}})\int_{-L}^0(\alpha-\cos(\pi z/L)(\alpha+\pi/L))^2\,dz\\[10pt]
&=&\cfrac{3 L^3(\lambda_D(S_{\mathcal{R}}))^3}{2\pi^2} + L\lambda_D(S_{\mathcal{R}})^2 + \cfrac{\pi^2\lambda_D(S_{\mathcal{R}})}{2L}
\end{array}
\]
\[
\int_{\Om} |\E_{p}|^2\,d\bfx=\int_{-L}^0\lambda_D(S_{\mathcal{R}})\sin(\pi z/L)^2+\alpha^2(1-\cos(\pi z/L))^2\,dz=\cfrac{3 L^3(\lambda_D(S_{\mathcal{R}}))^2}{2\pi^2} + \cfrac{L\lambda_D(S_{\mathcal{R}})}{2}\,.
\]
Taking the limit $L\to+\infty$ in the quotient of these two quantities, with the min-max principle (\ref{MaxMinFormula}), we establish the following result.
\begin{proposition}
Assume $\Om$ as in (\ref{HypoOm}) with $\mathcal{R}$ as in (\ref{DefSections}). If $\lambda_D(S_{\mathcal{R}})<\lambda_N$, then for $L$ large enough there holds $\sigma_{\mrm{d}}(A)\ne\emptyset$. 
\end{proposition}
\begin{remark}
It is interesting to compare this result with the one provided by Proposition \ref{PropoTMmodes}. From \cite{Filo05}, as already mentioned, we know that $\lambda_N(S_{\mathcal{R}})<\lambda_D(S_{\mathcal{R}})$. On the other hand, it is easy to show that the $\E^{\mrm{TE}}_{p}$ defined in (\ref{DefEModeTEbis}) and the $\E^{\mrm{TM}}_{p}$ defined in (\ref{DefEpTM}) are linearly independent in $\boldsymbol{\mX}_N(\Om)$. With the min-max principle we deduce that if $\lambda_D(S_{\mathcal{R}})<\lambda_N$, then for $L$ large enough, the total multiplicity of $\sigma_{\mrm{d}}(A)$ is at least two. 
\end{remark}
\begin{remark}
Let us mention that the method we use in this paragraph to find relevant test fields differs from the strategy presented in \S\ref{paraStrategy}. Indeed, the restriction to $\mathcal{R}$ of the $\E_p$ defined in (\ref{DefEpTM}) is not an eigenfunction of the resonator problem (\ref{ResPb}) but a combination of eigenfunctions associated with different eigenvalues.
\end{remark}

\noindent Let us conclude this section by observing that the assumption made on the geometry can be relaxed. More precisely, we worked in (\ref{DefSections}) with a resonator of the form $\mathcal{R}=S_{\mathcal{R}}\times(-L;0)$. Since we use extension by zero, everything presented above can be adapted to situations where only a part of the resonator coincides, in \textit{ad hoc} coordinates, with a cylinder. This is the case for example in the domain represented in Figure \ref{GeomComplex}.

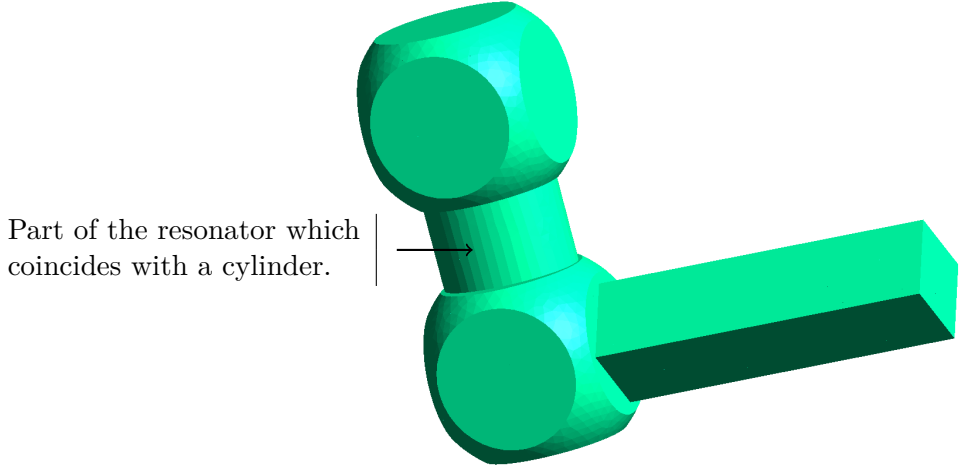
\begin{figure}[!ht]
\centering
\includegraphics[width=8cm]{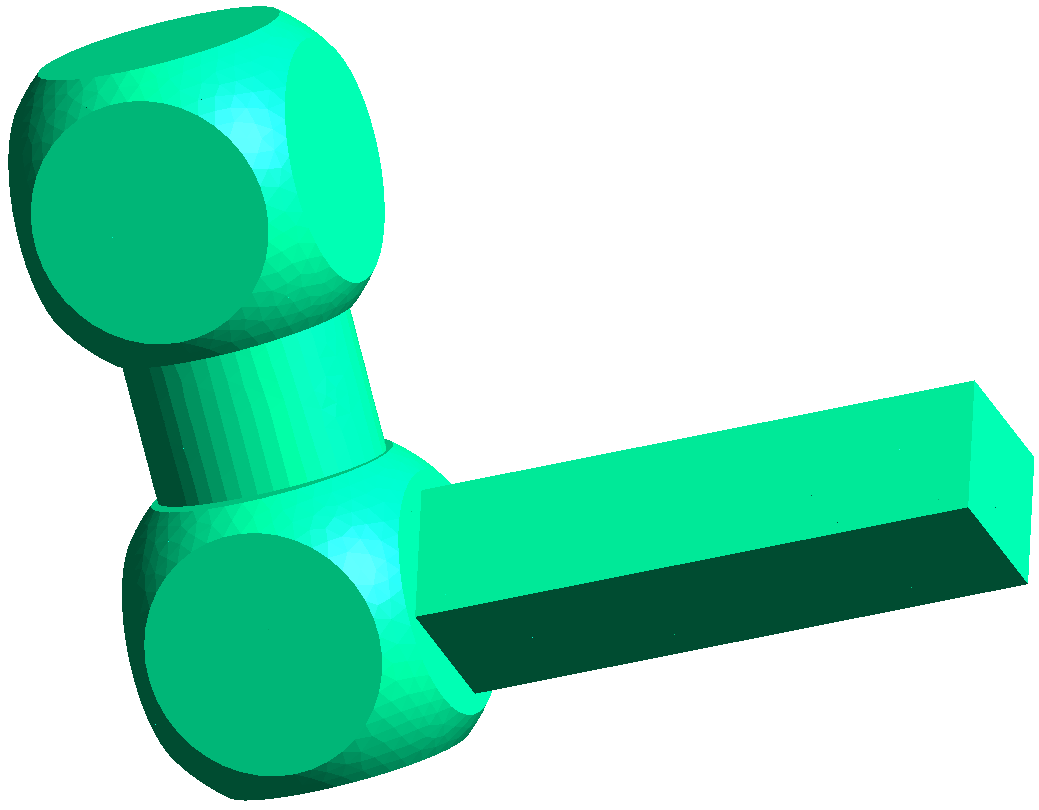}
\hspace{-13cm}\raisebox{2.3cm}{\begin{tikzpicture}
\draw[->,thick] (-1,0)--(0,0);
\node at (-3.8,0) {\begin{tabular}{l|}Part of the resonator which\\ coincides with a cylinder.\end{tabular}};
\end{tikzpicture}\hspace{2.5cm}}
\caption{Domain $\Om$ for which there is no complete separation of variables in the resonator. \label{GeomComplex}}
\end{figure}

\section{Waveguides without separation of variables}\label{Section6Legs}
In the previous two sections, we exploited in a crucial way separation of variables to prove our results of existence of eigenvalues. Here we consider other geometries where we have to proceed differently. We will work again with the min-max principle but this time the test fields will not be compactly supported. 

\subsection{Existence of discrete spectrum for resonators large enough}\label{ParaBigRes}

In \S\ref{AbsenceOfMono}, we illustrated the absence of monotonicity of the spectrum of $A$ with respect to the geometric inclusion. Here we show however that if the resonator $\mathcal{R}$ of $\Om$ is large enough, then for sure $A$ admits a non-empty discrete spectrum. To proceed, we adapt an idea proposed by J. Rohleder in \cite{Rohl24}. In this work, he demonstrates in particular that if $\Om$ is a bounded domain of $\R^3$ with Lipschitz boundary, then Maxwell's operator has at least two eigenvalues strictly below the first eigenvalue of the Dirichlet Laplacian in $\Om$. We establish here a similar result in waveguides. Since we have monotonicity of the spectrum of the Dirichlet Laplacian with respect to the geometry\footnote{Meaning that $\Om_1\subset\Om_2$ implies $\inf \sigma(-\Delta_D(\Om_2) \le \inf \sigma(-\Delta_D(\Om_1))$, where $-\Delta_D(\Om_i)$ is the Dirichlet Laplacian in $\Om_i$, $i=1,2$.}, this allows us to obtain a criterion of existence of discrete spectrum for Maxwell's operator in a broad class of waveguides.\\
\newline 
Consider $\Om$ as in (\ref{HypoOm}). Recall that we denote by $\lambda_N$ the first positive eigenvalue of the Neumann Laplacian in $S$.
\begin{theorem}\label{ThmBigRes}
Assume $\Om$ as in (\ref{HypoOm}) and that the (scalar) Dirichlet Laplacian in $\Om$ has a non-empty point spectrum. Let $\lambda_D(\Om)$ stand for its smallest eigenvalue. If $\lambda_D(\Om)<\lambda_N$, then there holds $\sigma_{\mrm{d}}(A)\ne\emptyset$.
\end{theorem}
\begin{remark}
Note that since $\lambda_N$ is less than the first eigenvalue of the Dirichlet Laplacian in $S$, $\lambda_D(\Om)$ is necessarily in the discrete spectrum of $-\Delta_D(\Om)$.
\end{remark}
\begin{remark}
The first eigenvalue of the Dirichlet Laplacian in the cube $\mathcal{C}\coloneqq(0;a)^3$, $a>0$, is equal to $3\pi^2/a^2$. As a consequence, if $\Om$ is such that there is a rigid transformation $T$ such that $T(\mathcal{C})\subset\Om$ and $3\pi^2/a^2<\lambda_N$, then Theorem \ref{ThmBigRes} together with the min-max principle for the Dirichlet Laplacian guarantee that $\sigma_{\mrm{d}}(A)\ne\emptyset$.
\end{remark}
\begin{proof}
Let $\Phi\in\mH^1_0(\Om)$ be an eigenfunction of the Dirichlet Laplacian in $\Om$ associated with $\lambda_D(\Om)$, \textit{i.e.} a non-zero function such that 
\[
-\Delta \Phi= \lambda_D(\Om)\Phi\qquad\mbox{ in }\Om.
\]
Define the field $\E_p$ such that
\[
\E_p=\left(\begin{array}{c}
\Phi\\
0\\
0
\end{array}\right)-\nabla\psi
\]
where $\psi$ is the function of $\mH^1_0(\Om)$ such that 
\begin{equation}\label{DefFunPsi}
\Delta \psi=\div \left(\begin{array}{c}
\Phi\\
0\\
0
\end{array}\right)=\partial_x\Phi.
\end{equation}
Let us emphasize that the Poincar\'e inequality
\begin{equation}\label{PoincareIneq}
\int_\Om \phi^2\,d\bfx \le C_P \int_\Om|\nabla \phi|^2\,d\bfx,\qquad\forall\phi\in\mH^1_0(\Om),
\end{equation}
with $C_P>0$, guarantees that $\psi$ is well-defined. From (\ref{DefFunPsi}), we obtain
\begin{equation}\label{Identity}
\dsp\int_{\Om}|\nabla\psi|^2\,d\bfx=\dsp\int_{\Om}\Phi\partial_x\psi\,d\bfx.
\end{equation}
Using in particular that $\Phi$ vanishes on $\partial\Om$, one can check that $\E_p$ belongs to $\boldsymbol{\mX}_N(\Om)$. Moreover, by exploiting (\ref{DefFunPsi}),  (\ref{Identity}), one finds
\begin{equation}\label{CalculCurl}
\dsp\int_{\Om}|\curlvec\E_p|^2\,d\bfx = \dsp\int_{\Om}(\partial_y\Phi)^2+(\partial_z\Phi)^2\,d\bfx = \dsp\int_{\Om}|\nabla\Phi|^2\,d\bfx - \dsp\int_{\Om}(\Delta\psi)^2\,d\bfx
\end{equation}
as well as 
\begin{equation}\label{CalculNormL2}
\dsp\int_{\Om}|\E_p|^2\,d\bfx =\dsp\int_{\Om}\Phi^2+|\nabla\psi|^2-2\Phi\partial_x\psi\,d\bfx=\dsp\int_{\Om}\Phi^2-|\nabla\psi|^2\,d\bfx.
\end{equation}
From (\ref{Identity}), using the Cauchy-Schwarz inequality, we obtain 
\[
\dsp\int_{\Om}|\nabla\psi|^2\,d\bfx \le \|\Phi\|_{\mL^2(\Om)}\|\partial_x\psi\|_{\mL^2(\Om)} <\|\Phi\|_{\mL^2(\Om)}\|\nabla\psi\|_{\boldsymbol{\mL}^2(\Om)}.
\]
To get the strict inequality, we exploit that $\partial_z\psi\not\equiv0$ otherwise $\psi$ would be null. This gives 
\[
\dsp\int_{\Om}|\nabla\psi|^2\,d\bfx < \dsp\int_{\Om}\Phi^2 \,d\bfx
\]
and guarantees that $\E_p$ is not identically zero. Define the space $\mH^1_0(\Delta;\Om)\coloneqq\{\zeta\in\mH^1_0(\Om)\,|\,\Delta\zeta\in\mL^2(\Om)\}$. It is clear that $\psi$ belongs to $\mH^1_0(\Delta;\Om)$. Below, we shall prove the following statement
\begin{lemma}\label{EstimMin}
We have 
\[
\int_{\Om}(\Delta\zeta)^2\,d\bfx \ge \lambda_D(\Om)\int_{\Om}|\nabla\zeta|^2\,d\bfx\qquad\mbox{ for all }\zeta\in\mH^1_0(\Delta;\Om).
\]
\end{lemma}
\noindent Using Lemma \ref{EstimMin}, from (\ref{CalculCurl}), (\ref{CalculNormL2}), we deduce
\[
\dsp\int_{\Om}|\curlvec\E_p|^2\,d\bfx \le \lambda_D(\Om) \dsp\int_{\Om}\Phi^2-|\nabla\psi|^2\,d\bfx = \lambda_D(\Om)
\dsp\int_{\Om}|\E_p|^2\,d\bfx.
\]
Since by assumption $\lambda_D(\Om)<\lambda_N$, the min-max principle ensures that $A$ has a non-empty discrete spectrum.
\end{proof}
\begin{proof}[Proof of Lemma \ref{EstimMin}. ] The inverse of $-\Delta_D(\Om)$ is 
a bounded and self-adjoint operator of $\mL^2(\Om)$. Since $\lambda_D(\Om)>0$ is the minimum of the spectrum of $-\Delta_D(\Om)$, one has 
\[
\cfrac{1}{\lambda_D(\Om)}=\sup_{f\in\mL^2(\Om)\setminus\{0\}} \cfrac{\dsp\int_{\Om} (-\Delta_D(\Om)^{-1}f)\,f\,d\bfx}{\dsp\int_{\Om}f^2\,d\bfx}=\sup_{\zeta\in\mH^1_0(\Delta;\Om)\setminus\{0\}} \cfrac{\dsp\int_{\Om}|\nabla\zeta|^2\,d\bfx}{\dsp\int_{\Om}(\Delta\zeta)^2\,d\bfx}\,.
\]  
\end{proof}
\noindent Note that as in the previous section, the statement of Theorem \ref{ThmBigRes} is also valid in domains which are the union of a resonator and several semi-infinite waveguides.

\subsection{The 6 legs geometry }\label{Para6legs}

Now let us consider a more complex geometry, which can be of interest for practical applications, where separation of variables cannot be used directly. 

\subsubsection*{Setting}

\begin{figure}[!ht]
\centering
\includegraphics[width=4.9cm]{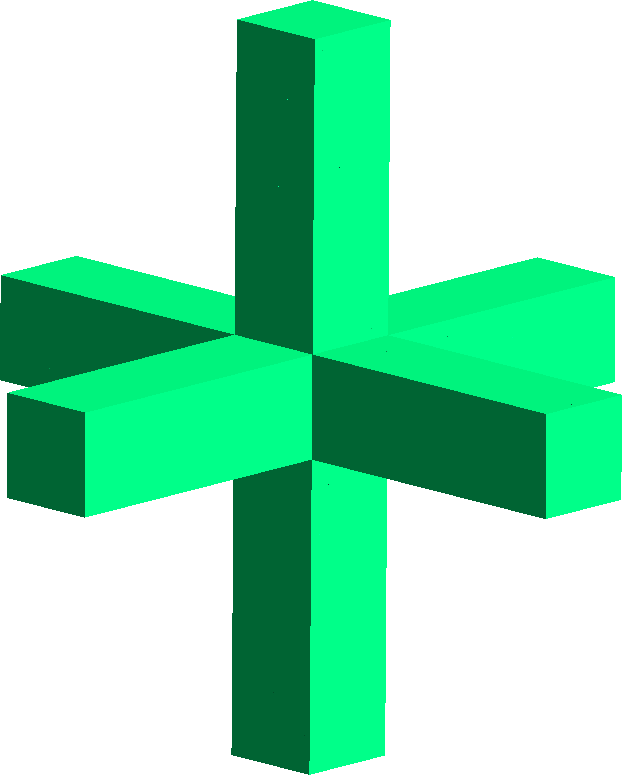}\quad
\begin{tikzpicture}[scale=0.94]
\draw[fill=gray!30,draw=none](-2.8,-0.5) rectangle (2.8,0.5);
\draw[fill=gray!30,draw=none](-0.5,-2.8) rectangle (0.5,2.8);
\draw[black] (0.5,2.8)--(0.5,0.5)--(2.8,0.5);
\draw[black] (-0.5,2.8)--(-0.5,0.5)--(-2.8,0.5);
\draw[black] (0.5,-2.8)--(0.5,-0.5)--(2.8,-0.5);
\draw[black] (-0.5,-2.8)--(-0.5,-0.5)--(-2.8,-0.5);
\draw[black,dashed] (0.5,-0.5)--(0.5,0.5);
\draw[black,dashed] (-0.5,-0.5)--(-0.5,0.5);
\draw[black,dashed] (-0.5,0.5)--(0.5,0.5);
\draw[black,dashed] (0.5,-0.5)--(-0.5,-0.5);
\node at (1.4,0) {$\GX_x^+$};
\node at (-1.4,0) {$\GX_x^-$};
\node at (0,1.4) {$\GX_y^+$};
\node at (0,-1.4) {$\GX_y^-$};
\node at (0,0) {$\square$};
\end{tikzpicture}\quad\raisebox{0.0cm}{\includegraphics[width=5.3cm]{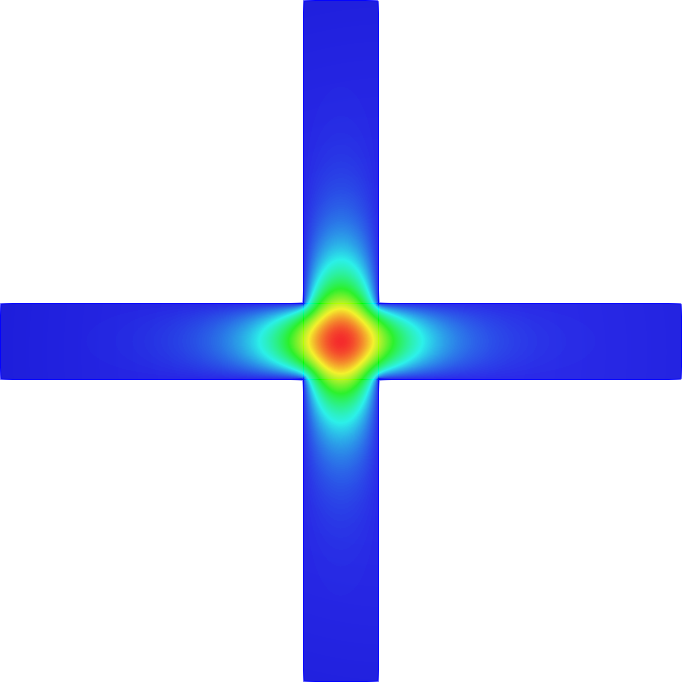}}
\caption{The 6 legs geometry $\Om\subset\R^3$ (left), the domain $\GX\subset\R^2$ (center) and an eigenfunction associated with the first eigenvalue of the Dirichlet Laplacian in $\GX$. \label{GeomBug}}
\end{figure}

Define 
\[
\begin{array}{l}
\Pi^+_x\coloneqq[1/2;+\infty)\times(-1/2;1/2)\times(-1/2;1/2),\qquad \Pi^-_x\coloneqq(-\infty;-1/2]\times(-1/2;1/2)\times(-1/2;1/2)\\[2pt]
\Pi^+_y\coloneqq(-1/2;1/2)\times[1/2;+\infty)\times(-1/2;1/2),\qquad \Pi^-_y\coloneqq(-1/2;1/2)\times(-\infty;-1/2]\times(-1/2;1/2)\\[2pt]
\Pi^+_z\coloneqq(-1/2;1/2)\times(-1/2;1/2)\times[1/2;+\infty),\qquad \Pi^-_z\coloneqq(-1/2;1/2)\times(-1/2;1/2)\times(-\infty;-1/2]\\[2pt]
\mathcal{C}\coloneqq(-1/2;1/2)^3
\end{array}
\]
and consider the waveguide $\Om$ represented in Figure \ref{GeomBug} left such that
\[
\Om\coloneqq\Pi^+_x\cup\Pi^+_y\cup\Pi^+_z\cup\Pi^-_x\cup\Pi^-_y\cup\Pi^-_z\cup\mathcal{C}.
\]
We wish to prove that the discrete spectrum of $A$ in $\Om$ is not empty. To proceed, we will exploit that the Dirichlet Laplacian in the 2D X-shaped domain 
\[
\GX\coloneqq\{(y,z)\in\R^2\,|\,|y|<1/2\mbox{ or }|z|<1/2\}
\]
(see Figure \ref{GeomBug} center) admits an eigenvalue $\lambda_\bullet<\pi^2$ in its discrete spectrum (see \cite{ScRW89,Naza11,Naza14b}). Consider $\varphi\in\mH^1_0(\GX)$ a corresponding eigenfunction, \textit{i.e.} a function such that 
\begin{equation}\label{Spectral2DX}
-\Delta\varphi=\lambda_\bullet\varphi\quad\mbox{ in }\GX
\end{equation}
(see Figure \ref{GeomBug} right for a numerical approximation). Classically, with the Krein-Rutman theorem, one shows that $\varphi$ can be chosen such that $\varphi>0$ in $\GX$. Moreover, one proves that $\varphi$ satisfies $\varphi(y,z)=\varphi(-y,z)$, $\varphi(y,z)=\varphi(y,-z)$ and $\varphi(y,z)=\varphi(z,y)$ for $(y,z)\in \GX$. Solving numerically the 2D scalar problem (\ref{Spectral2DX}), with \texttt{Freefem++}, one obtains 
\begin{equation}\label{Vp2DX}
\lambda_\bullet\approx 6.5186 \approx 0.6605\pi^2.
\end{equation}

\subsubsection*{Existence of discrete spectrum}

To show existence of discrete spectrum for $A$, we work with the min-max principle (\ref{MaxMinFormula}). Since the section of the six branches constituting $\Om$ is a square of side one, Proposition \ref{DefSigmaEss} ensures that $\sigma_{\mrm{ess}}(A)=[\pi^2;+\infty)$. Therefore we have to construct some $\E\in\boldsymbol{\mX}_N(\Om)$ such that 
\begin{equation}\label{QtyToEstimateX}
\int_{\Om} |\curlvec\E|^2\,d\bfx-\pi^2\int_{\Om} |\E|^2\,d\bfx<0.
\end{equation}
Define $\E_p$ such that 
\begin{equation}\label{DefEpX}
\E_p\coloneqq\left(\begin{array}{c}
E_x\\
0\\
0
\end{array}\right)
\end{equation}
with
\[
E_x(\bfx)=\begin{array}{|ll}
\varphi(y,z)& \mbox{for }|x|<1/2\\
0& \mbox{for }|x|>1/2.
\end{array}
\]
We have $\E_{p}\in \boldsymbol{\mH}_N(\curlvec)$ as well as $\E_{p}\times\bfnu=0$ on $\partial\Om$. However $\div\,\E_{p}\not\equiv0$ in $\Om$ because the normal component of $\E_{p}$ is discontinuous at $x=\pm1/2$. To obtain a divergence free vector field, introduce $\psi\in\mH^1_0(\Om)$ such that
\begin{equation}\label{PbDefPsi}
\int_\Om\nabla\psi\cdot\nabla\psi'\,d\bfx=\int_\Om\E_{p}\cdot\nabla\psi'\,d\bfx=\int_\Om E_x\partial_x\psi'\,d\bfx=\int_{\{\bfx\in\Om\,|\,|x|<1/2\}} E_x\partial_x\psi'\,d\bfx,\quad \forall\psi'\in\mH^1_0(\Om).
\end{equation}
Note that the Poincar\'e inequality (\ref{PoincareIneq}) guarantees that $\psi$ is well-defined. Moreover, observe that we have $\psi(x,y,z)=-\psi(-x,y,z)$ in $\Om$.  Set $\tilde{\E}_{p}\coloneqq\E_{p}-\nabla\psi$. This $\tilde{\E}_{p}$ belongs to $\boldsymbol{\mX}_N(\Om)$. Now our goal is to establish that $\tilde{\E}_p$ satisfies (\ref{QtyToEstimateX}).\\
\newline
A direct calculation gives 
\begin{equation}\label{EstimAboveX}
\dsp\int_{\Om}|\curlvec \tilde{\E}_p|^2\,d\bfx=\int_0^1\int_\GX|\nabla\varphi|^2\,dxdydz=\int_\GX|\nabla\varphi|^2\,dydz=\lambda_\bullet\int_\GX\varphi^2\,dydz.
\end{equation}
\noindent In (\ref{QtyToEstimateX}), it remains to estimate $\int_{\Om} |\tilde{\E}_{p}|^2\,d\bfx$. By exploiting that 
\begin{equation}\label{TermNRJ}
\int_{\Om}|\nabla\psi|^2\,d\bfx=\int_{\Om}\E_p\cdot\nabla\psi\,d\bfx,
\end{equation}
we get
\begin{equation}\label{Rel2X}
\int_{\Om} |\tilde{\E}_{p}|^2\,d\bfx=\int_{\Om} |\E_{p}|^2\,d\bfx-\int_{\Om} |\nabla\psi|^2\,d\bfx =\dsp\int_\GX\varphi^2\,dydz-\int_{\Om} |\nabla\psi|^2\,d\bfx.
\end{equation}
Let us control the term $\int_{\Om} |\nabla\psi|^2\,d\bfx$ by $\int_\GX\varphi^2\,dydz$. Define $\Sigma^\pm_x\coloneqq\{\pm1/2\}\times(-1/2;1/2)\times(-1/2;1/2)$. Integrating by parts in the right-hand side of (\ref{TermNRJ}) and using that $\div\,\E_p=0$ in $\Om\setminus(\Sigma^{+}_x\cup\Sigma^{-}_x)$, we find
\[
\dsp\int_\Om|\nabla\psi|^2\,d\bfx=\dsp\sum_\pm\pm\int_{\Sigma^\pm_x}\varphi\,\psi(\pm1/2,\cdot)\,dydz.
\]
Writing
\[
\bigg|\int_{\Sigma^\pm_x}\varphi\,\psi(\pm1/2,\cdot)\,dydz\bigg|\le \|\varphi\|_{\mL^2(\square)}\|\psi\|_{\mL^2(\Sigma^\pm_x)}
\]
with $\square=(-1/2;1/2)^2$, we obtain 
\begin{equation}\label{EstimDepX}
\int_\Om|\nabla\psi|^2\,d\bfx \le \sqrt{2}\|\varphi\|_{\mL^2(\square)}\big(\sum_\pm\|\psi\|^2_{\mL^2(\Sigma^\pm_x)}\big)^{1/2}.
\end{equation}
\begin{lemma}\label{LemmaProof6Legs1}
For the function $\psi$ introduced in (\ref{PbDefPsi}), we have the estimate
\begin{equation}\label{EstimEtaQuaX}
\sum_\pm\|\psi\|^2_{\mL^2(\Sigma^\pm_x)} \le \cfrac{1}{\sqrt{2}\pi+\sqrt{2\kappa(\sqrt{5\pi^2})}} \int_{\Om}|\nabla\psi|^2\,d\bfx
\end{equation}
where the constant $\kappa(\sqrt{5\pi^2})$ is defined in (\ref{TransEqnLemmaX}).
\end{lemma}
\noindent We postpone the proof of this result to the end of the paragraph. From (\ref{EstimDepX}), (\ref{EstimEtaQuaX}), we derive 
\begin{equation}\label{EstimSuiteX}
\int_\Om|\nabla\psi|^2\,d\bfx \le \cfrac{2}{\sqrt{2}\pi+\sqrt{2\kappa(\sqrt{5\pi^2})}}\ \|\varphi\|^2_{\mL^2(\square)}.
\end{equation}
By gathering (\ref{EstimAboveX}), (\ref{Rel2X}) and (\ref{EstimSuiteX}), we can write
\begin{equation}\label{EstimaInterX}
\begin{array}{rcl}
\dsp\int_{\Om} |\curlvec\tilde{\E}_{p}|^2\,d\bfx-\pi^2\int_{\Om} |\tilde{\E}_{p}|^2\,d\bfx &\hspace{-0.2cm}=&\hspace{-0.2cm}(\lambda_\bullet-\pi^2)\dsp\int_\GX\varphi^2\,dydz+\pi^2\int_{\Om} |\nabla\psi|^2\,d\bfx\\[10pt]
&\hspace{-0.2cm}\le&\hspace{-0.2cm}(\lambda_\bullet-\pi^2)\dsp\int_\GX\varphi^2\,dydz+\cfrac{2\pi^2}{\sqrt{2}\pi+\sqrt{2\kappa(\sqrt{5\pi^2})}}\,\dsp\int_{\square}\varphi^2\,dydz.
\end{array}\hspace{-0.2cm}
\end{equation}
\begin{lemma}\label{LemmaProof6Legs2}
For the function $\varphi$ introduced in (\ref{Spectral2DX}), we have the estimate
\begin{equation}\label{Lemma2DPF}
\dsp\int_{\square}\varphi^2\,dydz \le \cfrac{1}{2\kappa(\pi)}\int_{\GX}|\nabla\varphi|^2\,dydz
\end{equation}
where the constant $\kappa(\pi)$ is defined in (\ref{TransEqnLemmaX}).
\end{lemma}
\noindent Again we give the proof of this result below. Inserting (\ref{Lemma2DPF}) in (\ref{EstimaInterX}) and using that $\varphi$ satisfies (\ref{Spectral2DX}), we obtain 
\[
\dsp\int_{\Om} |\curlvec\tilde{\E}_{p}|^2\,d\bfx-\pi^2\int_{\Om} |\tilde{\E}_{p}|^2\,d\bfx \le  \bigg(\lambda_\bullet-\pi^2+\cfrac{\pi^2\lambda_\bullet}{\kappa(\pi)(\sqrt{2}\pi+\sqrt{2\kappa(\sqrt{5\pi^2})})}\ \bigg)\dsp\int_\GX\varphi^2\,dydz.
\]
By solving (\ref{TransEqnLemmaX}), we get $\kappa(\pi)\approx4.0214$, $\kappa(\sqrt{5\pi^2})\approx6.0827$ and so
\[
\lambda_\bullet-\pi^2+\cfrac{\pi^2\lambda_\bullet}{\kappa(\pi)(\sqrt{2}\pi+\sqrt{2\kappa(\sqrt{5\pi^2})})} \approx -1.0355.
\]
This shows that $\tilde{\E}_{p}$ satisfies (\ref{QtyToEstimateX}) and that the discrete spectrum of $A$ has at least one eigenvalue.  Admittedly, this is not a completely rigorous proof because at the end we exploit the value of $\lambda_\bullet$ computed numerically with a finite element method. This is the reason why we do not write a proposition. However $\lambda_\bullet$ has been obtained with a rather refined mesh and error estimates can be established concerning its computation. Moreover, from the min-max principle, we know that the numerical approximations we get are upper bounds for $\lambda_\bullet$ if we neglect the rounding errors made by the computer. As a consequence, we are convinced that $\sigma_{\mrm{d}}(A)\ne\emptyset$.
\begin{remark}\label{RmkAppThm}
Numerically we find that the Dirichlet Laplacian in $\Om$ has 
a non-empty discrete spectrum with a smallest eigenvalue approximately equal to 12.72. Since it is larger than $\pi^2$, we could not exploit Theorem \ref{ThmBigRes} to deduce directly that $\sigma_{\mrm{d}}(A)\ne\emptyset$.
\end{remark}
\noindent We end this paragraph by establishing the two technical results used above. \\
 
\noindent\textit{Proof of Lemma \ref{LemmaProof6Legs1}.}
First, we can write
\[
\begin{array}{rcl}
\dsp\|\psi\|_{\mL^2(\Sigma^\pm_x)}^2 =\dsp\int_{-1/2}^{1/2}\int_{-1/2}^{1/2}\psi^2(\pm1/2,y,z)\,dydz&=&\mp\dsp\int_{\Pi^\pm_x}2\psi\partial_x\psi\,d\bfx \\[8pt]
&\le&\sqrt{2\pi^2}\dsp\int_{\Pi^\pm_x}\psi^2\,d\bfx+\cfrac{1}{\sqrt{2\pi^2}}\dsp\int_{\Pi^\pm_x}(\partial_x\psi)^2\,d\bfx.
\end{array}
\]
From the Poincar\'e inequality
\[
2\pi^2\dsp\int_{\Pi^\pm_x}\psi^2\,d\bfx \le \dsp\int_{\Pi^\pm_x}(\partial_y\psi)^2+(\partial_z\psi)^2\,d\bfx,
\]
we deduce 
\[
\|\psi\|_{\mL^2(\Sigma^\pm_x)}^2 \le \cfrac{1}{\sqrt{2\pi^2}}\int_{\Pi^\pm_x}|\nabla\psi|^2\,d\bfx.
\]
Thus we obtain 
\begin{equation}\label{EstimT1bisX}
\sum_\pm\|\psi\|^2_{\mL^2(\Sigma^\pm_x)} \le \cfrac{1}{\sqrt{2\pi^2}}\int_{\Pi^-_x\cup\Pi^+_x}|\nabla\psi|^2\,d\bfx.
\end{equation}
At this point, it could have been tempting to control the right-hand side of (\ref{EstimT1bisX}) by $\int_{\Om}|\nabla\psi|^2\,d\bfx$. However this is a too crude estimate which is not sufficient for our needs (see (\ref{EstimSuiteX}) and below), we have to work more carefully. To proceed, we will also exploit the term $\int_{\mathcal{C}\cup\Pi^-_y\cup\Pi^+_y\cup\Pi^-_z\cup\Pi^+_z}|\nabla\psi|^2\,d\bfx$. First, recall that we have $\psi=0$ at $x=0$ because $\psi(x,y,z)=-\psi(-x,y,z)$ in $\Om$. This allows us to write, for all $\alpha>0$,
\begin{equation}\label{BiduleAlpha}
\|\psi\|_{\mL^2(\Sigma^+_x)}^2 =\dsp\int_{-1/2}^{1/2}\int_{-1/2}^{1/2}\psi^2(1/2,y,z)\,dydz\le\alpha\dsp\int_{\mathcal{C}^{\circ}}\psi^2\,d\bfx+\alpha^{-1}\dsp\int_{\mathcal{C}^{\circ}}(\partial_x\psi)^2\,d\bfx
\end{equation}
where $\mathcal{C}^{\circ}\coloneqq(0;1/2)\times(-1/2;1/2)^2$. In the Appendix we show that for $a>0$, we have the Poincar\'e-Friedrichs inequality
\begin{equation}\label{Fried_ineqX}
 \kappa(a)\int_{0}^{1/2}\phi^2\,dt \le \int_{0}^{+\infty}(\partial_t\phi)^2\,dt+a^2\int_{1/2}^{+\infty}\phi^2\,dt,\qquad \forall \phi\in\mH^1(0;+\infty),
\end{equation}
where $\kappa(a)$ is the smallest positive root of the transcendental equation
\begin{equation}\label{TransEqnLemmaX}
\sqrt{\kappa}\tan\bigg(\cfrac{\sqrt{\kappa}}{2}\bigg)=a.
\end{equation}
By taking $a=\sqrt{5\pi^2}$ in (\ref{Fried_ineqX}) with $t$ replaced by the variable $y$, by exploiting a similar estimate in $(-\infty;0)$, after integration with respect to $(x,z)\in(0;1/2)\times(-1/2;1/2)$, we obtain 
\begin{equation}\label{PoincareFried1}
\kappa(\sqrt{5\pi^2})\dsp\int_{\mathcal{C}^{\circ}}\psi^2\,d\bfx \le \int_{\mathcal{C}^{\circ}\cup\Pi^{\circ}_y}(\partial_y\psi)^2\,d\bfx+5\pi^2\int_{\Pi^{\circ}_y}\psi^2\,d\bfx.
\end{equation}
Here $\Pi^{\circ}_y\coloneqq\{\bfx\in\Pi^+_y\cup\Pi^-_y\,|\,x>0\}$. Since the section of $\Pi^{\circ}_y$ is a rectangle of width $1/2$ and length $1$, we have the Poincar\'e inequality
\begin{equation}\label{PoincareRect}
5\pi^2\int_{\Pi^{\circ}_y}\psi^2\,d\bfx \le \int_{\Pi^{\circ}_y}(\partial_x\psi)^2+(\partial_z\psi)^2\,d\bfx.
\end{equation}
Inserting (\ref{PoincareRect}) in (\ref{PoincareFried1}) gives 
\[
\kappa(\sqrt{5\pi^2})\dsp\int_{\mathcal{C}^{\circ}}\psi^2\,d\bfx \le \int_{\mathcal{C}^{\circ}}(\partial_y\psi)^2\,d\bfx+\int_{\Pi^{\circ}_y}|\nabla\psi|^2\,d\bfx.
\]
Similarly, we establish the estimate
\[
\kappa(\sqrt{5\pi^2})\dsp\int_{\mathcal{C}^{\circ}}\psi^2\,d\bfx \le \int_{\mathcal{C}^{\circ}}(\partial_z\psi)^2\,d\bfx+\int_{\Pi^{\circ}_z}|\nabla\psi|^2\,d\bfx
\]
where $\Pi^{\circ}_z\coloneqq\{\bfx\in\Pi^+_z\cup\Pi^-_z\,|\,x>0\}$. By taking $\alpha=\sqrt{2\kappa(\sqrt{5\pi^2})}$ in (\ref{BiduleAlpha}), we find
\[
\|\psi\|_{\mL^2(\Sigma^+_x)}^2 \le \cfrac{1}{\sqrt{2\kappa(\sqrt{5\pi^2})}}\int_{\mathcal{C}^{\circ}\cup\Pi^{\circ}_y\cup\Pi^{\circ}_z}|\nabla\psi|^2\,d\bfx.
\]
Working similarly to estimate $\|\psi\|^2_{\mL^2(\Sigma^-_x)}$, we obtain
\begin{equation}\label{EstimEtaTerX}
\sum_\pm\|\psi\|_{\mL^2(\Sigma^\pm_x)}^2 \le \cfrac{1}{\sqrt{2\kappa(\sqrt{5\pi^2})}}\int_{\mathcal{C}\cup\Pi^+_y\cup\Pi^+_z\cup\Pi^-_y\cup\Pi^-_z}|\nabla\psi|^2\,d\bfx.
\end{equation}
By multiplying (\ref{EstimT1bisX}), (\ref{EstimEtaTerX}) respectively by $\sqrt{2\pi^2}$, $\sqrt{2\kappa(\sqrt{5\pi^2})}$ and summing the two resulting inequalities, we deduce (\ref{EstimEtaQuaX}).\hfill $\square$\\
\newline
\noindent \textit{Proof of Lemma \ref{LemmaProof6Legs2}.} Using the Poincar\'e-Friedrichs inequality (\ref{Fried_ineqX}) with $a=\pi$ and $t$ replaced by the variable $y$, after integration with respect to $z\in(-1/2;1/2)$, we find
\[
\kappa(\pi)\int_{\square^+_y}\varphi^2\,dydz \le \int_{\square^+_y\cup\GX^+_y}(\partial_y\varphi)^2\,dydz+\pi^2 \int_{\GX^+_y}\varphi^2\,dydz.
\]
Here $\square^+_y\coloneqq(0;1/2)\times(-1/2;1/2)$ and $\GX^+_y\coloneqq(1/2;+\infty)\times(-1/2;1/2)$. But we have the Poincar\'e inequality
\[
\pi^2 \int_{\GX^+_y}\varphi^2\,dydz \le  \int_{\GX^+_y}(\partial_z\varphi)^2\,dydz. 
\]
Thus, there holds
\[
\kappa(\pi)\int_{\square^+_y}\varphi^2\,dydz \le \int_{\square^+_y}(\partial_y\varphi)^2\,dydz+ \int_{\GX^+_y}|\nabla\varphi|^2\,dydz.
\]
By establishing a similar estimate for $y<0$ and by summing the two, we get
\begin{equation}\label{EstimPF1}
\kappa(\pi)\int_{\square}\varphi^2\,dydz \le \int_{\square}(\partial_y\varphi)^2\,dydz+ \int_{\GX^+_y\cup\GX^-_y}|\nabla\varphi|^2\,dydz,
\end{equation}
where $\GX^-_y\coloneqq(-\infty;-1/2)\times(-1/2;1/2)$. In the same way, one proves 
\begin{equation}\label{EstimPF2}
\kappa(\pi)\int_{\square}\varphi^2\,dydz \le \int_{\square}(\partial_z\varphi)^2\,dydz+ \int_{\GX^+_z\cup\GX^-_z}|\nabla\varphi|^2\,dydz,
\end{equation}
with $\GX^+_z\coloneqq(-1/2;1/2)\times(1/2;+\infty)$, $\GX^-_z\coloneqq(-1/2;1/2)\times(-\infty;-1/2)$. Adding (\ref{EstimPF1}) and (\ref{EstimPF2}) leads to (\ref{Lemma2DPF}). \hfill $\square$

\subsection{The tripode}

Let us consider a second 3D geometry with branches having square sections. A bit unexpectedly, its study reveals more complicated than that of the waveguide met in the previous paragraph. Often we will use similar notation corresponding though to different objects.

\subsubsection*{Setting}

\begin{figure}[!ht]
\centering
\includegraphics[width=5.4cm]{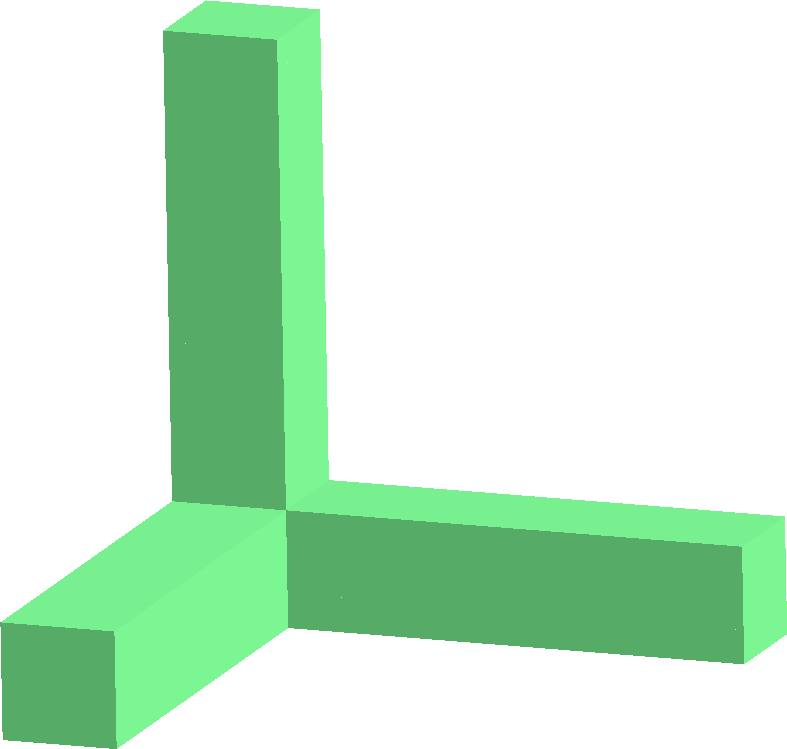}\qquad
\begin{tikzpicture}[scale=1.4]
\draw[fill=gray!30,draw=none](0,0) rectangle (2.8,1);
\draw[fill=gray!30,draw=none](0,0) rectangle (1,2.8);
\draw[black] (0,2.8)--(0,0)--(2.8,0);
\draw[black] (1,2.8)--(1,1)--(2.8,1);
\draw[black,dotted] (0,2.8)--(0,3.2);
\draw[black,dotted] (1,2.8)--(1,3.2);
\draw[black,dotted] (2.8,0)--(3.2,0);
\draw[black,dotted] (2.8,1)--(3.2,1);
\draw[black,dashed] (1,0)--(1,1);
\draw[black,dashed] (0,1)--(1,1);
\node at (2,0.45) {$\GL_+$};
\node at (0.5,0.45) {$\square$};
\node at (1,-0.2) {$1$};
\end{tikzpicture}
\quad\raisebox{0.0cm}{\includegraphics[width=5cm]{TrappedModeL}}
\caption{The tripode $\Om\subset\R^3$ (left), the domain $\GL\subset\R^2$ (center) and an eigenfunction associated with the first eigenvalue of the Dirichlet Laplacian in $\GL$. \label{Tripode}}
\end{figure}

\noindent Define
\[
\begin{array}{l}
\Pi_x\coloneqq[1;+\infty)\times(0;1)\times(0;1)\\[2pt]
\Pi_y\coloneqq(0;1)\times[1;+\infty)\times(0;1)\\[2pt]
\Pi_z\coloneqq(0;1)\times(0;1)\times[1;+\infty)\\[2pt]
\mathcal{C}\coloneqq(0;1)^3
\end{array}
\]
and consider the tripode waveguide $\Om$ represented in Figure \ref{Tripode} left such that
\begin{equation}\label{DefTripode}
\Om\coloneqq\Pi_x\cup\Pi_y\cup\Pi_z\cup\mathcal{C}.
\end{equation}
We want to show that the discrete spectrum of $A$ in $\Om$ is not empty. To proceed, we will exploit that the Dirichlet Laplacian in the 2D L-shaped domain (already met in (\ref{Lshape2D}))
\[
\GL\coloneqq\{(x,y)\in(0;+\infty)^2\,|\,x<1\mbox{ or }y<1\}=(0;1)\times(0;+\infty)\cup(0;+\infty)\times(0;1)
\]
(see Figure \ref{Tripode} center) admits an eigenvalue $\lambda_\bullet$ in its discrete spectrum (see \cite{ExSt89,ABGM91},\cite[Prop.\,4.1]{DaLR12}). Consider $\varphi\in\mH^1_0(\GL)$ a corresponding eigenfunction, \textit{i.e.} a function such that 
\begin{equation}\label{Spectral2D}
-\Delta\varphi=\lambda_\bullet\varphi\quad\mbox{ in }\GL
\end{equation}
(see Figure \ref{Tripode} right for a representation). With the Krein-Rutman theorem, one proves that $\varphi$ can be chosen such that $\varphi>0$ in $\GL$. Moreover, $\varphi$ satisfies $\varphi(x,y)=\varphi(y,x)$ for all $(x,y)\in \GL$. As said after (\ref{Lshape2D}), numerically one obtains $\lambda_\bullet\approx 9.1722 \approx 0.9293\pi^2$. Observe that this value is much closer to the threshold of the essential spectrum ($\pi^2$) than in the X-shaped geometry (see (\ref{Vp2DX})). This might be the reason why we need estimates that are sharper for the tripode than for the 6 legs geometry of \S\ref{Para6legs}.\\
\newline
Since the section of the three branches of $\Om$ is a square of side one, we have $\sigma_{\mrm{ess}}(A)=[\pi^2;+\infty)$. As a consequence, we have to exhibit some $\E\in\boldsymbol{\mX}_N(\Om)$ such that 
\[
\int_{\Om} |\curlvec\E|^2\,d\bfx-\pi^2\int_{\Om} |\E|^2\,d\bfx<0.
\]

\subsubsection*{Construction of the test field}

Define $\E_p$ such that 
\[
\E_p\coloneqq\left(\begin{array}{c}
E_x\\
E_y\\
E_z
\end{array}\right)
\]
with
\[
E_x(\bfx)=\begin{array}{|ll}
\varphi(y,z)& \mbox{ for }x<1\\
0& \mbox{ for }x>1\\
\end{array},\quad E_y(\bfx)=\begin{array}{|ll}
\varphi(x,z)& \mbox{ for }y<1\\
0& \mbox{ for }y>1\\
\end{array},\quad E_z(\bfx)=\begin{array}{|ll}
\varphi(x,y)& \mbox{ for }z<1\\
0& \mbox{ for }z>1.
\end{array}
\]
Note that contrary to what was done in the previous paragraph, we start with a field which has some symmetries and which is non-zero in all branches. Let us emphasize that we have not been able to conclude with some $\E_p$ simply defined as in (\ref{DefEpX}).\\
There holds $\E_{p}\in \boldsymbol{\mH}_N(\curlvec)$ as well as $\E_{p}\times\bfnu=0$ on $\partial\Om$. However, we have $\div\,\E_{p}\not\equiv0$ in $\Om$. Therefore we introduce $\psi\in\mH^1_0(\Om)$ such that
\begin{equation}\label{PbDefPsiBis}
\int_\Om\nabla\psi\cdot\nabla\psi'\,d\bfx=\int_\Om\E_{p}\cdot\nabla\psi'\,d\bfx,\qquad \forall\psi'\in\mH^1_0(\Om),
\end{equation}
and we set $\tilde{\E}_{p}\coloneqq\E_{p}-\nabla\psi$. This $\tilde{\E}_{p}$ belongs to $\boldsymbol{\mX}_N(\Om)$. 

\subsubsection*{First identities for the test field}
Now, our goal is to establish that the quantity
\begin{equation}\label{QtyToEstimate}
\int_{\Om} |\curlvec\tilde{\E}_{p}|^2\,d\bfx-\pi^2\int_{\Om} |\tilde{\E}_{p}|^2\,d\bfx
\end{equation}
is negative. To proceed, we start with the following statement.
\begin{lemma}
There holds
\begin{equation}\label{EstimAbove}
\dsp\int_{\Om}|\curlvec \tilde{\E}_p|^2\,d\bfx=3\lambda_\bullet\int_\GL\varphi^2\,dxdy-6\int_0^1\varphi(1,y)^2\,dy.
\end{equation}
\end{lemma}
\begin{proof}
In $\Om$, we have $\curlvec \tilde{\E}_p=\curlvec \E_p$. Moreover, in $\Pi_x$ we find
\[
\curlvec \E_p=\curlvec \left(\begin{array}{c}
0\\
E_y\\
E_z
\end{array}\right)=\left(\begin{array}{c}
\partial_yE_z-\partial_zE_y\\
-\partial_xE_z\\
\partial_xE_y
\end{array}\right).
\]
Thus we obtain
\[
\begin{array}{rcl}
\dsp\int_{\Pi_x}|\curlvec \E_p|^2\,d\bfx&=&\dsp\int_{\Pi_x}(\partial_xE_y)^2+(\partial_zE_y)^2+(\partial_xE_z)^2+(\partial_yE_z)^2-2\partial_yE_z\partial_zE_y\,d\bfx.
\end{array}
\]
By observing that 
\[
\dsp\int_{\Pi_x}\partial_yE_z\partial_zE_y\,d\bfx=\int_1^{+\infty}\int_0^1 \partial_yE_z(x,y)\,dy\int_0^1\partial_zE_y(x,z)\,dzdx=0,
\]
we get 
\[
\dsp\int_{\Pi_x}|\curlvec \E_p|^2\,d\bfx=2\int_{\GL_+}|
\nabla\varphi|^2\,dxdy
\]
where $\GL_+\coloneqq(1;+\infty)\times(0;1)$. Similarly, we obtain
\[
\dsp\int_{\Pi_y}|\curlvec \E_p|^2\,d\bfx=\dsp\int_{\Pi_z}|\curlvec \E_p|^2\,d\bfx=2\int_{\GL_+}|
\nabla\varphi|^2\,dxdy.
\]
On the other hand, in $\mathcal{C}=(0;1)^3$, we have
\[
\curlvec \E_p=\left(\begin{array}{c}
\partial_yE_z-\partial_zE_y\\
\partial_zE_x-\partial_xE_z\\
\partial_xE_y-\partial_yE_x
\end{array}\right).
\]
Therefore, we find
\begin{equation}\label{TermLemma}
\dsp\int_{\mathcal{C}}|\curlvec \E_p|^2\,d\bfx=3\int_{\square}|
\nabla\varphi|^2\,dxdy-2\dsp\int_{\mathcal{C}} \partial_yE_z\partial_zE_y+\partial_xE_z\partial_zE_x+\partial_xE_y\partial_yE_x\,d\bfx
\end{equation}
with $\square=(0;1)^2$. But 
\[
\dsp\int_{\mathcal{C}} \partial_yE_z\partial_zE_y\,d\bfx=\int_0^{1}\int_0^1 \partial_yE_z(x,y)\,dy\int_0^1\partial_zE_y(x,z)\,dzdx=\int_0^{1}E_z(x,1)E_y(x,1)\,dx=\int_0^1\varphi(x,1)^2\,dx.
\]
By exploiting that $\varphi(x,y)=\varphi(y,x)$ in $\GL$, similarly we get
\[
\dsp\int_{\mathcal{C}} \partial_xE_z\partial_zE_x\,d\bfx=\dsp\int_{\mathcal{C}} \partial_xE_y\partial_yE_x\,d\bfx=\int_0^1\varphi(x,1)^2\,dx=\int_0^1\varphi(1,y)^2\,dy.
\]
With (\ref{TermLemma}), we deduce (\ref{EstimAbove}).
\end{proof}
\noindent Next we estimate the term $\int_{\Om} |\tilde{\E}_{p}|^2\,d\bfx$ in  (\ref{QtyToEstimate}). From the relation
\begin{equation}\label{TermNRJBis}
\int_{\Om}|\nabla\psi|^2\,d\bfx=\int_{\Om}\E_p\cdot\nabla\psi\,d\bfx,
\end{equation}
we obtain
\begin{equation}\label{Rel2}
\int_{\Om} |\tilde{\E}_{p}|^2\,d\bfx=\int_{\Om} |\E_{p}|^2\,d\bfx-\int_{\Om} |\nabla\psi|^2\,d\bfx =3\dsp\int_\GL\varphi^2\,dxdy-\int_{\Om} |\nabla\psi|^2\,d\bfx.
\end{equation}

\subsubsection*{Estimates of the 3D scalar function $\psi$}

Let us control the term $\int_{\Om} |\nabla\psi|^2\,d\bfx$ by $\int_\GL\varphi^2\,dxdy$. Integrating by parts in the right-hand side of (\ref{TermNRJBis}), we find
\begin{equation}\label{LShadeTermsBord}
\int_\Om|\nabla\psi|^2\,d\bfx=\int_{\Sigma_x}\hspace{-0.1cm}\varphi(y,z)\psi(1,y,z)\,dydz+\int_{\Sigma_y}\hspace{-0.1cm}\varphi(x,z)\psi(x,1,z)\,dxdz+\int_{\Sigma_z}\hspace{-0.1cm}\varphi(x,y)\psi(x,y,1)\,dxdy
\end{equation}
where $\Sigma_x\coloneqq\{1\}\times(0;1)\times(0;1)$, $\Sigma_y\coloneqq(0;1)\times\{1\}\times(0;1)$, $\Sigma_z\coloneqq(0;1)\times(0;1)\times\{1\}$. Here we have used that $\E_p$ is divergence free in $\Om\setminus\{\Sigma_x\cup\Sigma_y\cup\Sigma_z\}$. The Cauchy-Schwarz inequality in $\mL^2(\Sigma_x)$ gives
\[
\bigg|\int_{\Sigma_x}\varphi(y,z)\psi(1,y,z)\,dydz\bigg|\le \|\varphi\|_{\mL^2(\square)}\|\psi\|_{\mL^2(\Sigma_x)}.
\]
Estimating similarly the integrals on $\Sigma_y$, $\Sigma_z$ in (\ref{LShadeTermsBord}) and using that for $a,b,c>0$, $a+b+c\le\sqrt{3}(a^2+b^2+c^2)^{1/2}$, we deduce
\begin{equation}\label{EstimDep}
\int_\Om|\nabla\psi|^2\,d\bfx \le \sqrt{3}\|\varphi\|_{\mL^2(\square)}(\|\psi\|^2_{\mL^2(\Sigma_x)}+\|\psi\|^2_{\mL^2(\Sigma_y)}+\|\psi\|^2_{\mL^2(\Sigma_z)})^{1/2}.
\end{equation}
\begin{lemma}
For the function $\psi$ introduced in (\ref{PbDefPsiBis}), we have the estimate
\begin{equation}\label{EstimEtaQua}
\|\psi\|^2_{\mL^2(\Sigma_x)}+\|\psi\|^2_{\mL^2(\Sigma_y)}+\|\psi\|^2_{\mL^2(\Sigma_z)} \le \cfrac{4\sqrt{2}}{\pi(8+\sqrt{2})} \int_{\Om}|\nabla\psi|^2\,d\bfx.
\end{equation}
\end{lemma}
\begin{proof}
We start by writing 
\[
\|\psi\|_{\mL^2(\Sigma_z)}^2 =\dsp\int_0^1\int_0^1\psi^2(x,y,1)\,dxdy=-\dsp\int_{\Pi_z}2\psi\partial_z\psi\,d\bfx\le\sqrt{2\pi^2}\dsp\int_{\Pi_z}\psi^2\,d\bfx+\cfrac{1}{\sqrt{2\pi^2}}\dsp\int_{\Pi_z}(\partial_z\psi)^2\,d\bfx.
\]
From the Poincar\'e inequality, for all $\phi\in\mH^1(\Pi_z)$ such that $\phi=0$ on $\partial\Pi_z\cap\partial\Om$,
\[
2\pi^2\dsp\int_{\Pi_z}\phi^2\,d\bfx \le \dsp\int_{\Pi_z}(\partial_x\phi)^2+(\partial_y\phi)^2\,d\bfx,
\]
we deduce 
\[
\|\psi\|_{\mL^2(\Sigma_z)}^2 \le \cfrac{1}{\sqrt{2\pi^2}}\int_{\Pi_z}|\nabla\psi|^2\,d\bfx.
\]
Deriving similar estimates in $\Pi_x$ and $\Pi_y$, we obtain 
\begin{equation}\label{EstimT1bis}
\|\psi\|^2_{\mL^2(\Sigma_x)}+\|\psi\|^2_{\mL^2(\Sigma_y)}+\|\psi\|^2_{\mL^2(\Sigma_z)} \le \cfrac{1}{\sqrt{2\pi^2}}\int_{\Pi_x\cup\Pi_y\cup\Pi_z}|\nabla\psi|^2\,d\bfx.
\end{equation}
As in (\ref{EstimT1bisX}), it is tempting  to control the right-hand side of (\ref{EstimT1bis}) by $\int_{\Om}|\nabla\psi|^2\,d\bfx$. However this is too rough. To improve the estimate, let us exploit the quantity $\int_{\mathcal{C}}|\nabla\psi|^2\,d\bfx$. More precisely, writing, for all $\eta>0$,
\[
\|\psi\|_{\mL^2(\Sigma_z)}^2 =\dsp\int_0^1\int_0^1\psi^2(x,y,1)\,dxdy=\int_{\mathcal{C}}2\psi\partial_z\psi\,d\bfx
\le \eta\int_{\mathcal{C}}\psi^2\,d\bfx+\eta^{-1}\int_{\mathcal{C}}(\partial_z\psi)^2\,d\bfx,
\]
and deriving similar inequalities for $\|\psi\|_{\mL^2(\Sigma_x)}$, $\|\psi\|_{\mL^2(\Sigma_x)}$, we find 
\begin{equation}\label{EstimInterEta}
\|\psi\|^2_{\mL^2(\Sigma_x)}+\|\psi\|^2_{\mL^2(\Sigma_y)}+\|\psi\|^2_{\mL^2(\Sigma_z)}
\le\dsp 3\eta\int_{\mathcal{C}}\psi^2\,d\bfx+\eta^{-1}\int_{\mathcal{C}}|\nabla\psi|^2\,d\bfx . 
\end{equation}
Since the first eigenvalue of the problem
\[
\begin{array}{|rcll}
-\Delta \gamma&=&\tau\gamma&\mbox{ in }\mathcal{C}\\[2pt]
\gamma&=& 0 &\mbox{ on }\{\bfx\in\partial\mathcal{C}\,|\,x=0\mbox{ or }y=0\mbox{ or }z=0\}\\[2pt]
\bfnu\cdot\nabla\gamma&=& 0 &\mbox{ on }\{\bfx\in\partial\mathcal{C}\,|\,x=1\mbox{ or }y=1\mbox{ or }z=1\}
\end{array}
\]
is equal to $3\pi^2/4$, we have the Poincar\'e estimate
\[
\cfrac{3\pi^2}{4}\int_{\mathcal{C}}\psi^2\,d\bfx \le \int_{\mathcal{C}}|\nabla\psi|^2\,d\bfx.
\]
Inserting it in (\ref{EstimInterEta}), we obtain 
\[
\|\psi\|^2_{\mL^2(\Sigma_x)}+\|\psi\|^2_{\mL^2(\Sigma_y)}+\|\psi\|^2_{\mL^2(\Sigma_z)} \le \Big(\cfrac{4\eta}{\pi^2}+\eta^{-1}\Big) \int_{\mathcal{C}}|\nabla\psi|^2\,d\bfx.
\]
Choosing the optimal $\eta=\pi/2$ yields 
\begin{equation}\label{EstimEtaTer}
\|\psi\|^2_{\mL^2(\Sigma_x)}+\|\psi\|^2_{\mL^2(\Sigma_y)}+\|\psi\|^2_{\mL^2(\Sigma_z)} \le \cfrac{4}{\pi} \int_{\mathcal{C}}|\nabla\psi|^2\,d\bfx.
\end{equation}
By multiplying (\ref{EstimT1bis}), (\ref{EstimEtaTer}) respectively by $\sqrt{2\pi^2}$, $\pi/4$ and summing the two resulting inequalities, we get (\ref{EstimEtaQua}).
\end{proof}
\noindent From (\ref{EstimDep}), (\ref{EstimEtaQua}), we derive 
\begin{equation}\label{EstimSuite}
\int_\Om|\nabla\psi|^2\,d\bfx \le \cfrac{12\sqrt{2}}{\pi(8+\sqrt{2})}\,\|\varphi\|^2_{\mL^2(\square)}.
\end{equation}
Gathering (\ref{EstimAbove}), (\ref{Rel2}) and (\ref{EstimSuite}), we find
\begin{equation}\label{CalculImp}
\begin{array}{ll}
&\dsp\int_{\Om} |\curlvec\tilde{\E}_{p}|^2\,d\bfx-\pi^2\int_{\Om} |\tilde{\E}_{p}|^2\,d\bfx \\[10pt] =&3(\lambda_\bullet-\pi^2)\dsp\int_\GL\varphi^2\,dxdy+\pi^2\int_{\Om} |\nabla\psi|^2\,d\bfx-6\int_0^1\varphi(1,y)^2\,dy\\[10pt]
\le&3(\lambda_\bullet-\pi^2)\dsp\int_\GL\varphi^2\,dxdy+\cfrac{12\sqrt{2}\pi}{(8+\sqrt{2})}\,\dsp\int_{\square}\varphi^2\,dxdy -6\int_0^1\varphi(1,y)^2\,dy.
\end{array}
\end{equation}

\subsubsection*{Estimates of the 2D scalar function $\varphi$}

Now we have to compare $\int_\GL\varphi^2\,dxdy$, $\int_{\square}\varphi^2\,dxdy$, $\int_0^1\varphi(1,y)^2\,dy$. The technique based on the Poincar\'e-Friedrichs inequality (\ref{Fried_ineqX}) presented in the previous paragraph revealed insufficient to show that the right-hand side of (\ref{CalculImp}) is negative. This is why we decided to work with an explicit representation of $\varphi$ obtained by using decomposition in Fourier series.\\
\newline
Since $\varphi(x,y)=\varphi(y,x)$ in $\GL$, there holds 
\begin{equation}\label{DecompoTermL2}
\|\varphi\|^2_{\mL^2(\GL)}=2\|\varphi\|^2_{\mL^2(\GL_+)}+\|\varphi\|^2_{\mL^2(\square)}.
\end{equation}
Set $I=(0;1)$ and for $n\in\N^\ast$, $\xi_n(t)=\sqrt{2}\sin(n\pi t)$. We have the representation
\[
\varphi(x,y)=\sum_{n=1}^{+\infty} \alpha_n e^{-\beta_n(x-1)}\xi_n(y)\qquad\mbox{ in }\GL_+, 
\] 
with $\beta_n\coloneqq\sqrt{n^2\pi^2-\lambda_\bullet}$ and $\alpha_n\coloneqq\int_I \varphi(1,t)\xi_n(t)\,dt$. This yields 
\[
\int_0^1\varphi(1,y)^2\,dy=\sum_{n=1}^{+\infty} \alpha^2_n \qquad\mbox{ and }\qquad \|\varphi\|^2_{\mL^2(\GL_+)} = \sum_{n=1}^{+\infty} \cfrac{\alpha^2_n }{2\beta_n}\,.
\]
Next we compute $\|\varphi\|^2_{\mL^2(\square)}$. Using again that $\varphi(x,y)=\varphi(y,x)$ in $\GL$, we obtain for $(x,y)\in\square$ 
\[
\varphi(x,y)=\sum_{n=1}^{+\infty} \alpha_n\,\xi_n(y)\,\cfrac{\sinh(\beta_nx)}{\sinh(\beta_n)}+\sum_{n=1}^{+\infty} \alpha_n\,\xi_n(x)\,\cfrac{\sinh(\beta_ny)}{\sinh(\beta_n)}\,.
\]
Set
\[
g(x,y)=\sum_{n=1}^{+\infty} \alpha_n\,\xi_n(y)\,\cfrac{\sinh(\beta_nx)}{\sinh(\beta_n)}
\]
so that $\varphi(x,y)=g(x,y)+g(y,x)$ and 
\[
\|\varphi\|^2_{\mL^2(\square)} =2\|g\|^2_{\mL^2(\square)}+2\int_{\square}g(x,y)g(y,x)\,dxdy.
\]
We have
\[
\|g\|^2_{\mL^2(\square)}=\sum_{n=1}^{+\infty} \alpha^2_n \int_I \cfrac{\sinh(\beta_nt)^2}{\sinh(\beta_n)^2}\,dt=\sum_{n=1}^{+\infty} \alpha^2_n \,\cfrac{e^{2\beta_n} - e^{-2\beta_n} - 4\beta_n}{8\sinh(\beta_n)^2\beta_n}\,.
\]
On the other hand, we can write
\[
\int_{\square}g(x,y)g(y,x)\,dxdy = \sum_{n=1}^{+\infty}\sum_{p=1}^{+\infty} \cfrac{\alpha_n\alpha_p}{\sinh(\beta_n)\sinh(\beta_p)}\,J_{np}J_{pn}
\]
with 
\[
J_{np}=\int_{I}\xi_n(t)\sinh(\beta_pt)\,dt=\cfrac{(-1)^{n+1}n\pi\sqrt{2}\sinh(\beta_p)}{(n^2\pi^2+\beta_p^2)}\,.
\]
Hence, using that $\beta_n^2=n^2\pi^2-\lambda_\bullet\ge\pi^2-\lambda_\bullet>0$ for all $n\in\N^\ast$, we obtain 
\[
\begin{array}{rcl}
\dsp\left|\int_{\square}g(x,y)g(y,x)\,dxdy\right| &=& \dsp\left|\sum_{n=1}^{+\infty}\sum_{p=1}^{+\infty} 2\,\cfrac{(-1)^{n+1}n\pi\alpha_n}{n^2\pi^2+\beta_p^2}\,\cfrac{(-1)^{p+1}p\pi\alpha_p}{p^2\pi^2+\beta_n^2}\right|\\[14pt]
& \le & \dsp\sum_{n=1}^{+\infty}\sum_{p=1}^{+\infty} 2\,\cfrac{n\pi|\alpha_n|}{n^2\pi^2+\beta_p^2}\,\cfrac{p\pi|\alpha_p|}{p^2\pi^2+\beta_n^2}|
\leq\left(\sum_{n=1}^{+\infty}\cfrac{\sqrt{2}n\pi}{(n^2+1)\pi^2-\lambda_\bullet}\,|\alpha_n|\right)^2.
\end{array}
\]
Using the Cauchy-Schwarz inequality in $\ell^2(\N)$, the set of square summable sequences, this gives  
\begin{equation}\label{Termg2bis}
\bigg|\int_{\square}g(x,y)g(y,x)\,dxdy\bigg| \le C_{\square}\,\sum_{n=1}^{+\infty} \alpha^2_n\qquad\mbox{ with }C_{\square}\coloneqq\sum_{n=1}^{+\infty} \cfrac{2n^2\pi^2}{((n^2+1)\pi^2-\lambda_\bullet)^2}\,.
\end{equation}
\subsubsection*{Conclusion}
Inserting (\ref{DecompoTermL2}) in (\ref{CalculImp}) and using (\ref{DecompoTermL2})--(\ref{Termg2bis}),  we arrive at 
\[
\begin{array}{l}
\phantom{\le}\dsp\int_{\Om} |\curlvec\tilde{\E}_{p}|^2\,d\bfx-\pi^2\int_{\Om} |\tilde{\E}_{p}|^2\,d\bfx \\[10pt] 
\le 2C_1\dsp\int_{\GL_+}\hspace{-0.2cm}\varphi^2\,dxdy+C_2\dsp\int_{\square}\varphi^2\,dxdy -6\int_0^1\varphi(1,y)^2\,dy\le  \dsp\sum_{n=1}^{+\infty}q(n)\alpha_n^2
\end{array}
\]
with 
\[
C_1\coloneqq3(\lambda_\bullet-\pi^2),\quad C_2\coloneqq\bigg(3(\lambda_\bullet-\pi^2)+\cfrac{12\sqrt{2}\pi}{(8+\sqrt{2})}\bigg),\quad q(n)\coloneqq\cfrac{C_1}{\beta_n}+C_2\bigg(2C_{\square}+\cfrac{e^{2\beta_n} - e^{-2\beta_n} - 4\beta_n}{4\sinh(\beta_n)^2\beta_n}\bigg)-6.
\]
We find that $C_2$ is positive ($C_2\approx3.571$). Thus we can write
\begin{equation}\label{Calcqn}
q(n) \le \cfrac{C_1}{\beta_n}+C_2\bigg(2C_{\square}+\cfrac{e^{2\beta_n}}{4\sinh(\beta_n)^2\beta_n}\bigg)-6.
\end{equation}
Using that $4C_1+C_2<0$, it is simple to prove that the quantity $C_1/\beta_n+C_2e^{2\beta_n}/(4\sinh(\beta_n)^2\beta_n)$ is negative for all $n\in\N^\ast$. Therefore from (\ref{Calcqn}), we obtain $\sup_{n\in\N^\ast}q(n) \le 2C_{\square}C_2-6\approx -3.8205$ because $C_\square\approx0.3052$. Thus the right-hand side in (\ref{CalculImp}) is negative and the min-max principle ensures that $\sigma_{\mrm{d}}(A)\ne\emptyset$. As in the previous paragraph, we do not state the result as a proposition because the final computation is based on the numerical value of $\lambda_\bullet$, the first eigenvalue of the 2D problem (\ref{Spectral2D}). However we have great confidence in the calculation of that quantity. 

\begin{remark}
Similar to Remark \ref{RmkAppThm}, we find that the Dirichlet Laplacian in the tripode $\Om$ defined in (\ref{DefTripode}) 
has a non-empty discrete spectrum with a smallest eigenvalue approximately equal to $16.98>\pi^2$. Therefore, we could not simply apply Theorem \ref{ThmBigRes} to infer that $\sigma_{\mrm{d}}(A)\ne\emptyset$.
\end{remark}

\section{Perturbation of the material coefficients}\label{SectionVariableCoef}

Up to now we have provided examples of homogeneous waveguides with well-chosen resonators or particular geometries for which we can prove existence of eigenvalues.  In this section, we wish to show how local perturbations of the dielectric permittivity $\eps$ and/or the magnetic permeability $\mu$ can support trapped modes.\\
\newline
Define the straight waveguide 
\[
\Om\coloneqq S\times\R\subset\R^3
\]
where $S$ is a simply connected bounded domain of $\R^2$ with Lipschitz boundary $\partial S$. We assume that $\eps$, $\mu$ are scalar real valued functions which belong to $\mL^{\infty}(\Om)$ with $\eps$, $\mu\ge c$ for some constant $c>0$. Additionally we assume that $\eps-1$ and $\mu-1$ have compact support. In inhomogeneous waveguides, the analysis of the spectrum of the electric problem leads to consider the system
\begin{equation}\label{DefPb0Inhomo}
\begin{array}{|rcll}
\curlvec\left(\mu^{-1}\curlvec\E\right)&=&\lambda\eps\E&\mbox{ in }\Om\\[2pt]
\div(\eps\E)&=&0&\mbox{ in }\Om\\[2pt]
\E\times\bfnu&=&0&\mbox{ on }\partial\Om.
\end{array}
\end{equation}
We endow $\boldsymbol{\mL}^2(\Om)$ with the weighted inner product 
\[
(\E,\E')\mapsto\int_{\Om}\eps\E\cdot\E'\,d\bfx
\]
and still denote by $A$ the unbounded operator of the Hilbert space 
\[
\boldsymbol{\mH}(\div;\eps,0)\coloneqq\{\E\in\boldsymbol{\mL}^2(\Om)\,|\,\div(\eps\E)=0\mbox{ in }\Om\}
\]
associated with (\ref{DefPb0Inhomo}) such that 
\[
\begin{array}{|rcl}
D(A)&=&\{\E\in\boldsymbol{\mX}_N(\eps)\,|\,\curlvec(\mu^{-1}\curlvec\E)\in\boldsymbol{\mL}^2(\Om)\}\\[4pt]
A\E&=&\eps^{-1}\curlvec(\mu^{-1}\curlvec\E).
\end{array}
\]
Here the space $\boldsymbol{\mX}_N(\eps)$ is defined by
\[
\boldsymbol{\mX}_N(\eps)\coloneqq\boldsymbol{\mH}_N(\curlvec)\cap \boldsymbol{\mH}(\div;\eps,0).
\]
This operator $A$ is self-adjoint and non-negative. Again, according to Proposition \ref{DefSigmaEss}, we have 
\[
\sigma_{\mrm{ess}}(A)=[\lambda_N;+\infty),
\]
where $\lambda_N$ is the first positive eigenvalue of the 2D Neumann Laplacian in $S$. With the min-max principle, to show that $A$ has discrete spectrum ($\sigma_{\mrm{d}}(A)\ne\emptyset$), we have to exhibit some $\E\in\boldsymbol{\mX}_N(\eps)\setminus\{0\}$ such that
\begin{equation}\label{MinMaxNonHomo}
\cfrac{\dsp\int_{\Om}\mu^{-1}|\curlvec\E|^2\,d\bfx}{\dsp\int_{\Om}\eps|\E|^2\,d\bfx}<\lambda_N\qquad\Leftrightarrow\qquad \dsp\int_{\Om}\mu^{-1}\,|\curlvec\E|^2-\lambda_N\eps|\E|^2\,d\bfx<0.
\end{equation}
From this, it is natural to expect that discrete spectrum for $A$ appears when increasing the values of $\eps$, $\mu$ from the reference situation $\eps\equiv1$, $\mu\equiv1$. This is what we justify below, the main difficulty being that the space $\boldsymbol{\mX}_N(\eps)$ depends on $\eps$.

\subsection{$z$-variable $\eps$ and variable $\mu$}
Assume in this paragraph that $\eps$ is a function of the $z$-variable only, \textit{i.e.} $\eps(x,y,z)=\eps(z)$. On the other hand, $\mu$ is allowed to depend on $x$, $y$ and $z$. To achieve (\ref{MinMaxNonHomo}), classically, we work with the test function $\E_\alpha$ defined from the generalized eigenfunction associated to the bottom of the essential spectrum $\lambda_N$ such that for $\alpha>0$ constant, 
\begin{equation}\label{DefEModeTE}
\E_\alpha(\bfx)=\left(\begin{array}{c}
\curlvec_{\mrm{2D}}\,\varphi_N(x,y)\\
0
\end{array}\right)e^{-\alpha|z|}.
\end{equation}
Here $\varphi_N$ is an eigenfunction of the Neumann Laplacian in $S$ associated with $\lambda_N$ such that $\|\varphi_N\|_{\mL^2(S)}=1$. For $\alpha>0$, we have $\E_\alpha\in\boldsymbol{\mH}_N(\curlvec)$ and since $\eps$ depends only on $z$, there holds $\div(\eps\E_\alpha)=0$ in $\Om$. Thus $\E_\alpha$ belongs to $\boldsymbol{\mX}_N(\eps)$.\\
\newline
A direct calculation yields
\[
\curlvec\E_\alpha(\bfx)=\left(\begin{array}{c}
-\alpha\,\mrm{sign}(z)\nabla\varphi_N(x,y)\\[2pt]
\lambda_N\varphi_N(x,y)
\end{array}\right)e^{-\alpha|z|},
\]
so that we obtain
\begin{equation}\label{RelationHomogeneous}
\begin{array}{rcl}
\dsp\int_{\Om}\hspace{-0.05cm}|\curlvec\E_\alpha|^2-\lambda_N|\E_\alpha|^2\,d\bfx &\hspace{-0.2cm}=&\hspace{-0.2cm} \dsp\alpha^2\int_{\Om}|\nabla\varphi_N|^2e^{-2\alpha |z|}\,d\bfx+\lambda_N\int_{\Om}(\lambda_N\varphi_N^2-|\nabla\varphi_N|^2)e^{-2\alpha |z|}\,d\bfx\\[10pt]
&\hspace{-0.2cm}=&\hspace{-0.2cm} \dsp\alpha^2\int_{\Om}|\nabla\varphi_N|^2e^{-2\alpha |z|}\,d\bfx=\alpha\lambda_N.
\end{array}
\end{equation}
This allows us to write
\begin{equation}\label{CriterionAlpha}
\int_{\Om}\mu^{-1}\,|\curlvec\E_\alpha|^2-\lambda_N\eps|\E_\alpha|^2\,d\bfx = \alpha\lambda_N+\int_{\Om}(\mu^{-1}-1)\,|\curlvec\E_\alpha|^2-\lambda_N(\eps-1)|\E_\alpha|^2\,d\bfx.
\end{equation}
Since $\mu^{-1}-1$ and $\eps-1$ are compactly supported, we can apply the Lebesgue dominated convergence theorem to take the limit as $\alpha$ tends to zero above. We find that the quantity in (\ref{CriterionAlpha}) converges to 
\[
\int_{\Om}(\mu^{-1}-1)\,\lambda^2_N\varphi_N^2+(1-\eps)\lambda_N|\nabla\varphi_N|^2\,d\bfx=\lambda^2_N\int_{\Om}((\eps\mu)^{-1}-1)\,\eps\varphi_N^2\,d\bfx
\]
(above we use again that $\eps$ depends only on $z$). Thus we can state the following result:
\begin{proposition}\label{Proposition1Variable}
Assume that $\eps$ depends only on $z$. If $\eps$, $\mu$ are such that
\begin{equation}\label{CriterionEpszvariable}
\int_{\Om}((\eps\mu)^{-1}-1)\,\eps\varphi_N^2\,d\bfx<0,
\end{equation}
where $\varphi_N$ is an eigenfunction of the Neumann Laplacian in $S$ associated with $\lambda_N$, then $\sigma_{\mrm{d}}(A)\ne\emptyset$. 
\end{proposition}
\begin{remark}
If we introduce the wave celerity $c>0$ such that $c^{2}=(\eps\mu)^{-1}$, we observe that the condition (\ref{CriterionEpszvariable}) is satisfied if $c\le1$ in $\Om$ with $c<1$ in a set with non-empty interior. It is the case for example when
$\eps\ge1$, $\mu\ge1$ in $\Om$ with $\eps>1$ or $\mu>1$ in a set with non-empty interior. Note that since $\varphi_N$ varies, it can also hold for certain $\eps$, $\mu$ such that $c-1$ is sign-changing. 
\end{remark}
\begin{remark}
When $\lambda_N$ is a double eigenvalue, for example when $S$ is a square, the technique above ensures that the total multiplicity of $\sigma_{\mrm{d}}(A)$ is at least two if $\eps$, $\mu$ satisfy the assumptions of Proposition \ref{Proposition1Variable}. Indeed if $\varphi_N^i$, $i=1,2$ are linearly independent eigenfunctions of the Neumann Laplacian associated with $\lambda_N$, define $\E_\alpha^i$ as in (\ref{DefEModeTE}) with $\varphi_N$ replaced by $\varphi_N^i$. It is easy to show that $\E_\alpha^1$, $\E_\alpha^2$ are orthogonal in $\boldsymbol{\mX}_N(\eps)$ for any $\alpha>0$. Then we conclude classically with the min-max principle.
\end{remark}
\noindent The constraint that $\eps$ must depend only on $z$ seems quite artificial and appears only for technical reasons. Let us see how to weaken it.

\subsection{Variable $\eps$ and $\mu$}

We use the same notation as in the previous paragraph and work again with the $\E_\alpha$ defined in (\ref{DefEModeTE}). When $\eps$ varies also in $x,y$, in general we have $\div(\eps\E_\alpha)\not\equiv0$ in $\Om$. This leads us to introduce the function $\psi_\alpha\in\mH^1_0(\Om)$ such that, for all $\psi'\in\mH^1_0(\Om)$,
\begin{equation}\label{DefPsiVariableDouble}
\begin{array}{rcl}
\dsp\int_\Om \eps\nabla\psi_\alpha\cdot\nabla\psi'\,d\bfx=\int_\Om \eps\E_\alpha\cdot\nabla\psi'\,d\bfx&=&\dsp\int_\Om\E_\alpha\cdot\nabla\psi'\,d\bfx+\int_\Om(\eps-1)\E_\alpha\cdot\nabla\psi'\,d\bfx\\[10pt]
&=&\dsp\int_\Om(\eps-1)\E_\alpha\cdot\nabla\psi'\,d\bfx
\end{array}
\end{equation}
 (to obtain the last equality, we have used that $\div\,\E_\alpha=0$ in $\Om$). Note that in $\mH^1_0(\Om)$, one has a Poincar\'e inequality as in (\ref{PoincareIneq}) which guarantees that $\psi_\alpha$ is well-defined. Set $\tilde{\E}_\alpha\coloneqq\E_\alpha-\nabla\psi_\alpha$. This field belongs to $\boldsymbol{\mX}_N(\eps)$. By exploiting the first equality in
\begin{equation}\label{TermeCroise}
\int_\Om \eps|\nabla\psi_\alpha|^2\,d\bfx=\int_\Om \eps\E_\alpha\cdot\nabla\psi_\alpha\,d\bfx=\int_\Om(\eps-1)\E_\alpha\cdot\nabla\psi_\alpha\,d\bfx,
\end{equation}
as well as (\ref{RelationHomogeneous}), we find
\begin{equation}\label{TermDecompoVaria}
\begin{array}{c}
\dsp\int_\Om \mu^{-1}|\curlvec\tilde{\E}_\alpha|^2-\lambda_N\eps|\tilde{\E}_\alpha|^2\,d\bfx=\dsp\int_\Om \mu^{-1}|\curlvec\E_\alpha|^2-\lambda_N(\eps|\E_\alpha|^2-\eps|\nabla\psi_\alpha|^2)\,d\bfx\\[10pt]
\hspace{3cm}=\dsp\alpha\lambda_N+\int_\Om (\mu^{-1}-1)|\curlvec\E_\alpha|^2\,d\bfx+\lambda_N\dsp\int_\Om (1-\eps)|\E_\alpha|^2+\eps|\nabla\psi_\alpha|^2\,d\bfx.
\end{array}
\end{equation}
Let us estimate the quantity $\int_{\Om}\eps|\nabla\psi_\alpha|^2\,d\bfx$ in the right-hand side above. From (\ref{TermeCroise}) and the Young inequality $ab\le a^2/4+b^2$, assuming that $\eps-1\ge0$ in $\Om$, we can write
\[
\int_\Om \eps|\nabla\psi_\alpha|^2\,d\bfx = \int_\Om(\eps-1)\E_\alpha\cdot\nabla\psi_\alpha\,d\bfx\le \cfrac{1}{4} \int_\Om(\eps-1)|\E_\alpha|^2\,d\bfx+\int_\Om (\eps-1)|\nabla\psi_\alpha|^2\,d\bfx
\]
and so 
\begin{equation}\label{Term1}
\int_\Om |\nabla\psi_\alpha|^2\,d\bfx\le \cfrac{1}{4} \int_\Om(\eps-1)|\E_\alpha|^2\,d\bfx.
\end{equation}
From (\ref{TermeCroise}), we also have
\[
\int_\Om \eps|\nabla\psi_\alpha|^2\,d\bfx\le \cfrac{1}{4} \int_\Om(\eps-1)^2|\E_\alpha|^2\,d\bfx+\int_{\mrm{supp}(\eps-1)}  |\nabla\psi_\alpha|^2\,d\bfx,
\]
where $\mrm{supp}(\eps-1)$ stands for the support of $\eps-1$, and so 
\begin{equation}\label{Term2}
\int_{\Om\setminus\mrm{supp}(\eps-1)} |\nabla\psi_\alpha|^2\,d\bfx+\int_{\Om} (\eps-1)|\nabla\psi_\alpha|^2\,d\bfx \le \cfrac{1}{4} \int_\Om(\eps-1)^2|\E_\alpha|^2\,d\bfx.
\end{equation}
Summing (\ref{Term1}) and (\ref{Term2}), we obtain
\begin{equation}\label{Control}
\int_\Om \eps|\nabla\psi_\alpha|^2\,d\bfx\le \cfrac{1}{4} \int_\Om(\eps-1)|\E_\alpha|^2\,d\bfx+\cfrac{1}{4} \int_\Om(\eps-1)^2|\E_\alpha|^2\,d\bfx.
\end{equation}
Inserting (\ref{Control}) in (\ref{TermDecompoVaria}) yields 
\[
\begin{array}{ll}
&\dsp\int_\Om \mu^{-1}|\curlvec\tilde{\E}_\alpha|^2-\lambda_N\eps|\tilde{\E}_\alpha|^2\,d\bfx  \\[10pt]
\le &\dsp\alpha\lambda_N+\int_\Om (\mu^{-1}-1)|\curlvec\E_\alpha|^2\,d\bfx+\cfrac{\lambda_N}{4}\dsp\int_\Om (\eps-1)(\eps-4)|\E_\alpha|^2\,d\bfx.
\end{array}
\]
The right-hand side above converges as $\alpha$ tends to zero to the quantity
\[ 
\lambda_N\int_{\Om}(\mu^{-1}-1)\,\lambda_N\varphi_N^2+\cfrac{1}{4}\,(\eps-1)(\eps-4)|\nabla\varphi_N|^2\,d\bfx.
\]
Thus we have the following statement:
\begin{proposition}
If $\eps$, $\mu$ are such that $\eps\ge1$ in $\Om$ and 
\begin{equation}\label{CriterionEpszvariable2}
\int_{\Om}(\mu^{-1}-1)\,\lambda_N\varphi_N^2+\cfrac{1}{4}\,(\eps-1)(\eps-4)|\nabla\varphi_N|^2\,d\bfx<0,
\end{equation}
where $\varphi_N$ is an eigenfunction of the Neumann Laplacian in $S$ associated with $\lambda_N$, then $\sigma_{\mrm{d}}(A)\ne\emptyset$. 
\end{proposition}
\begin{remark}
The condition (\ref{CriterionEpszvariable2}) holds if $1\le\eps\le4$, $1\le\mu $ in $\Om$ with $1<\eps<4$ or $1<\mu$ in a set with non-empty interior. The criterion (\ref{CriterionEpszvariable2}) can also be satisfied in other situations, for example in cases where $\eps\ge4$ or $\mu\le 1$. However we emphasize that the calculation above has been done by exploiting the assumption $\eps\ge1$. 
\end{remark}
\noindent We still have an asymmetry in the assumptions made on $\eps$, $\mu$ which seems artificial. In the next subsection, we show how to remove it.

\subsection{Variable $\eps$ and $\mu$: an approach via the magnetic operator}

The analysis of the spectrum of the magnetic problem leads to consider the system
\begin{equation}\label{DefPb0InhomoH}
\begin{array}{|rcll}
\curlvec\left(\eps^{-1}\curlvec\H\right)&=&\lambda\mu\H&\mbox{ in }\Om\\[2pt]
\div(\mu\H)&=&0&\mbox{ in }\Om\\[2pt]
(\eps^{-1}\curlvec\H)\times\bfnu&=&0&\mbox{ on }\partial\Om.
\end{array}
\end{equation}
For the definition of an operator associated with (\ref{DefPb0InhomoH}), we endow $\boldsymbol{\mL}^2(\Om)$ with the weighted inner product 
\[
(\H,\H')\mapsto\int_{\Om}\mu\H\cdot\H'\,d\bfx
\]
and we denote by $A^{\H}$ the unbounded operator of the Hilbert space
\[
\boldsymbol{\mH}_T(\div;\mu,0)=\{\H\in\boldsymbol{\mL}^2(\Om)\,|\,\div(\mu\H)=0\mbox{ in }\Om,\,\mu\H\cdot\bfnu=0\mbox{ on }\partial\Om\}
\]
associated with (\ref{DefPb0InhomoH}) such that 
\begin{equation}\label{DefMagOp}
\begin{array}{|rcl}
D(A^{\H})&=&\{\H\in\boldsymbol{\mX}_T(\mu)\,|\,\curlvec(\eps^{-1}\curlvec\H)\in\boldsymbol{\mL}^2(\Om)\mbox{ and }(\eps^{-1}\curlvec\H)\times\bfnu=0\mbox{ on }\partial\Om\}\\[4pt]
A^{\H}\H&=&\mu^{-1}\curlvec(\eps^{-1}\curlvec\H).
\end{array}
\end{equation}
Here the space $\boldsymbol{\mX}_T(\mu)$ is defined by
\begin{equation}\label{DefSpaceH}
\boldsymbol{\mX}_T(\mu)\coloneqq\boldsymbol{\mH}(\curlvec)\cap\boldsymbol{\mH}_T(\div;\mu,0)
\end{equation}
with 
\[
\boldsymbol{\mH}(\curlvec)\coloneqq\{\H\in\boldsymbol{\mL}^2(\Om)\,|\,\curlvec\H\in\boldsymbol{\mL}^2(\Om)\}.
\]
This operator $A^{\H}$ is self-adjoint and non-negative. Exploiting the equivalence between Problem (\ref{DefPb0InhomoH}) and Maxwell's equations in $\E$, $\H$, classically, one establishes the following result.
\begin{lemma}\label{LemmaSpectresEH}
Assume $\Om$, $\eps$, $\mu$ as described in the beginning of the section. Then 
\[
\sigma_{\mrm{ess}}(A)=\sigma_{\mrm{ess}}(A^{\H})=[\lambda_N;+\infty).
\]
Moreover $\lambda_\bullet>0$ is an eigenvalue of $A$ if and only if it is an eigenvalue of $A^{\H}$.
\end{lemma}
\noindent From the min-max principle, to guarantee that $\sigma_{\mrm{d}}(A)=\sigma_{\mrm{d}}(A^{\H})\ne\emptyset$, it is sufficient to prove the existence of 
$\H\in\boldsymbol{\mX}_T(\Om)\setminus\{0\}$ such that
\[
\cfrac{\dsp\int_{\Om}\eps^{-1}|\curlvec\H|^2\,d\bfx}{\dsp\int_{\Om}\mu|\H|^2\,d\bfx}<\lambda_N\qquad\Leftrightarrow\qquad \dsp\int_{\Om}\eps^{-1}\,|\curlvec\H|^2-\lambda_N\mu|\H|^2\,d\bfx<0.
\]

\begin{proposition}\label{PropositionepsVaria}
If $\mu\equiv1$ in $\Om$  and $\eps$ is such that
\begin{equation}\label{CriterionH}
\dsp\int_\Om (\eps^{-1}-1)|\nabla\varphi_N|^2\,d\bfx<0,
\end{equation}
where $\varphi_N$ is an eigenfunction of the Neumann Laplacian in $S$ associated with $\lambda_N$, then $\sigma_{\mrm{d}}(A)=\sigma_{\mrm{d}}(A^{\H})\ne\emptyset$.  
\end{proposition}
\begin{remark}
Observe that criterion (\ref{CriterionH}) holds when $\eps\ge1$ satisfies $\eps\ge c>1$ in a set with non-empty interior. However situations where $\eps^{-1}-1$ changes sign can also be considered.
\end{remark}
\begin{proof}
For $\alpha>0$ constant, define $\H_\alpha$ such that
\[
\H_\alpha(\bfx)=\left(\begin{array}{c}
0\\
0\\
\varphi_N(x,y)
\end{array}\right)e^{-\alpha|z|},
\]
where $\varphi_N$ is an eigenfunction of the Neumann Laplacian in $S$ associated with $\lambda_N$. Thanks to the decay of the exponential, we have $\H_\alpha\in\boldsymbol{\mH}(\curlvec)$. A direct calculation yields
\[
\curlvec\H_\alpha(\bfx)=\left(\begin{array}{c}
\curlvec_{\mrm{2D}}\,\varphi_N(x,y)\\[2pt]
0
\end{array}\right)e^{-\alpha|z|}
\]
so that we find
\[
\dsp\int_{\Om}|\curlvec\H_\alpha|^2-\lambda_N|\H_\alpha|^2\,d\bfx =0.
\]
However note that
\[
\div\,\H_\alpha(\bfx)=-\alpha\,\mrm{sign}(z)\varphi_N(x,y)e^{-\alpha|z|}\not\equiv0\quad\mbox{ in }\Om.
\]
To obtain a divergence free field, let us remove a gradient to $\H_\alpha$. This looks similar to what has been done in (\ref{DefPsiVariableDouble}). However to satisfy the boundary condition appearing in $\boldsymbol{\mX}_T(1)$, the variational problem defining this gradient must be posed in $\mH^1(\Om)$ and not in $\mH^1_0(\Om)$ as in (\ref{DefPsiVariableDouble}) (in other words, the scalar field should satisfy Neumann and not Dirichlet boundary conditions). The existence of this gradient is not straightforward because the form $(\psi,\psi')\mapsto \int_\Om\nabla\psi\cdot\nabla\psi'\,d\bfx$ is not coercive in $\mH^1(\Om)$. However we can exploit separation of variables. Let us look for $\psi_\alpha\in\mH^1(\Om)$ such that $\H_\alpha-\nabla\psi_\alpha$ in $\boldsymbol{\mX}_T(1)$ (see the definition of this space in (\ref{DefSpaceH}) with $\mu=1$) with $\psi_\alpha$ of the form
\begin{equation}\label{DefPsi}
\psi_\alpha(\bfx)=\varphi_N(x,y)f_\alpha(z).
\end{equation}
Then we find that $f_\alpha\in\mH^1(\R)$ is the unique solution of 
\begin{equation}\label{Problem1D}
f_\alpha''(z)-\lambda_N f_\alpha(z)=-\alpha\,\mrm{sign}(z)e^{-\alpha|z|}\quad\mbox{ in }\R.
\end{equation}
Next we set $\tilde{\H}_\alpha\coloneqq\H_\alpha-\nabla\psi_\alpha$ with $\psi_\alpha$ as in (\ref{DefPsi}). This $\tilde{\H}_\alpha$ belongs to $\boldsymbol{\mX}_T(1)$. By exploiting the relation 
\[
\int_\Om |\nabla\psi_\alpha|^2\,d\bfx=\int_\Om \H_\alpha\cdot\nabla\psi_\alpha\,d\bfx,
\]
we find
\begin{equation}\label{TermDecompoVariaH}
\begin{array}{rcl}
\dsp\int_\Om \eps^{-1}|\curlvec\tilde{\H}_\alpha|^2-\lambda_N |\tilde{\H}_\alpha|^2\,d\bfx &=&\dsp\int_\Om \eps^{-1}|\curlvec\H_\alpha|^2-\lambda_N(|\H_\alpha|^2-|\nabla\psi_\alpha|^2)\,d\bfx\\[10pt]
&=&\dsp\int_\Om (\eps^{-1}-1)|\curlvec\H_\alpha|^2\,d\bfx+\lambda_N\dsp\int_\Om |\nabla\psi_\alpha|^2\,d\bfx.
\end{array}
\end{equation}
But from (\ref{DefPsi}), (\ref{Problem1D}), we can write 
\[
\begin{array}{rcl}
\dsp\int_\Om |\nabla\psi_\alpha|^2\,d\bfx&=&\dsp\int_{S} |\nabla\varphi_N(x,y)|^2f_\alpha(z)^2+\varphi_N(x,y)^2f_\alpha'(z)^2\,d\bfx\\[10pt]
&=&\dsp\int_{\R} (f_\alpha')^2+\lambda_Nf_\alpha^2\,dz 
\le \dsp\alpha\,\|e^{-\alpha |z|}\|_{\mL^2(\R)}\|f_\alpha\|_{\mL^2(\R)}\le \sqrt{\alpha}\,\|f_\alpha\|_{\mL^2(\R)}.
\end{array}
\]
Moreover from (\ref{Problem1D}), we obtain $\|f_\alpha\|_{\mL^2(\R)}\le C\sqrt{\alpha}$. We deduce  
\[
\dsp\int_\Om |\nabla\psi_\alpha|^2\,d\bfx \le C\alpha
\]
for some $C>0$ independent of $\alpha$. Using the above estimate as well as the fact that $\eps^{-1}-1$ is compactly supported, we can take the limit as $\alpha$ tends to zero in (\ref{TermDecompoVariaH}) to find that it converges to 
\[
\dsp\int_\Om (\eps^{-1}-1)|\nabla\varphi_N|^2\,d\bfx.
\]
Thus if (\ref{CriterionH}) holds, then for $\alpha>0$ small enough  we have $\int_\Om \eps^{-1}|\curlvec\tilde{\H}_\alpha|^2-\lambda_N |\tilde{\H}_\alpha|^2\,d\bfx<0$. 
\end{proof}

\noindent Finally we can state the following result for varying $\eps$ and $\mu$.
\begin{proposition}\label{PropoDouble}
Assume that $\eps,\mu\ge1$ satisfy $\eps\ge c>1$ or $\mu\ge c>1$ in a set with non-empty interior. Then $\sigma_{\mrm{d}}(A)=\sigma_{\mrm{d}}(A^{\H})\ne\emptyset$.  
\end{proposition}
\begin{proof}
If $\eps\equiv1$ in $\Om$, then Proposition \ref{Proposition1Variable} gives the result. Therefore, let us assume that $\eps\ge c>1$ in a set with non-empty interior. In this case, according to Proposition \ref{PropositionepsVaria}, the electric problem (\ref{DefPb0Inhomo}) with $\mu\equiv1$ admits an eigenvalue $\lambda_\bullet< \lambda_N$. As a consequence, there is some $\E_p\in\boldsymbol{\mX}_N(\eps)\setminus\{0\}$ such that 
\[
\begin{array}{|rcll}
\curlvec(\curlvec\E_p)&=&\lambda_\bullet\eps\E_p&\mbox{ in }\Om\\[2pt]
\div(\eps\E_p)&=&0&\mbox{ in }\Om\\[2pt]
\E_p\times\bfnu&=&0&\mbox{ on }\partial\Om.
\end{array}
\]
In particular, we have
\[
\dsp\int_{\Om}|\curlvec\E_p|^2-\lambda_N\eps|\E_p|^2\,d\bfx< \dsp\int_{\Om}|\curlvec\E_p|^2-\lambda_\bullet\eps|\E_p|^2\,d\bfx=0.
\]
This allows us to write
\[
\dsp\int_{\Om}\mu^{-1}\,|\curlvec\E_p|^2-\lambda_N\eps|\E_p|^2\,d\bfx<\dsp\int_{\Om}(\mu^{-1}-1)\,|\curlvec\E_p|^2\,d\bfx.
\]
If $\mu\ge1$ in $\Om$, the right-hand side above is non-positive and the min-max principle guarantees that $\sigma_{\mrm{d}}(A)\ne\emptyset$. 
\end{proof}

\begin{remark}
One can adapt the above approach to show that if $A$ has a non-empty discrete spectrum for $\eps\equiv\mu\equiv1$ in some $\Om$ presenting a local perturbation of the geometry with respect to the straight channel $S\times\R$, then $A$ also has a non-empty discrete spectrum in $\Om$ for $\eps\ge1$, $\mu\ge1$ such that $\eps-1$, $\mu-1$ have compact supports. 
\end{remark}

\subsection{Extension to matrix valued coefficients}\label{ParaAniso}

Assume in this paragraph that $\eps$, $\mu\in \mL^{\infty}(\Om,\Cplx^{3\times 3})$ are matrix valued functions such that $\eps(\bfx)$, $\mu(\bfx)$ are symmetric positive definite for almost all $\bfx\in\Om$. Suppose that $\eps^{-1}$, $\mu^{-1}\in \mL^{\infty}(\Om,\Cplx^{3\times 3})$ and that $\eps-\mrm{Id}$, $\mu-\mrm{Id}$ are compactly supported. In this case, the definitions for the spaces $\boldsymbol{\mX}_N(\eps)$, $\boldsymbol{\mX}_T(\mu)$ as well as the operators $A$, $A^{\H}$ extend naturally and some of the results above are still valid. Let us state them.\\
\newline
Assume additionally that $\eps$ is of the form 
\[
\eps(\bfx)=\left(\begin{array}{ccc}
\tilde{\eps}(z) & 0 & \eps_{xz}(x,y)\\
0 & \tilde{\eps}(z) & \eps_{yz}(x,y)\\
\eps_{xz}(x,y) &  \eps_{yz}(x,y) & \eps_{zz}(x,y,z)
\end{array}\right).
\] 
First, by adapting the proof of Proposition \ref{Proposition1Variable}, one obtains that $\sigma_{\mrm{d}}(A)\ne\emptyset$ when 
\[
\int_{\Om}(\tilde{\eps}^{-1}(\mu^{-1})_{zz}-1)\,\tilde{\eps}\varphi_N^2\,d\bfx<0.
\]
To establish this result, in the definition of $\E_\alpha$ in (\ref{DefEModeTE}) it suffices to replace the term $e^{-\alpha |z|}$ by $\chi_\alpha(z)$ with $\chi_\alpha(z)=\chi(\alpha z)$, the function $\chi\in\mathscr{C}^{\infty}(\R)$ being such that $\chi(z)=1$ for $|z|\le 1$, $\chi(z)=0$ for $|z|\ge 2$. Then the 
calculations are completely similar. By adapting in the same way the proof of Proposition \ref{PropositionepsVaria}, one shows that $\sigma_{\mrm{d}}(A)\ne\emptyset$ when $\mu\equiv1$ in $\Om$ and $\eps$ is such that
\[
\dsp\int_\Om (\tilde{\eps}^{-1}-1)|\nabla\varphi_N|^2\,d\bfx<0.
\]
Finally, one proves as in Proposition \ref{PropoDouble} that when $\eps,\mu\ge \mrm{Id}$ (in the sense of symmetric matrices) satisfy $\eps\ge c\,\mrm{Id}>\mrm{Id}$ or $\mu\ge c\,\mrm{Id}>\mrm{Id}$ in a set with non-empty interior, $\sigma_{\mrm{d}}(A)\ne\emptyset$.  

\section{Playing with symmetries to exhibit embedded eigenvalues}\label{SectionSym}

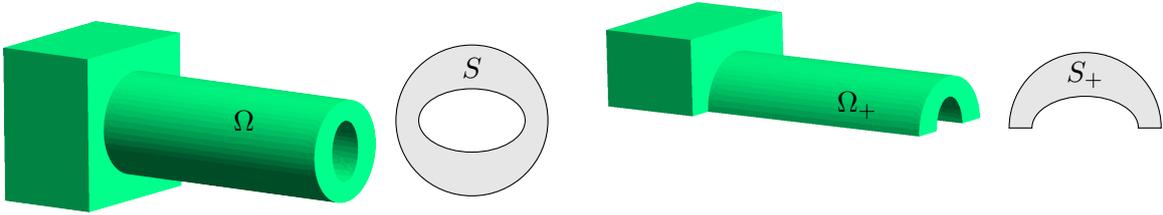
\begin{figure}[!ht]
\centering
\includegraphics[width=4.9cm]{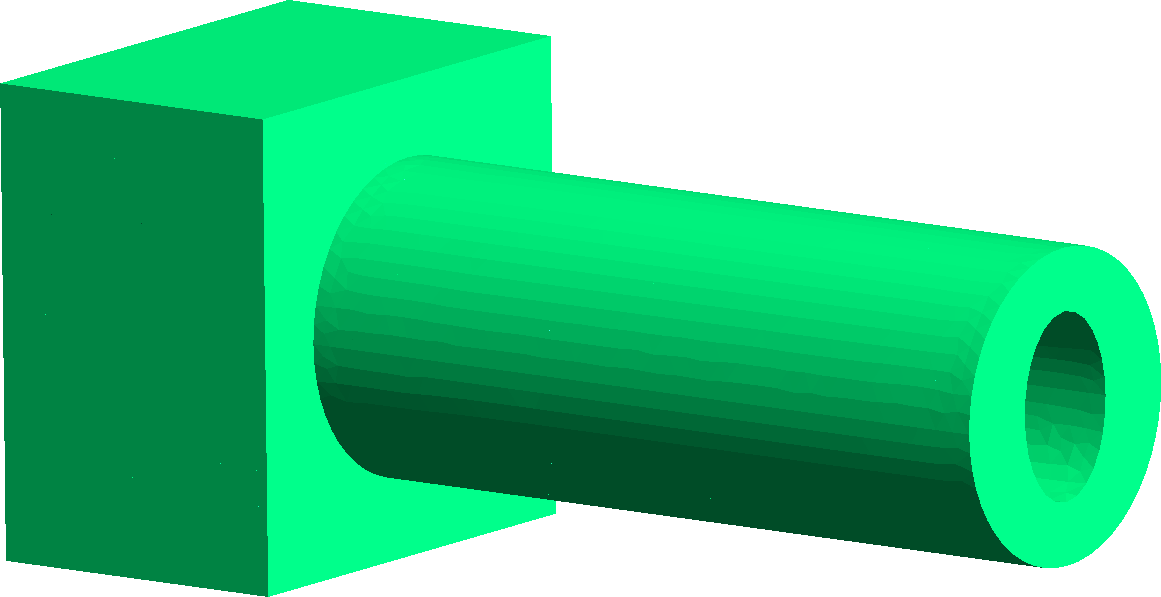}\hspace{-2cm}\raisebox{0.2cm}{\begin{tikzpicture}
\filldraw[fill=gray!20](0,0) circle (1);
\filldraw[fill=white] (0,0)ellipse(0.7 and 0.42);
\node at (0,0.7) {$S$};
\node at (-3,0) {$\Om$};
\end{tikzpicture}}\qquad
\raisebox{1cm}{\includegraphics[width=4.9cm]{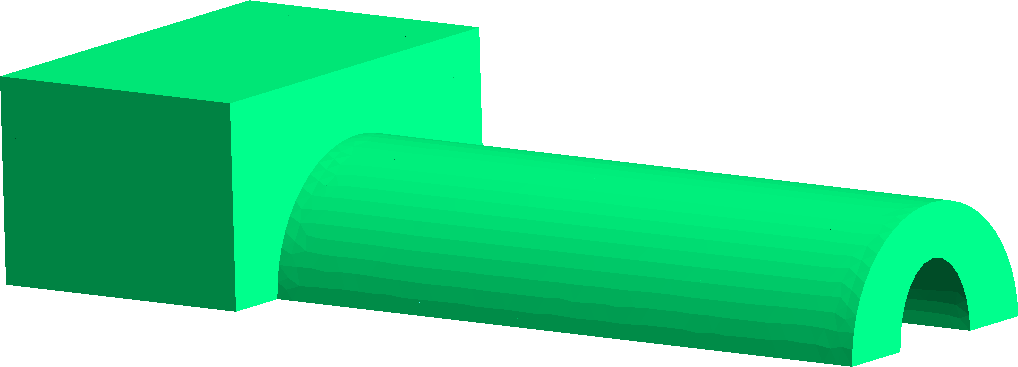}}\hspace{-2cm}\raisebox{1.1cm}{\begin{tikzpicture}
\begin{scope}
    \clip(-1.1,0) rectangle (1.1,1.1);
\filldraw[fill=gray!20](0,0) circle (1);
\filldraw[fill=white] (0,0)ellipse(0.7 and 0.42);
\node at (0,0.7) {$S_+$};
\end{scope}
\draw (-1,0)--(-0.7,0);
\draw (1,0)--(0.7,0);
\node at (-3,0.3) {$\Om_+$};
\end{tikzpicture}}\vspace{-0.2cm}
\caption{Examples of domains $\Om$, $S$, $\Om_+$, $S_+$.\label{EmbeddedGeom}}
\end{figure}

\noindent Let us come back to homogeneous waveguides, \textit{i.e.} with $\eps\equiv\mu\equiv1$. Except in Section \ref{SectionFullySeparable} where we worked in geometries where we have complete separation of variables, we only proved existence of eigenvalues in the discrete spectrum of the electric operator. Here we wish to explain how to find homogeneous waveguides where $A$ has eigenvalues which are embedded in the essential spectrum. To proceed, we adapt the classical symmetry trick presented for example in \cite{EvLV94}.\\
\newline
Let $\Om$ be as in (\ref{DefOm1})--(\ref{DefOm2}). Moreover we assume that $\Om$ is symmetric with respect to the plane $\Sigma\coloneqq\{\bfx\in\R^3\,|\,x=0\}$, \textit{i.e.} that
\[
(x,y,z)\in\Om\qquad\Rightarrow\qquad (-x,y,z)\in\Om.
\]
Set 
\[
\Om_\pm\coloneqq\{(x,y,z)\in\Om\,|\,\pm x>0\},\qquad\qquad S_+\coloneqq\{(x,y)\in S\,|\,x>0\}.
\]
Importantly, we assume that $S$ is not simply connected while $S_+$ is (see Figure \ref{EmbeddedGeom}). We consider the case of an homogeneous material, that is $\eps\equiv1$, $\mu\equiv1$, and denote by $A$ (resp. $A_+$) the operator defined in (\ref{DefOpA}) in the geometry $\Om$ (resp. $\Om_+$). According to Proposition \ref{DefSigmaEss}, we have  
\[
\sigma_{\mrm{ess}}(A)=[0;+\infty),\qquad\qquad \sigma_{\mrm{ess}}(A_+)=[\lambda_N(S_+);+\infty),
\]
where $\lambda_N(S_+)$ is the first positive eigenvalue of the Neumann Laplacian in $S_+$. Let us choose $\Om_+$ such that $\sigma_{\mrm{d}}(A_+)\ne\emptyset$. This can be done by working for example in the waveguides presented in Section \ref{Section1} or \S\ref{ParaBigRes}. For $\lambda_\bullet$ an eigenvalue in the discrete spectrum of $A_+$, introduce $\E_+=(E^+_x,E^+_y,E^+_z)^\top\in\boldsymbol{\mX}_N(\Om_+)$ a corresponding eigenfunction. Here $\boldsymbol{\mX}_N(\Om_+)$ is defined as $\boldsymbol{\mX}_N(\Om)$ in (\ref{DefSobolev}) with $\Om$ replaced by $\Om_+$. Let us exploit the symmetry of the problem with respect to $\Sigma$ to create an eigenpair for $A$.\\
\newline
Define $\E$ such that $\E=\E_\pm$ in $\Om_\pm$ with 
\[
\E_-(x,y,z)=\left(\begin{array}{c}
\phantom{-}E^+_x(-x,y,z)\\
-E^+_y(-x,y,z)\\
-E^+_z(-x,y,z)\\
\end{array}
\right)\mbox{ in }\Om_-.
\]
Using that $\E_+$ satisfies $\E_+\times\bfnu=0$ on $\Sigma$ with $\bfnu=(-1,0,0)^\top$, we deduce $E^+_y=E^+_z=0$ on $\Sigma$. Denoting classically by $[\cdot]|_{\Sigma}$ the jump of traces trough $\Sigma$, this yields
\[
[\E\times\bfnu]|_{\Sigma}=0,\qquad\qquad [\E\cdot\bfnu]|_{\Sigma}=0.
\]
Since we have $\E|_{\Om_\pm}\in\boldsymbol{\mX}_N(\Om_\pm)$, this ensures that $\E$ belongs to $\boldsymbol{\mX}_N(\Om)$ (in particular there holds $\div\,\E=0$ in $\Om$). On the other hand, we clearly have $\boldsymbol{\Delta} \E+\lambda_\bullet\E=0$ in $\Om$. Using again that $\curlvec\curlvec\cdot=-\boldsymbol{\Delta}\cdot+\nabla(\div\cdot)$, we obtain
\[
\begin{array}{|rcll}
\curlvec\curlvec\E&=&\lambda_\bullet\E&\mbox{ in }\Om\\[2pt]
\div\,\E&=&0&\mbox{ in }\Om\\[2pt]
\E\times\bfnu&=&0&\mbox{ on }\partial\Om.
\end{array}
\]
This proves that $\lambda_\bullet$ is an eigenvalue embedded in $\sigma_{\mrm{ess}}(A)=[0;+\infty)$.

\section{Concluding remarks}\label{SectionConclusion}~

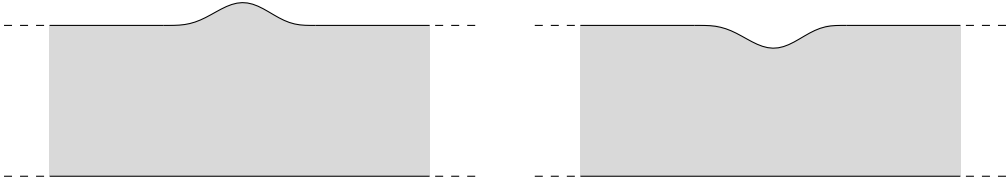
\begin{figure}[!ht]
\centering
\begin{tikzpicture}[scale=1]
\draw[fill=gray!30,draw=none](-2.5,0) rectangle (2.5,2);
\draw (-2.5,0)--(2.5,0);
\draw (-2.5,2)--(-1,2);
\draw (1,2)--(2.5,2);
\draw [dashed](-3.1,0)--(-2.5,0);
\draw [dashed](3.1,0)--(2.5,0);
\draw [dashed](-3.1,2)--(-2.5,2);
\draw [dashed](3.1,2)--(2.5,2);
\draw[gray!30] (-0.8,2)--(0.8,2);
\draw[samples=30,domain=-1:1,draw=black,fill=gray!30] plot(\x,{2+0.1*(\x+1)^4*(\x-1)^4*(\x+3)});
\end{tikzpicture}\qquad\begin{tikzpicture}[scale=1]
\draw[fill=gray!30,draw=none](-2.5,0) rectangle (2.5,2);
\draw (-2.5,0)--(2.5,0);
\draw (-2.5,2)--(-1,2);
\draw (1,2)--(2.5,2);
\draw [dashed](-3.1,0)--(-2.5,0);
\draw [dashed](3.1,0)--(2.5,0);
\draw [dashed](-3.1,2)--(-2.5,2);
\draw [dashed](3.1,2)--(2.5,2);
\draw[gray!30] (-0.8,2)--(0.8,2);
\draw[samples=30,domain=-1:1,draw=none,fill=white] plot(\x,{2.02-0.1*(\x+1)^4*(\x-1)^4*(\x+3)});
\draw[samples=30,domain=-1:1,draw=black,fill=white] plot(\x,{2-0.1*(\x+1)^4*(\x-1)^4*(\x+3)});
\end{tikzpicture}
\caption{2D waveguides with an exterior (left) or interior (right) bump. \label{PerturbedOm}}
\end{figure}

In this article, we showed results of existence of eigenvalues for Maxwell's equations in various electromagnetic waveguides which are unbounded in at least one direction. During this work, we have faced several questions that we have not been able to answer. We list a few of them here:\\[2pt]
- It is well-known that 2D waveguides with exterior bumps, as represented in Figure \ref{PerturbedOm} left, support trapped modes associated with discrete spectrum for the Dirichlet Laplacian. Can one hope for a similar result for Maxwell's operator? More precisely, can one obtain criteria on small smooth perturbations of the geometry of straight waveguides ensuring the existence of discrete spectrum? This is not obvious to prove for Maxwell's equations. One idea consists in working with the so-called Piola transform to convert the geometrically perturbed problem into a problem set in a straight waveguide with non-constant dielectric coefficients. Then one may imagine to apply the results of \S\ref{ParaAniso}. However \textit{a priori} one cannot proceed so simply because the signs of $\eps-\mrm{Id}$, $\mu-\mrm{Id}$, where  $\eps$, $\mu$ stand for the matrix valued coefficients obtained with the Piola transform, are not clear.\\[2pt]
- By adding small perfectly conducting obstacles in a straight waveguide and by using techniques of asymptotic analysis, can one ensure the existence of positive eigenvalues for the operator $A$?\\
- It is also known that the Dirichlet Laplacian in geometries with interior bumps as in Figure \ref{PerturbedOm} right has no discrete spectrum. Can one establish results of absence of discrete spectrum for Maxwell's operator $A$? Note for example that one can exploit Remark \ref{RmkFourierSeries} to guarantee that $\sigma_{\mrm{d}}(A)=\emptyset$ in the wedge geometry of Figure \ref{WedgeGeom} such that for $a>0$, $\alpha\in[0;\pi/2)$,  
\[
\Om\coloneqq (0;a)\times\Om_{\mrm{2D}}\quad\mbox{ with }\quad\Om_{\mrm{2D}}\coloneqq\{(y,z)\in\R^2\,|\,y\in(0;1),\,z\in(y\tan\alpha;+\infty)\}.
\]
- In the literature, some works have been devoted to the comparison of Maxwell's eigenvalues with the eigenvalues of other operators in bounded domains (see \cite{Paul15,Zhan18,Rohl24}). In Section \ref{ParaBigRes}, we exploited the comparison with the eigenvalues of the Dirichlet Laplacian in waveguides. Is there anything fruitful to be obtained from the (embedded) eigenvalues of the Neumann Laplacian?

\begin{figure}[!ht]
\centering
\includegraphics[width=8cm]{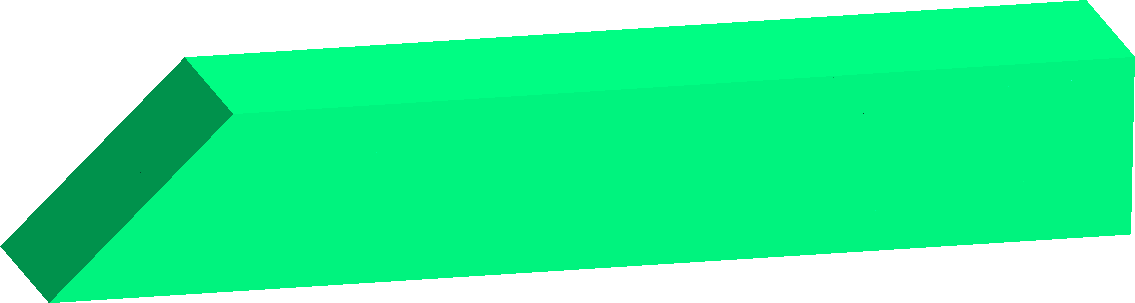}
\caption{Wedge geometry. \label{WedgeGeom}}
\end{figure} 

\section{Appendix}

\subsection{non-simply connected geometries with non-connected boundaries}\label{ParaNonTrivTop}

\begin{figure}[!ht]
\centering
\includegraphics[width=13cm]{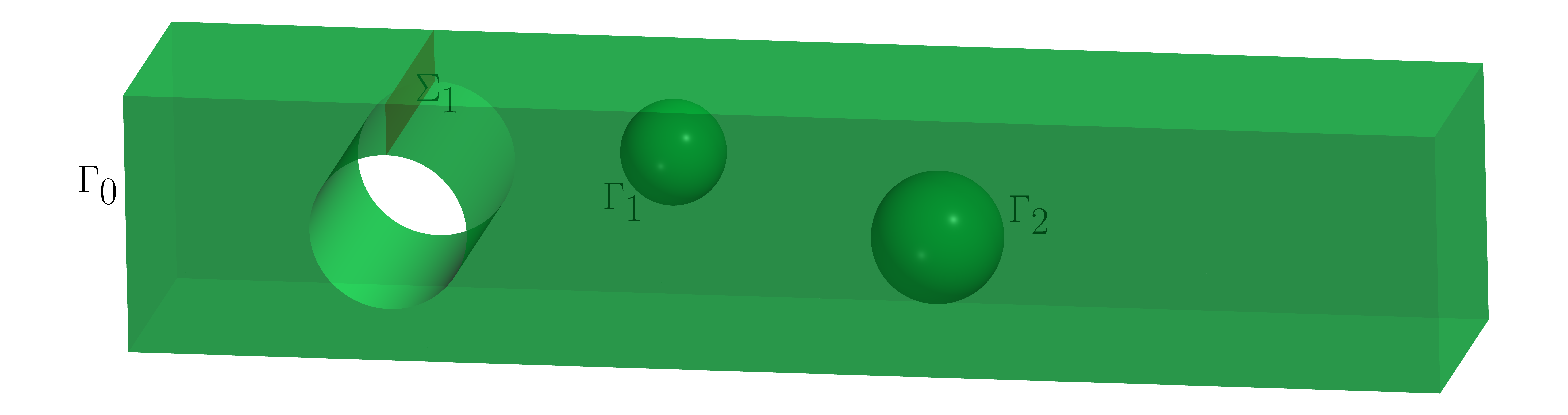}
\caption{Example of a non-simply connected waveguide $\Om$ with a non-connected boundary. Note that $\Om$ excludes the two balls and $\Sigma_1$ is an artificial cut ($\Sigma_1\subset\Om$).\label{NonConnectedBdry}}
\end{figure}

\noindent In the whole article, the domain $\Om$ was simply connected with a connected boundary. The various techniques we presented to show existence of eigenvalues for Maxwell's operator can also be used when this assumption is not satisfied. In this paragraph, we show that additionally, the electric and/or magnetic operators have non-zero kernels in these configurations, so that their discrete spectrum contains the value zero. Below we use notation similar to \cite{ABDG98}.
\paragraph{Waveguides with a non-connected boundary.}
\noindent Denote by $\Gamma_i$, $i=0,\dots, I$, the connected
components of the boundary $\partial\Om$ (see Figure \ref{NonConnectedBdry}). Assume that only $\Gamma_0$ is unbounded. Define the space 
\[
\mrm{H}^1_\Gamma(\Omega)\coloneqq\left\{\varphi\in \mrm{H}^1(\Omega)\,|\,\varphi_{|\Gamma_0}=0,\ \varphi_{|\Gamma_i}=cst,\ i=1,\dots, I\right\}. 
\]
For $i=1,\dots, I$, using a Poincar\'e inequality similar to (\ref{PoincareIneq}) and a lifting function, one can prove that there is a unique $p_i\in \mrm{H}^1_\Gamma(\Om)$ satisfying
\[
\begin{array}{|rcll}
\Delta p_i   & = & 0 & \mbox{ in }\Om \\[2pt]
p_i & = &\delta_{ik} & \mbox{ on }\Gamma_k,\ k=1,\dots, I.
\end{array}
\]
Here and below $\delta_{ik}$ stands for the Kronecker symbol such that $\delta_{ik}=1$ if $i=k$, $\delta_{ik}=0$ otherwise. Then one easily checks that $\nabla p_1,\dots,\nabla p_I$ is a family of linearly independent vector fields which belong to $\boldsymbol{\mX}_N(\Om)$. Moreover there holds 
\[
\curlvec\curlvec\nabla p_i=0,\qquad \mbox{ for }i=1,\dots, I.
\]
This proves that $\nabla p_1,\dots,\nabla p_I$ are elements of $\ker\,A$, \textit{i.e.} that $0\in\sigma_{\mrm{p}}(A)$. Note however that since the corresponding magnetic fields, $\curlvec\nabla p_i$, are null, they do not provide non-trivial elements of the kernel of the magnetic operator. 

\paragraph{non-simply connected waveguides.}
\noindent Assume that there exist bounded connected open
surfaces $\Sigma_j$, $j=1,\dots, J$, called ``cuts'' such that:\\[3pt]
~\hspace{0.5cm}\begin{tabular}{ll}
$i)$ & each surface $\Sigma_j$ is an open subset of a smooth variety;\\
$ii)$ & the boundary of $\Sigma_j$ is contained in $\partial\Om$, $j=1,\dots, J$;\\
$iii)$ & the intersection $\overline{\Sigma_j}\cap\,\overline{\Sigma_k}$ is empty for $j\ne k$;\\
$iv)$ & the open set $\dot{\Om}\coloneqq\Omega\setminus\bigcup_{i=1}^{J}\Sigma_j$ is pseudo-Lipschitz (see \cite[Def.\,3.1]{ABDG98}) and simply connected;
\end{tabular}\\[5pt]
(see the cut $\Sigma_1$ on Figure \ref{NonConnectedBdry}). The domain $\Omega$ is said topologically trivial when we can take  $J=0$. The extension operator from $\mrm{L}^2(\dot\Omega)$ to $\mrm{L}^2(\Omega)$ is denoted by $\ \widetilde{\cdot}\,$ whereas $[\cdot]_{\Sigma_j}$ denotes the jump through $\Sigma_j$, $j=1,\dots, J$. In this definition of the jump, we assume that a convention has been established for the sign. We also assume that a unit vector $\bfnu$ normal to $\Sigma_j$, $j=1,\dots, J$, is chosen, consistent with the choice of the sign of the jump. Define the space
\[
\Theta(\dot\Omega)\coloneqq\dsp\left\{\varphi\in\mrm{H}^1(\dot\Omega)\,|\,
\int_{S\times\{0\}}\varphi\,dxdy=0\mbox{ and }
[\varphi]_{\Sigma_j}=cst,\ j=1,\dots, J\right\}. 
\]
One can show by using in particular the zero-mean value constraint on $S\times\{0\}$, that for $j=1,\dots, J$, there exists a unique $q_j\in\Theta(\dot\Omega)$ solving
\[
\begin{array}{|rcll}
\Delta q_j  & = & 0 & \ \mbox{in }\dot\Om \\[2pt]
\partial_{\bfnu} q_j & = & 0 & \ \mbox{on }\partial\Om \\[2pt]
\left[q_j\right]_{\Sigma_k} & = &\delta_{jk}, &  \ k=1,\dots, J\\[2pt]
\left[\partial_{\bfnu} q_j\right]_{\Sigma_k} & = & 0, &  \ k=1,\dots, J.
\end{array}
\]
Then one observes that $\widetilde{\nabla q_1},\dots,\widetilde{\nabla q_J}$ is a family of linearly independent vector fields which belong to the space $\boldsymbol{\mX}_T(\Om)$ defined as $\boldsymbol{\mX}_T(\mu)$ in (\ref{DefSpaceH}) with $\mu\equiv1$. Moreover there holds 
\[
\curlvec\curlvec\nabla q_i=0,\qquad \mbox{ for }j=1,\dots, J.
\]
This is enough to conclude that $\widetilde{\nabla q_1},\dots,\widetilde{\nabla q_J}$ are elements of $\ker\,A^{\H}$, \textit{i.e.} that $0\in\sigma_{\mrm{p}}(A^{\H})$, where $A^{\H}$ is the magnetic operator introduced in (\ref{DefMagOp}) with $\eps=\mu\equiv1$. Note however that the corresponding electric fields, $\curlvec\nabla q_j$, are null, and so do not provide non-trivial elements of the kernel of the electric operator.

\subsection{Poincar\'e-Friedrichs inequalities}
For the convenience of the reader, we reproduce here Lemma 5.1 of \cite{BaMaNa17} (see also \cite[Lem.\,A.1]{ChNa23}).
\begin{lemma}
Assume that $a>0$. Then we have the Poincar\'e-Friedrichs inequality 
\begin{equation}\label{Fried_ineq}
 \kappa(a)\int_{0}^{1/2}\phi^2\,dt \le \int_{0}^{+\infty}(\partial_t\phi)^2\,dt+a^2\int_{1/2}^{+\infty}\phi^2\,dt,\qquad \forall \phi\in\mH^1(0;+\infty),
\end{equation}
where $\kappa(a)$ is the smallest positive root of the transcendental equation
\begin{equation}\label{TransEqnLemma}
\sqrt{\kappa}\tan\bigg(\cfrac{\sqrt{\kappa}}{2}\bigg)=a.
\end{equation}
\end{lemma}
\begin{proof}
For $a>0$, consider the spectral problem 
\begin{equation}\label{Eigen1D}
\begin{array}{|rcll}
-\partial^2_t \phi+a^2\mathbbm{1}_{(1/2;+\infty)} \phi &=& \lambda(a)\,\mathbbm{1}_{(0;1/2)}\phi&\mbox{ in }(0;+\infty)\\[4pt]
\partial_t\phi(0)&=&0
\end{array}
\end{equation}
where $\mathbbm{1}_{(1/2;+\infty)}$, $\mathbbm{1}_{(0;1/2)}$ stand for the indicator functions of the sets $(1/2;+\infty)$, $(0;1/2)$ respectively. Let us equip $\mH^1(0;+\infty)$ with the inner product
\[
(\phi,\phi')_a=\int_{0}^{+\infty}\partial_t\phi\,\partial_t\phi'\,dt+a^2\phi\,\phi'\,dt.
\]
With the Riesz representation theorem, define the linear and continuous operator $T:\mH^1(0;+\infty)\to \mH^1(0;+\infty)$ such that
\[
(T\phi,\phi')_a=\int_{0}^{1/2}\phi\,\phi'\,dt,\qquad\forall \phi,\phi'\in\mH^1(0;+\infty).
\]
With this definition, we find that $(\lambda(a),\phi)$ is an eigenpair of (\ref{Eigen1D}) if and only if we have 
\[
T\phi=(\lambda(a)+a^2)^{-1}\phi.
\]
Since $T$ is bounded and symmetric, it is self-adjoint. Additionally the Rellich theorem ensures that $T$ is compact. This guarantees that the spectrum of (\ref{Eigen1D}) coincides with a sequence of positive eigenvalues whose only accumulation point is $+\infty$. Let us denote by $\kappa(a)$ the smallest eigenvalue of (\ref{Eigen1D}). From classical results concerning compact self-adjoint operators (see \textit{e.g.} \cite[Thm. 2.7.2]{BiSo87}), we know that
\begin{equation}\label{Def_quotient}
(\kappa(a)+a^2)^{-1}=\underset{\phi\in\mH^1(0;+\infty),\,(\phi,\phi)_a=1}{\sup} (T\phi,\phi)_a.
\end{equation}
Rearranging the terms, we find that (\ref{Def_quotient}) provides the desired estimates (\ref{Fried_ineq}). Now we compute $\kappa(a)$. Solving the ordinary differential equation (\ref{Eigen1D}) with $\lambda(a)=\kappa(a)$, we obtain, up to a multiplicative constant,
\[
\phi(t)=\begin{array}{|ll}
\cos(\sqrt{\kappa(a)}t) & \mbox{ for }t\in(0;1/2)\\[3pt]
c\,e^{-at} & \mbox{ for }t\ge 1/2
\end{array}
\]
where $c$ is a constant to determine. Writing the transmission conditions at $t=1/2$, we find that a non-zero solution exists if and only if $\kappa(a)>0$ satisfies the relation (\ref{TransEqnLemma}).
\end{proof}

\section*{Acknowledgements}

This work started after fertile discussions with Philippe Briet, Maxence Cassier, Thomas Ourmi\`eres-Bonafos and Michele Zaccaron, we thank them warmly.

\bibliographystyle{plain}
\bibliography{biblio}

@article{ScRW89,
  title={Quantum bound states in a classically unbound system of crossed wires},
  journal={Phys. Rev. B},
  author={Schult, {R.L.} and Ravenhall, {D.G.} and Wyld, {H.W.}},
  volume={39},
  number={8},
  pages={5476},
  year={1989},  
}

@article{ABDG98,
 author 	 = "Amrouche, C. and Bernardi, C. and Dauge, M. and Girault, V.",
 title  	 = "Vector potentials in three-dimensional non-smooth domains",
 journal	 = "Math. Meth. Appl. Sci.",
 year   	 = "1998",
 volume 	 = "21",
 pages  	 = "823--864"}

@article{ExSe90,
  title={Trapping modes in a curved electromagnetic waveguide with perfectly conducting walls},
  author={Exner, P. and {\v{S}}eba, P.},
  journal={Phys. Lett. A},
  volume={144},
  number={6-7},
  pages={347--350},
  year={1990},  
}

@incollection{zbMATH07224760,
 author = {Costabel, M. and Dauge, M.},
 title = {Maxwell eigenmodes in product domains},
 booktitle = {Maxwell's equations. Analysis and numerics. Contributions from the workshop on analysis and numerics of acoustic and electromagnetic problems, RICAM, Linz, Austria, October 17--22, 2016},
 pages = {171--198},
 year = {2019},
 publisher = {Berlin: De Gruyter},
}

@article{ChPa19,
  title={From zero transmission to trapped modes in waveguides},
  author={Chesnel, L. and Pagneux, V.},
  journal={J. Phys. A Math. Theor. },
  volume={52},
  number={16},
  pages={165304},
  year={2019},  
}

@article{FeMa24,
  title={{Essential spectrum for dissipative Maxwell equations in domains with cylindrical ends}},
  author={Ferraresso, F. and Marletta, M.},
  journal={J. Math. Anal. Appl.},
  volume={536},
  number={1},
  year={2024},  
}

@article{Rohl25,
  title={{Inequalities between Neumann and Dirichlet Laplacian eigenvalues on planar domains}},
  author={Rohleder, J.},
  journal={Math. Ann.},
  pages={1--19},
  year={2025},  
}

@article{BCOZ25,
  title={Geometric spectral properties of electromagnetic waveguides},
  author={Briet, P. and Cassier, M. and Ourmi{\`e}res-Bonafos, T. and Zaccaron, M.},
  journal={arXiv:2508.13591},
  year={2025}
}

@article{CLMM93,
  title={Multiple bound states in sharply bent waveguides},
  author={Carini, {J.P.} and Londergan, {J.T.} and Mullen, K. and Murdock, {D.P.}},
  journal={Phys. Rev. B},
  volume={48},
  number={7},
  pages={4503},
  year={1993}  
}

@article{ShTu12,
 title		 = "Total resonant transmission and reflection by periodic structures",
 author		 = "Shipman, {S.P.} and Tu, H.",
 journal	 = "SIAM J. Appl. Math.",
 volume		 = "72",
 number		 = "1",
 pages		 = "216--239",
 year		 = "2012",
}

@article{ChNaFano,
  title={{Exact zero transmission during the Fano resonance phenomenon in non-symmetric waveguides}},
  author={Chesnel, L. and Nazarov, {S.A.}},
  journal={Z. Angew. Math. Phys.},
  volume={71},
  number={3},
  pages={82},
  year={2020}  
}

@article{Naza20Elast,
  title={Constructing a trapped mode at low frequencies in an elastic waveguide},
  author={Nazarov, {S.A.}},
  journal={Funct. Anal. Appl.},
  volume={54},
  pages={31--44},
  year={2020},  
}

@article{Naza08Elast,
  title={Trapped modes in a cylindrical elastic waveguide with a damping gasket},
  author={Nazarov, {S.A.}},
  journal={Comput. Math. Math. Phys.},
  volume={48},
  pages={816--833},
  year={2008}
}

@article{na457,
 AUTHOR		 = "Nazarov, {S.A.}",
 TITLE		 = "Variational and asymptotic methods for finding eigenvalues below the continuous spectrum threshold",
 JOURNAL	 = "Sibirsk. Mat. Zh.",
 VOLUME		 = "51",
 YEAR		 = "2010", 
 pages       = "1086--1101",
 note		 = "(English transl. Siberian Math. J. 51, 866--878, 2010.)" 
}

@article{ChNa18,
 title		 = "Non reflection and perfect reflection via {Fano} resonance in waveguides",
 author		 = "Chesnel, L. and Nazarov, {S.A.}",
 journal	 = "Comm. Math. Sci.",
 volume		 = "16",
 number		 = "7",
 pages		 = "1779--1800",
 year		 = "2018"
}

@article {KaNa02,
 AUTHOR		 = "Kamotski{\u\i}, {I.V.} and Nazarov, {S.A.}",
 TITLE		 = "An augmented scattering matrix and exponentially decreasing solutions of an elliptic problem in a cylindrical domain",
 JOURNAL	 = "Zap. Nauchn. Sem. S.-Peterburg. Otdel. Mat. Inst. Steklov. (POMI)",
 VOLUME		 = "264",
 YEAR		 = "2000",
 NUMBER		 = "Mat. Vopr. Teor. Rasprostr. Voln. 29",
 PAGES		 = "66--82",
 note		 = "English transl. J. Math. Sci. 111(4):3657--3666, 2002"
 }

@article{Naza13,
 title	 = "Enforced stability of a simple eigenvalue in the continuous spectrum of a waveguide",
 author	 = "Nazarov, {S.A.}",
 journal = "Funct. Anal. Appl.",
 volume  = "47",
 number  = "3",
 pages   = "195--209",
 year    = "2013",
}

@article{DaPa98,
 title		= "Trapped modes in acoustic waveguides",
 author		= "Davies, {E.B.} and Parnovski, L.",
 journal	= "Quart. J. Mech. Appl. Math.",
 volume		= "51",
 number		= "3",
 pages		= "477--492",
 year		= "1998"
}

@article {AsPV00,
 AUTHOR		 = "Aslanyan, A. and Parnovski, L. and Vassiliev, D.",
 TITLE		 = "Complex resonances in acoustic waveguides",
 JOURNAL	 = "Quart. J. Mech. Appl. Math.",
 VOLUME		 = "53",
 YEAR		 = "2000",
 NUMBER		 = "3",
 PAGES		 = "429--447",
}

@book {CoKr13,
 AUTHOR 	 = "Colton, D. and Kress, R.",
 TITLE		 = "Inverse acoustic and electromagnetic scattering theory",
 SERIES		 = "Applied Mathematical Sciences",
 VOLUME		 = "93",
 EDITION	 = "{Third}",
 PUBLISHER  = "Springer-Verlag",
 ADDRESS	 = "Berlin",
 YEAR 		 = "2013",
 PAGES		 = "xiv + 405"
}

@article{AYCM06,
  title={Properties of trapped electromagnetic modes in coupled waveguides},
  author={Annino, G. and Yashiro, H. and Cassettari, M. and Martinelli, M.},
  journal={Phys. Rev. B Condens. Matter},
  volume={73},
  number={12},
  pages={125308},
  year={2006},  
}

@article{GoJa92,
  title={Bound states in twisting tubes},
  author={Goldstone, J. and Jaffe, R.L.},
  journal={ Phys. Rev. B},
  volume={45},
  number={24},
  pages={14100},
  year={1992}
}

@article{Jone53,
  title={The eigenvalues of $\nabla^2u+\lambda u= 0$ when the boundary conditions are given on semi-infinite domains},
  author={Jones, {D.S.}},
  journal={Math. Proc. Camb.},
  volume={49},
  number={4},
  pages={668--684},
  year={1953},  
}

@incollection{Pagn13,
 title		 = "Trapped modes and edge resonances in acoustics and elasticity",
 author		 = "Pagneux, V.",
 booktitle	 = "Dynamic Localization Phenomena in Elasticity, Acoustics and Electromagnetism",
 pages		 = "181--223",
 year		 = "2013",
 publisher	 = "Springer Vienna"
}

@article{Evan92,
 title	 = "Trapped acoustic modes",
 author	 = "Evans, {D.V.}",
 journal = "IMA J. Appl. Math.",
 volume  = "49",
 number	 = "1",
 pages	 = "45--60",
 year	 = "1992",
}

@article{Urse51,
 Author  = "Ursell, F.",
 Title	 = "Trapping modes in the theory of surface waves",
 Journal = "{Proc. Camb. Philos. Soc.}",
 Volume	 = "47",
 Pages	 = "347--358",
 Year	 = "1951",
}

@article{LiMc07,
 title	 = "Embedded trapped modes in water waves and acoustics",
 author	 = "Linton, {C.M.} and McIver, P.",
 journal = "Wave motion",
 volume	 = "45",
 number	 = "1",
 pages	 = "16--29",
 year	 = "2007",
}

@article{DuEx95,
author = {Duclos, P. and Exner, P.},
title = {Curvature-induced bound states in quantum waveguides in two and three dimensions},
journal = {Rev. Math. Phys.},
volume = {07},
number = {01},
pages = {73-102},
year = {1995}
}

@article{ChPa18,
 title		 = "Simple examples of perfectly invisible and trapped modes in waveguides",
 author		 = "Chesnel, L. and Pagneux, V.",
 journal	 = "Quart. J. Mech. Appl. Math.",
 volume		 = "71",
 number		 = "3",
 pages		 = "297--315",
 year		 = "2018"
}

@article{Wits90,
  title={{Examples of embedded eigenvalues for the Dirichlet Laplacian in perturbed waveguides}},
  author={Witsch, {K.J.}},
  journal={Math. Methods Appl. Sci.},
  volume={12},
  number={1},
  pages={91--93},
  year={1990},  
}

@ARTICLE{BiSo87b, 
 AUTHOR  	 = "Birman, {M. Sh.} and Solomyak, {M. Z.}", 
 TITLE  	 = "{$L^2$-theory of the Maxwell operator in arbitrary domains}", 
 JOURNAL 	 = "Russ. Math. Surv.", 
 YEAR   	 = "1987", 
 VOLUME 	 = "42", 
 PAGES  	 = "75--96"}

@article{EvLV94,
 title		= "Existence theorems for trapped modes",
 author		= "Evans, {D.V.} and Levitin, M. and Vassiliev, D.",
 journal	= "J. Fluid Mech.",
 volume		= "261",
 pages		= "21--31",
 year		= "1994"
}

@article{Zhan18,
  title={Comparison results for eigenvalues of curl curl operator and {S}tokes operator},
  author={Zhang, Z.},
  journal={Z. Angew. Math. Phys.},
  volume={69},
  number={4},
  pages={104},
  year={2018},  
}

@article{Paul15,
  title={On constants in {Maxwell} inequalities for bounded and convex domains},
  author={Pauly, D.},
  journal={J. Math. Sci.},
  volume={210},
  pages={787--792},
  year={2015},  
}

@article{Rohl24,
  title={Curl curl versus {Dirichlet Laplacian} eigenvalues},
  author={Rohleder, J.},
  journal={Bull. London Math. Soc.},
  volume={57},
  number={9},
  pages={2738--2747},
  year={2025},
}

@article{Naza11,
  title={Discrete spectrum of cranked, branchy, and periodic waveguides},
  author={Nazarov, {S.A.}},
  journal={Algebra i analiz},
  volume={23},
  number={2},
  pages={206--247},
  year={2011},  
}

@article{ABGM91,
  title={Quantum bound states in open geometries},
  author={Avishai, Y. and Bessis, D. and Giraud, {B.G.} and Mantica, G. },
  journal={Phys. Rev. B},
  volume={44},
  issue={15},
  pages={8028--8034},
  year={1991},  
}

@article{Naza14b,
  title={Discrete spectrum of cross-shaped quantum waveguides},
  author={Nazarov, {S.A.}},
  journal={J. Math. Sci.},
  volume={196},
  number={3},
  pages={346--376},
  year={2014},  
}

@article {Hech12,
    AUTHOR = {Hecht, F.}, 
	TITLE = {{New development in FreeFem++}},
   JOURNAL = {J. Numer. Math.},  FJOURNAL = {Journal of Numerical Mathematics},
    VOLUME = {20}, YEAR = {2012},
    NUMBER = {3-4}, PAGES = {251--265},
      ISSN = {1570-2820}, MRCLASS = {65Y15}, MRNUMBER = {3043640},
	  note = {\url{http://www3.freefem.org/}}
}

@article{PlPo2014,
  title={The {Maxwell} system in waveguides with several cylindrical ends},
  Author = {Plamenevskii, {B.A.} and Poretskii, {A.S.}},
  journal={St. Petersburg Math. J.},
  volume={25},
  number={1},
  pages={63--104},
  year={2014}
}

@inproceedings{DaLR12,
  title={Quantum waveguides with corners},
  author={Dauge, M. and Lafranche, Y. and Raymond, N.},
  booktitle={ESAIM: Proceedings},
  volume={35},
  pages={14--45},
  year={2012},
  organization={EDP Sciences}
}

@article{ExSt89,
  title={On existence of a bound state in an {L-shaped} waveguide},
  author={Exner, P. and {\v{S}}eba, P. and {\v{S}}tovi{\v{c}}ek, P.},
  journal={Czech. J. Phys. B},
  volume={39},
  number={11},
  pages={1181--1191},
  year={1989},  
}

@article{GeRe09,
  title={Gmsh: A 3-D finite element mesh generator with built-in pre-and post-processing facilities},
  author={Geuzaine, C. and Remacle, J.-F.},
  journal={Int. J. Numer. Methods Eng.},
  volume={79},
  number={11},
  pages={1309--1331},
  year={2009},  
}

@book{Cess96,
 title 		 = "Mathematical methods in electromagnetism: linear theory and applications",
 author		 = "Cessenat, M.",
 series		 = "Series on advances in mathematics for applied sciences",  
 year		 = "1996",
 publisher   = "World Scientific"}

@article{KrLZ25,
  title={{A note on the failure of the Faber--Krahn inequality for the vector Laplacian}},
  author={Krej{\v{c}}i{\v{r}}{\'\i}k, D. and Lamberti, {P.D.} and Zaccaron, M.},
  journal={ESAIM - Control Optim. Calc. Var.},
  volume={31},
  pages={21},
  year={2025},  
}

@article{ChNa23,
  title={{Spectrum of the Dirichlet Laplacian in a thin cubic lattice}},
  author={Chesnel, L. and Nazarov, {S.A.}},
  journal	 = "Math. Model. Numer. Anal.",
  volume      = "57",  
 pages		 = "3251--3273",
 year		 = "2023"
}

@book {BiSo87,
 AUTHOR		 = "Birman, {M.Sh.} and Solomyak, {M.Z.}",
 TITLE		 = "Spectral theory of selfadjoint operators in {H}ilbert space",
 SERIES		 = "Mathematics and its Applications (Soviet Series)",
 PUBLISHER	 = "D. Reidel Publishing Co.",
 ADDRESS	 = "Dordrecht",
 YEAR		 = "1987",
 PAGES		 = "xv+301",
}

@article{Mois09,
 title		 = "{Suppression of Feshbach resonance widths in two-dimensional waveguides and quantum dots: a lower bound for the number of bound states in the continuum}",
 author		 = "Moiseyev, N.",
 journal	 = "Phys. Rev. Lett.",
 volume		 = "102",
 number		 = "16",
 pages		 = "167404",
 year		 = "2009",
}

@article{SaBR06,
 title		 = "Bound states in the continuum in open quantum billiards with a variable shape",
 author		 = "Sadreev, {A.F.} and Bulgakov, {E.N.} and Rotter, I.",
 journal	 = "Phys. Rev. B",
 volume		 = "73",
 number		 = "23",
 pages		 = "235342",
 year		 = "2006",
}

@article{HZSJS16,
 title		 = "Bound states in the continuum",
 author		 = "Hsu, {C.W.} and Zhen, B. and Stone, {A.D.} and Joannopoulos, {J.D.} and Solja{\v{c}}i{\'c}, M.",
 journal	 = "Nat. Rev. Mater.",
 volume		 = "1",
 pages		 = "16048",
 year		 = "2016",
}

@article{GPRO10,
 title		 = "Bound states in the continuum in graphene quantum dot structures",
 author		 = "Gonz{\'a}lez, {J.W.} and Pacheco, M. and Rosales, L. and Orellana, {P.A.}",
 journal	 = "Europhys. Lett.",
 volume		 = "91",
 number		 = "6",
 pages		 = "66001",
 year		 = "2010",
}

@article{gomis2017anisotropy,
  title={Anisotropy-induced photonic bound states in the continuum},
  author={Gomis-Bresco, J. and Artigas, D. and Torner, L.},
  journal={Nat. Photon.},
  volume={11},
  number={4},
  pages={232},
  year={2017},
  publisher={Nature Publishing Group}
}

@article{zhen2014topological,
  title={Topological nature of optical bound states in the continuum},
  author={Zhen, B. and Hsu, {C.W.} and Lu, L. and Stone, {A.D.} and Solja{\v{c}}i{\'c}, M.},
  journal={Phys. Rev. Lett.},
  volume={113},
  number={25},
  pages={257401},
  year={2014},
  publisher={APS}
}

@article{BaMaNa17,
  title={The discrete spectrum of cross-shaped waveguides},
  author={Bakharev, {F.L.} and Matveenko, {S.G.} and Nazarov, {S.A.}},
  journal={St. Petersburg Math. J.},
  volume={28},
  number={2},
  pages={171--180},
  year={2017}
}

@article{Filo19,
  title={{Maxwell} operator in a cylinder with coefficients that do not depend on the longitudinal variable},
  author={Filonov, N.},
  journal={St. Petersburg Math. J.},
  volume={30},
  number={3},
  pages={545--572},
  year={2019}
}

@article{Kim17,
  title={Analysis of the non-reflecting boundary condition for the time-harmonic electromagnetic wave propagation in waveguides},
  author={Kim, S.},
  journal={J. Math. Anal. Appl.},
  volume={453},
  number={1},
  pages={82--103},
  year={2017},
}

@article{PlPo18,
  title={The {Maxwell} system in waveguides with several cylindrical outlets to infinity and nonhomogeneous anisotropic filling},
  author={Plamenevskii, {B.A.} and Poretskii, {A.S.}},
  journal={St. Petersburg Math. J.},
  volume={29},
  number={2},
  pages={289--314},
  year={2018}
}

@article{PlPoSa18,
  title={On a Method of Approximate Computing of Scattering Matrices for Electromagnetic Waveguides},
  author={Plamenevskii, {B.A.} and Poretskii, {A.S.} and Sarafanov, O.},
  journal={Dokl. Phys.},
  volume={63},
  number={10},
  pages={414--417},
  year={2018},
}

@article{Filo05,
  title={{On an inequality between Dirichlet and Neumann eigenvalues for the Laplace operator}},
  author={Filonov, N.},
  journal={St. Petersburg Math. J.},
  volume={16},
  number={2},
  pages={413--416},
  year={2005}
}

@article{Frie91,
  title={Some inequalities between {Dirichlet} and {Neumann} eigenvalues},
  author={Friedlander, L.},
  JOURNAL = "Arch. Rational Mech. Anal.",
  volume={116},
  pages={153--160},
  year={1991},  
}

\end{document}